\documentclass[11pt]{amsart}

\usepackage{latexsym}
\usepackage{amssymb}
\usepackage{amsmath}
\usepackage{color}

\usepackage{amscd} 

\usepackage{comment}

\calclayout

\usepackage{bbm}
\usepackage{hyperref}

\usepackage{tikz-cd}
\usepackage{appendix}
\newtheorem{theorem}{Theorem}[section]
\newtheorem{lemma}[theorem]{Lemma}
\newtheorem{proposition}[theorem]{Proposition}
\newtheorem{corollary}[theorem]{Corollary}
\newtheorem{definition}[theorem]{Definition}

\newtheorem{example}[theorem]{Example}
\newtheorem{remark}[theorem]{Remark}

\newcommand\supp{\mathop{\rm supp}}

\newcommand\id{\mathop{\rm id}}
\newcommand\tr{\mathop{\rm tr}}

\newcommand{\cl}[1]{\mathcal{#1}}
\newcommand{\bb}[1]{\mathbb{#1}}

\begin{document}

\title{Cantor correlations I. Operator systems and Cantor games}

\author[G. Baziotis]{Georgios Baziotis}
\address{
School of Mathematical Sciences, University of Delaware, 501 Ewing Hall,
Newark, DE 19716, USA}
\email{baziotis@udel.edu}

\author[A. Chatzinikolaou]{Alexandros Chatzinikolaou}
\address{
Department of Mathematics, National and Kapodistrian University of Athens, Athens 157 84, Greece}
\email{achatzinik@math.uoa.gr}

\author[I. G. Todorov]{Ivan G. Todorov}
\address{
School of Mathematical Sciences, University of Delaware, 501 Ewing Hall,
Newark, DE 19716, USA}
\email{todorov@udel.edu}

\author[L. Turowska]{Lyudmila Turowska}
\address{Department of Mathematical Sciences, Chalmers University
of Technology and the University of Gothenburg, Gothenburg SE-412 96, Sweden}
\email{turowska@chalmers.se}

\date{3 November 2025}

\begin{abstract}
    We study no-signalling correlations over Cantor spaces, 
    placing the product of infinitely many copies of a finite non-local game
    in a unified general setup. 
    We define the subclasses of local, quantum spatial, approximately quantum and quantum commuting Cantor correlations and describe them in terms of states on tensor products of inductive limits of operator systems. We provide a correspondence between no-signalling (resp.\,approximately quantum, quantum commuting) Cantor correlations and sequences of correlations of the same type over the projections onto increasing number of finitely many coordinates. We introduce Cantor games, and 
    associate canonically such a game to a sequence of finite input/output games, showing that the numerical sequence of the values of the games in the sequence  converges to the corresponding value of the compound Cantor game.
\end{abstract}

\maketitle

\tableofcontents


\section{Introduction}
 
{\it Non-local games} have been at the centre of the fruitful interactions 
between operator algebras and quantum information theory witnessed in the 
past decade (see e.g. \cite{cjpp, lmprsstw, lmr, mror, pszz, pal-vid, psstw}). 
These are games, played cooperatively by players Alice and Bob against a Verifier; 
in a single round of the game, the Verifier draws a pair $(s,t)$ of inputs 
form the cartesian product $S\times T$ of two finite sets according 
to a certain probability distribution, and sends $s$ (resp.\,$t$)
to Alice (resp.\,Bob). The players respond with a pair $(u,v)$ of outputs
from the cartesian product $U\times V$ of two (perhaps different) finite sets; 
the tandem Alice-Bob wins the round if the quadruple $(s,t,u,v)$ satisfies 
a given predicate, known to the players, and interpreted as the 
rules of the game. 
The players are not allowed to communicate during the course of the game, 
but they may agree beforehand on using a specific strategy. 

Several types of strategies thus appear, 
depending on the physical model the players avail of (that is, 
the way of forming the joint physical system of their individual 
systems), leading to a type hierarchy 
of {\it no-signalling correlations} between them, 
expressed through a proper inclusion chain 
\begin{equation}\label{eq_in0}
\cl C_{\rm loc}\subseteq \cl C_{\rm qs}\subseteq \cl C_{\rm qa}
\subseteq \cl C_{\rm qc}\subseteq \cl C_{\rm ns},
\end{equation}
where each of $\cl C_{\rm t}$ is the set of 
correlations between Alice and Bob observed during a repetition of game rounds. 
In particular, the class $\cl C_{\rm qc}$, arising by utilising the commuting operator model, strictly contains the 
the class $\cl C_{\rm qa}$ obtained by using liminal finite dimensional 
entanglement \cite{jnvwy}.
The inequality $\cl C_{\rm qa}\neq \cl C_{\rm qc}$ 
answers in the negative the Tsirelson problem in theoretical physics \cite{tsirelson}
and, simultaneously, thanks to \cite{fri-CEP, jnppsw, oz}, the Connes Embedding Problem in operator algebra theory \cite{connes}. 
At the heart of the equivalence between the aforementioned problems 
lie characterisations of the strategies from the classes $\cl C_{\rm qa}$ and 
$\cl C_{\rm qc}$ via states on, respectively, 
the minimal and maximal tensor products of two 
universal C*-algebras, each associated to one of the players. 

The correspondence between strategies and states on the tensor 
product of these universal C*-algebras is not bijective; in fact, a given strategy 
may arise from multiple such states. This phenomenon lies at the 
core of quantum self-testing \cite{ctt-self, pszz}, and leads to the necessity 
to use more economical, bijective, correspondences between 
strategies and states. This was achieved in
\cite{lmprsstw, pt}, where characterisations of strategies were 
obtained via states on different types of 
tensor products of universal {\it operator systems}, 
as opposed to C*-algebras, associated with the players of the game.

Parallel repetition \cite{csuu,r}, that is, the formation of 
the product \cite{mptw} of a sequence of copies 
of the game, played successively and independently, 
is at the base of defining and studying
the asymptotic value of the game which, among others, is 
instrumental for demonstrating separation between classical and 
quantum models in device-independent cryptography. 
The product of countably many copies of a game 
can naturally be viewed as a single game, 
played over the Cantor spaces arising from the underlying 
finite sets of inputs/outputs. 
In this {\it Cantor setup}, the inputs/outputs 
are elements of those Cantor spaces; 
from an operational point of view, such {\it Cantor games}
can be thought of as non-local games, 
in which the players receive a string of inputs of arbitrary (but equal) length, and are required to respond with 
strings of outputs of the same length. 
In fact, the Cantor setup is substantially more general, as
the induced rules on the successively 
larger (but finite) input/output sets do not need to be 
products of a given common underlying rule, that is, 
the successive rounds captured within the compound 
Cantor game are not necessarily independent or identical. 

The aim of the present paper is the study of strategies
of Cantor games and the one-shot values 
(that is, the optimal 
winning probabilities in a single game round, 
according to the strategy type used) thereof. 
While a general definition and basic properties of 
no-signalling correlations over standard measure spaces
was given in \cite{btt}, the Cantor setup contains the 
crucial distinctive element of {\it inductivity}. 
More precisely, a no-signalling 
strategy $\Gamma$ of a given Cantor game 
$\cl G$ corresponds in a
unique way to a sequence $(\Gamma_n)_{n\in \bb{N}}$, where $\Gamma_n$ is a no-signalling strategy of the 
restriction of $\cl G$ to its first $n$ coordinates. 
We show that the passage from $\Gamma$ to $(\Gamma_n)_{n\in \bb{N}}$ preserves the quantum commuting and 
approximately quantum correlation types, but not necessarily the 
quantum spatial correlation type.
We define a universal operator system for each of the players 
in the Cantor setup, 
and show that the Cantor correlations of different types 
arise from states on different kinds of operator system tensor product \cite{kptt} 
of these universal operator systems. 
As a consequence, the class of quantum commuting Cantor 
correlations is closed in the (natural to employ in our setting \cite{gb, btt}) Arveson BW topology
\cite{arveson-acta}. 
Our main tools are drawn 
from operator algebra theory, including 
ultrapowers (see e.g. \cite{AndoHaagerup} and 
\cite[Section 11.5]{Pi2}), 
operator system theory, 
including their tensor products \cite{kptt}, 
co-products \cite{kavruk-nucl} 
and inductive limits \cite{mt}, and
operator-valued information theory over abstract 
alphabets (see \cite{gb, btt, ffp, kakihara}). 

We apply the Cantor correlation setup to examine 
the behaviour of the values
of Cantor games. 
We focus on the case where the question set is endowed with the uniform probability distribution, which, in the Cantor setting, corresponds to the product of uniform probability measures, and obtain a continuity result for 
the values of inductive sequences of games, 
inscribing the latter fact in a series of results about 
tensor norm expressions of game values (see \cite{cjpp, cltt, pal-vid}).

We have postponed some topics, naturally arising from our results, 
for future work, such as 
descriptions of synchronicity (which is examined in the upcoming article \cite{bctt}), 
and consideration of inputs that are not independent and identically distributed.

In the sequel, we describe the content of the paper in more detail. 
Section \ref{s_finite} is preliminary and 
contains the necessary 
background in operator system theory and the operator-algebraic 
approach to no-signalling correlations in the finite case. 
After reviewing operator-valued information channels 
in Section \ref{s_ovic}, given Cantor spaces arising from 
sequences $X = (X_n)_{n\in \bb{N}}$ and $A = (A_n)_{n\in \bb{N}}$
of increasing sets of underlying finite coordinates, 
we introduce Alice's universal operator system $\cl S_{X,A}$
as a direct limit of an appropriately defined 
inductive sequence, where 
the embeddings take into account the uniform distributions
over the question sets $X_n$, $n\in \bb{N}$. 
We show that the unital completely positive maps from 
$\cl S_{X,A}$ into the C*-algebra $\cl B(H)$ of all bounded operators
on a Hilbert space $H$ are in a canonical 
correspondence with the $\cl B(H)$-valued information channels
indexed by elements of the Cantor space associated with $X$. 

In Section \ref{s_rvs}, we 
define the main Cantor correlation types, namely those of 
local, quantum spatial, approximately quantum, quantum commuting and 
no-signalling ones, establish the inclusion chain (\ref{eq_in0}) in the Cantor setup,
and show that the Cantor no-signalling correlations 
correspond to states on the maximal tensor product $\cl S_{X,A}\otimes_{\max}\cl S_{Y,B}$ (here $\cl S_{X,A}$ and $\cl S_{Y,B}$ are the Alice and Bob universal 
operator systems, respectively). En route, we provide a 
general result about diagonals of successive inductive limits in the operator system category. We establish the no-signalling type preservation 
for the operation $\Gamma\to (\Gamma_n)_{n\in \bb{N}}$. 

The same property for the quantum commuting correlation type 
is provided in Section \ref{s_oct} and relies on an ultraproduct construction. 
It leads to a characterisation of quantum commuting Cantor correlations
via states on the commuting tensor product 
$\cl S_{X,A}\otimes_{\rm c}\cl S_{Y,B}$. 
An analogous characterisation is shown to hold for the 
approximately quantum correlations, this time via states on the minimal 
tensor product $\cl S_{X,A}\otimes_{\min}\cl S_{Y,B}$. 
While the classes of quantum commuting and, by definition, approximately quantum,
Cantor correlations are closed in Arveson's BW topology, we show that 
this is not the case for the quantum spatial correlation type. 

Section \ref{s_Cantorg} is devoted to Cantor games and their values. We show that 
a finite game has the same value (of any type ${\rm t}$ among 
the local, quantum spatial, quantum commuting or no-signalling ones) as the canonical Cantor game
arising from it after embedding the rules in the first coordinate
of the corresponding Cantor spaces. 
By projecting on the first $n$ coordinates, every 
Cantor game $\cl G$ gives rise to a decreasing sequence $(\cl G_n)_{n\in \bb{N}}$ of finite games. We show that the value $\omega_{\rm t} (\cl G)$ of $\cl G$ is the limit of the (decreasing) sequence 
$(\omega_{\rm t} (\cl G_n))_{n\in \bb{N}}$ of finite game values. 
As a consequence, $\omega_{\rm t} (\cl G)$ 
can be obtained as a limit of a 
(decreasing) sequence of norms of increasingly larger game tensors, canonically 
associated with $\cl G$, in the maximal, commuting and minimal tensor 
product of the corresponding universal operator systems. 
Finally, we discuss a class of examples to which our results apply.

\medskip

\noindent 
{\bf Acknowledgements. } 
The second named author was supported by the European Union -- Next Generation EU (Implementation Body: HFRI. Project
name: Noncommutative Analysis: Operator Systems and Nonlocality. HFRI
Project Number: 015825)
The third named author was supported by NSF grants 
2115071 and 2154459. The fourth named author was supported by the Swedish Research Council project grant 2023-04555 and GS Magnusons Fond MF2022-0006. 
The authors are grateful to BIRS for funding and hospitality during a 
\lq\lq Research in Teams'' stay in September 2025, 
during which this work was completed.


\section{Finite no-signalling correlations}
\label{s_finite}


In this section, we provide the necessary background 
on operator systems, completely positive maps, 
no-signalling correlations over finite sets,
and their operator systems characterisation,
for later reference.
We refer the reader to \cite{Pa} for further background and details.

Given a vector *-space $V$ over the complex field, let 
$M_n(V)$ be the vector *-space of all $n$ by $n$ matrices with entries 
in $V$, and $M_n(V)_h$ be the real vector space of all self-adjoint elements of $M_n(V)$. 
An \emph{operator system} is a tuple $(V,(C_n)_{n\in \bb{N}},e)$,
where $V$ is a vector *-space, $C_n$ is a proper cone 
in $M_n(V)_h$, the family $(C_n)_{n\in \bb{N}}$ is consistent 
in that $\alpha^* C_n \alpha\subseteq C_m$ for all scalar 
$n$ by $m$ matrices $\alpha$, and the element $e\in C_1$
is an Archimedean matrix order unit for $(C_n)_{n\in \bb{N}}$. We usually write 
$M_n(V)^+ = C_n$. 
Given operator systems $(V,(C_n)_{n\in \bb{N}},e)$ and 
$(V',(C_n')_{n\in \bb{N}},e')$ and a linear map $\phi$, 
we let $\phi^{(n)} : M_n(V)\to M_n(V')$ be the 
(linear) map, given by $\phi^{(n)}((x_{i,j})_{i,j}) = (\phi(x_{i,j}))_{i,j}$. The map $\phi$ is called \emph{positive} 
if $\phi(C_1)\subseteq C_1'$, and \emph{completely positive}
if $\phi^{(n)}$ is positive for every $n\in \bb{N}$; 
it is called a complete order embedding if it is injective and
$\phi^{(n)}(M_n(V))\cap M_n(V')^+ = \phi^{(n)}(M_n(V)^+)$, and 
a complete order isomorphism if it is a surjective complete order 
embedding. 

If $H$ is a Hilbert space, we denote by $\cl B(H)$ the 
space of all bounded linear operators acting on $H$ and by 
$I_H$ the identity operator on $H$. 
All Hilbert spaces we use will be assumed separable. 
After letting $\cl B(H)^+$ denote the cone of all 
positive operators in $\cl B(H)$ and making the identification 
$M_n(\cl B(H)) = \cl B(H^n)$, we have that 
every unital selfadjoint subspace $\cl S\subseteq \cl B(H)$ 
is an operator system with cones $M_n(\cl S)^+ := M_n(\cl S)\cap M_n(\cl B(H))^+$; we call operator systems 
of the latter type \emph{concrete}.
By virtue of the Choi-Effros Theorem
(see e.g. \cite[Theorem 13.1]{Pa}), every operator system 
is completely order isomorphic to a concrete operator system. 
We note that every operator system is an operator space in a canonical way, 
and write ${\rm CB}(\cl S,\cl T)$ (resp. ${\rm UCP}(\cl S,\cl T)$) for the operator space
(resp. convex set) of completely bounded (resp. unital completely positive) maps 
from $\cl S$ into $\cl T$.

We recall the definitions of the operator system tensor 
products that will be used subsequently, and refer to 
\cite{kptt} for further details. 
Given operator systems $\cl S\subseteq \cl B(H)$ and 
$\cl T\subseteq \cl B(K)$ (where $H$ and $K$ are 
Hilbert spaces), we write 
$\cl S\otimes \cl T$ for their algebraic tensor product; 
all three tensor products that we define have 
the latter as their underlying vector *-space. 
The \emph{minimal tensor product} 
$\cl S\otimes_{\min}\cl T$ is equipped with the matricial cones 
that make the inclusion $\cl S\otimes\cl T\subseteq \cl B(H\otimes K)$
a complete order embedding (we denote by $H\otimes K$ the Hilbertian tensor product).
The \emph{commuting tensor product} $\cl S\otimes_{\rm c}\cl T$
has matricial cones 
defined by letting 
$w\in M_n(\cl S \otimes_{\rm c}\cl T)^+$ if 
$(\phi\cdot\psi)^{(n)}(w)\in \cl B(L)^+$ whenever 
$\phi : \cl S\to \cl B(L)$ and $\psi : \cl T\to \cl B(L)$ are 
completely positive maps with commuting ranges, 
$L$ is a Hilbert space 
and $\phi\cdot\psi : \cl S\otimes\cl T\to \cl B(L)$ is the 
linear map, given by $(\phi\cdot\psi)(u\otimes v) = \phi(u)\psi(v)$
(we say that $\phi$ and $\psi$ form a \emph{commuting pair}).
Finally, the \emph{maximal tensor product} $\cl S\otimes_{\max}\cl T$
has matricial cone structure, generated by the 
elementary tensors of the form 
$S\otimes T$, where $S\in M_n(\cl S)^+$ and $T\in M_m(\cl T)^+$.

For $d\in \bb{N}$, let $[d] = \{1,\dots,d\}$. 
Given operator systems $\cl S_i$, $i\in [d]$, their 
\emph{coproduct} 
\cite{fri, kavruk-nucl} is a pair of the form
    $\left(\cl S_1\oplus_1\cdots \oplus_1 \cl S_d, (\iota_i)_{i=1}^d\right)$, 
where 
$\cl S_1\oplus_1\cdots \oplus_1 \cl S_d$ is an operator system, and 
$\iota_i : \cl S_i \to \cl S_1\oplus_1\cdots \oplus_1 \cl S_d$ is a
unital complete order embedding, such that 
if $\cl R$ is an operator system and $\phi_i : \cl S_i\to \cl R$ is a 
unital completely positive map, $i \in [d]$, then there exists
a unique unital completely positive map 
$\phi : \cl S_1\oplus_1\cdots \oplus_1 \cl S_d\to \cl R$, 
such that $\phi\circ\iota_i = \phi_i$, $i\in [d]$. 
We will write $\dot{\oplus}_{i=1}^d \cl S_i = \cl S_1\oplus_1\cdots \oplus_1 \cl S_d$. 
Given, in addition, operator systems $\cl T_i$, $i\in [d]$, and 
unital completely positive maps $\psi_i : \cl S_i\to \cl T_i$, $i\in [d]$, 
there exists a unique unital completely positive map 
$\dot{\oplus}_{i=1}^d \psi_i : \dot{\oplus}_{i=1}^d \cl S_i\to \dot{\oplus}_{i=1}^d \cl T_i$, 
such that $\left(\dot{\oplus}_{i=1}^d \psi_i\right)(\iota_i(u)) = (\iota_i\circ\psi_i)(u)$, 
$u\in \cl S_i$, $i\in [d]$. 
Indeed, the map $\iota_i\circ\psi_i : \cl S_i\to \dot{\oplus}_{j=1}^d \cl T_j$ is unital and 
completely positive, $i\in [d]$, and the existence of $\dot{\oplus}_{i=1}^d \psi_i$ follows 
from the universal property of the coproduct of the family $\{\cl S_i\}_{i=1}^d$. 

We will require some preliminaries on 
inductive limits of operator systems, which we now include; 
we refer the reader to \cite{mt} for further details. 
Let 
\begin{equation}\label{eq_indlimS}
\cl S_1\stackrel{\phi_1}{\longrightarrow} 
\cl S_2\stackrel{\phi_2}{\longrightarrow} 
\cl S_3\stackrel{\phi_3}{\longrightarrow} \cdots
\end{equation}
be an inductive system in the operator system 
category; this means that $\cl S_k$ is an operator 
system and $\phi_k$ is a unital completely positive map for every $k\in \bb{N}$. 
The inductive limit of (\ref{eq_indlimS}) 
is a pair $(\cl S, (\phi_{k,\infty})_{k\in\bb{N}})$, 
where $\cl S$ is an operator system and 
$\phi_{k,\infty} : \cl S_k\to \cl S$ is a unital completely positive map, $k\in \bb{N}$,
with the property that 
if $\cl R$ is an operator system and 
$\rho_k : \cl S_k\to \cl R$, $k\in \bb{N}$, are unital 
completely positive maps, such that 
$\rho_{k+1}\circ \phi_k = \rho_k$, $k\in \bb{N}$,
then there exists a unique unital completely positive map 
$\rho : \cl S\to \cl R$ such that 
$\rho\circ\phi_{k,\infty} = \rho_k$, $k\in \bb{N}$. 
Such an operator system $\cl S$ is unique up to 
a complete order isomorphism; we write 
$\cl S = \varinjlim \cl S_k$. 

Given a finite set $A$, we let $M_A$ be the algebra of all $|A|\times |A|$ matrices 
with complex entries, and $\cl D_A$ be its subalgebra of all diagonal matrices. 
Given another finite set $X$, we let
$\cl A_{X,A} = \underbrace{\cl D_A\ast\cdots\ast\cl D_A}_{|X| 
\hspace{0.05cm}\mbox{\tiny times}}$
be the free product, amalgamated over the units, and 
\begin{equation}\label{eq_SXAeq}
\cl S_{X,A} = \underbrace{\cl D_A\oplus_1\cdots\oplus_1\cl D_A}_{|X| 
\hspace{0.05cm}\mbox{\tiny times}}.
\end{equation}
We note the unital completely order isomorphic inclusion 
$\cl S_{X,A}\subseteq \cl A_{X,A}$ \cite{lmprsstw, pt}. 
We write $e_{x,a}$, $x\in X$, $a\in A$, for the canonical generators of $\cl S_{X,A}$, that is, 
$e_{x,a} = \iota_x(\delta_a)$, where $\iota_x : \cl D_A\to \cl S_{X,A}$ is the inclusion map of the 
$x$-th copy of $\cl D_A$, and $(\delta_a)_{a\in A}$ is the canonical basis of $\cl D_A$.

Given finite sets $X$, $Y$, $A$ and $B$, 
a \emph{no-signalling correlation} over the quadruple 
$(X,Y,A,B)$ is a family 
$\{(p(a,b|x,y))_{a,b} : x\in X, y\in Y\}$,
where $p(\cdot,\cdot|x,y)$ is a probability 
distribution over $A\times B$ for every $(x,y)\in X\times Y$, 
$$\sum_{b\in B} p(a,b|x,y) = \sum_{b\in B} p(a,b|x,y'), \ \ x\in X, y,y'\in Y, a\in A,$$
and 
$$\sum_{a\in A} p(a,b|x,y) = \sum_{a\in A} p(a,b|x',y), \ \ x,x'\in X,  y\in Y, b\in B.$$
Recall that a \emph{positive operator-valued measure (POVM)} is a (finite) family $(E_i)_{i=1}^d$ of positive operators acting on a 
Hilbert space, such that $\sum_{i=1}^d E_i = I$.
A no-signalling correlation $p$ over $(X,Y,A,B)$ is called 
\emph{quantum commuting} if it has the form
\begin{equation}\label{eq_defqcty}
p\left(a,b|x,y\right) = \langle P_{x,a} Q_{y,b}\xi ,\xi\rangle,
\end{equation}
where
$\xi$ is a unit vector in a Hilbert space $H$, and 
$\left(P_{x,a}\right) _{a\in A}$ and $\left(Q_{y,b}\right) _{b\in B}$, $x\in X$, $y\in Y$, are
POVM's on $H$ such that $P_{x,a} Q_{y,b} = Q_{y,b} P_{x,a}$ for all $x\in X$, $y\in Y$, $a\in A$ and $b\in B$.
The correlation $p$ is called \emph{quantum spatial} 
if the Hilbert space in the representation (\ref{eq_defqcty}) can be chosen 
of the form $H = H_A\otimes H_B$ for some Hilbert spaces $H_A$ and $H_B$, 
and $P_{x,a}$ (resp. $Q_{y,b}$) 
has the form $P_{x,a} = P_{x,a}'\otimes I_{H_B}$ 
(resp. $Q_{y,b} = I_{H_A}\otimes Q_{y,b}'$).
We further say that $p$ is an \emph{approximately quantum} correlation if 
$p$ is the limit (in the vector space $\bb{R}^X\times \bb{R}^Y\times \bb{R}^A\times
\bb{R}^B$) of quantum spatial correlations. 
We denote by $\cl C_{\rm ns}(X,Y,A,B)$ 
(resp. $\cl C_{\rm qc}(X,Y,A,B)$, $\cl C_{\rm qa}(X,Y,A,B)$, 
$\cl C_{\rm qs}(X,Y,A,B)$) the set 
of all no-signalling (resp. quantum commuting, approximately quantum, 
quantum spatial) correlations over the quadruple $(X,Y,A,B)$.
We refer the reader to \cite{lmprsstw} for 
further details regarding no-signalling 
correlations, and record here 
characterisations of
correlation types in terms of operator 
system tensor products 
that will be needed in the sequel
(\cite[Corollary 3.2]{lmprsstw} and \cite[Corollary 3.3]{lmprsstw}).

\begin{theorem}\label{th_orcha}
Let $X$, $Y$, $A$ and $B$ be finite sets.
The no-signalling (resp. quantum commuting, approximately quantum) correlations 
$p = \{(p(a,b|x,y))_{a,b}: x\in X, y\in Y\}$ 
are in bijective correspondence to states
$s : \cl S_{X,A}\otimes_{\max}\cl S_{Y,B}\to \bb{C}$
(resp. $s : \cl S_{X,A}\otimes_{\rm c}\cl S_{Y,B}\to \bb{C}$, 
$s : \cl S_{X,A}\otimes_{\min}\cl S_{Y,B}\to \bb{C}$) via the assignment 
$$p(a,b|x,y) = s(e_{x,a}\otimes e_{y,b}), \ \ \ x\in X, y\in Y, a\in A, b\in B.$$
\end{theorem}



\section{The Cantor operator system} \label{s_ovic}

In this section, we review the definitions of operator-valued information channels \cite{btt}, 
specialising to the context of Cantor spaces, introduce 
a class of operator systems that will be used in later sections, and 
describe their universal property.


\subsection{Operator-valued channels} \label{ss_genb}

Let $\frak{S}$ be a second countable compact Hausdorff space. 
We let $\frak{B}_{\frak{S}}$ be the Borel $\sigma$-algebra of $\frak{S}$, 
$C(\frak S)$ be the C*-algebra of all 
continuous complex-valued functions on $\frak S$, and $M(\frak S)$ be the 
space of all Radon measures on $\frak S$. 
 For a (separable) Hilbert space $H$, 
a \emph{quantum probability measure (QPM)} over $\frak{S}$ with values in $\cl{B}(H)$ is a 
map $E : \frak{B}_{\frak{S}}\to \cl B(H)^+$ such that $E(\emptyset) = 0$, $E(\frak{S}) = I$, and 
$E\left(\cup_{i=1}^{\infty} \alpha_i\right) = \sum_{i=1}^{\infty} E(\alpha_i)$
in the strong operator topology whenever 
$(\alpha_i)_{i\in \bb{N}}$ is a sequence of mutually disjoint elements of $\frak{B}_{\frak{S}}$.

Let $\frak X$ be a(nother) second countable compact Hausdorff space, 
equipped with a probability measure $\mu\in M(\frak X)$.
An \emph{operator-valued information channel} from $\frak X$ to $\frak S$ with values in $\cl{B}(H)$ 
is a family $E=E(\cdot|x)_{x\in \frak X}$ of QPM's over $\frak{S}$ such that, for every 
$\alpha\in\frak{B}_{\frak S}$, the function
$x\mapsto E(\alpha|x)$
is weakly $\mu$-measurable, that is, the functions 
$E_{\xi,\eta}(\alpha|\cdot) := \langle E(\alpha|\cdot)\xi,\eta\rangle$
are measurable for all $\xi,\eta\in H$ \cite{btt}. 
We write $\frak{C}( \frak{S},\frak X;H)$ for the set of all operator-valued information channels from $\frak X$ to $\frak S$ with values in $\cl{B}(H)$, 
and view its elements as measurable versions 
of families of POVM's in that the latter are 
operator-valued information channels over a pair of 
finite sets. 
We say that the channels 
$E,E'\in\frak{C}(\frak {S},\frak X;H)$ are $\mu$-equivalent 
(and write $E\sim_\mu E'$) if 
$$E(\alpha|x)=E'(\alpha|x)\ \  \text{$\mu$-almost everywhere},$$ 
for every $\alpha\in\frak{B}_{\frak S}$. We write
$\frak{C}_\mu(\frak S,\frak X;H)$ for the set of all $\sim_\mu$-equivalent classes of $\frak{C}(\frak S,\frak X;H)$. Elements of 
$\frak{C}_\mu(\frak {S},\frak X;H)$ will be called \emph{operator-valued $\mu$-information channels} 
from $\frak X$ to $\frak S$ with values in $\cl{B}(H)$ (see \cite{gb} 
and note that a slightly different terminology was used therein);
without risk of confusion, we will use the same symbol for 
an equivalence class and for a representative thereof.
Let $L_\sigma^\infty(\frak X,\mu,\cl{B}(H))$ be the von Neumman algebra of all equivalent classes of 
weak* measurable essentially bounded functions $F:\frak X\rightarrow\cl{B}(H)$, and note the 
canonical identification 
$L^\infty(\frak X,\mu)\bar{\otimes}\cl{B}(H)=L_\sigma^\infty(\frak X,\mu, \cl B(H))$. 
Here, and below, we use $\bar\otimes$ to denote the von Neumann algebra tensor product. 
For future reference, we recall the correspondence 
between $\mu$-information channels with values in $\cl B(H)$ and 
unital completely positive maps 
from $C(\frak S)$ into $L_\sigma^\infty(\frak X,\mu,\cl{B}(H))$ established in 
\cite[Theorem 3.11]{gb}.

\begin{theorem}\label{mod_mu}
If $E\in\frak{C}_\mu(\frak S,\frak X;H)$ then 
there exists a unital completely positive map 
$\Phi_{E} : C(\frak S)\rightarrow L_\sigma^\infty(\frak X,\mu,\cl{B}(H))$ such that     \begin{align}\label{eq_ovic}
    \langle\Phi_{E}(f)(x)\xi,\eta\rangle=\int_{\frak S}f(a)dE_{\xi,\eta}(a|x)\hspace{0.4cm} \mu\text{-a.e.}, \ \  f\in C(\frak S),
    \xi,\eta\in H.
    \end{align}
Conversely, if 
    $\Phi:C(\frak S)\rightarrow L_\sigma^\infty(\frak X,\mu,\cl{B}(H))$ 
    is a unital completely positive map then there exists 
    a (unique up to $\sim_{\mu}$-equivalence) channel
    $E\in\frak{C}_\mu(\frak S,\frak X;H)$ such that $\Phi = \Phi_{E}$.
\end{theorem}


Recall that, if $\cl X$ and $\cl Y$ are Banach spaces, 
the \emph{BW topology} \cite{arveson-acta} 
on the bounded subsets of the space
$\cl B(\cl X,\cl Y^*)$ of all bounded liner maps 
from $\cl X$ into $\cl Y^*$ is defined as the restriction of the 
point-weak* topology. The set $\frak{C}_\mu(\frak S,\frak X;H)$ will be 
hereafter equipped with the topology (which we continue to refer to 
as the BW topology) according to which
a net $(E^\lambda)_{\lambda\in\Lambda}\subseteq\frak{C}_\mu(\frak S,\frak X;H)$  converges to an element $E\in\frak{C}_\mu(\frak S,\frak X;H)$ if 
$\Phi_{E^\lambda}$ converges to $\Phi_{E}$ in the BW topology (see \cite{gb}).
We note that, by \cite[Theorem 3.14]{gb}, the space 
$(\frak{C}_\mu(\frak S,\frak X;H),\rm BW)$ is compact. 
Since the operator projective tensor product 
$C(\frak S)\hat{\otimes} L^1(\frak X,\mu)\hat\otimes\cl{T}(H)$ is separable,
the space $(\frak{C}_\mu(\frak S,\frak X;H),\rm BW)$ is metrisable
(see e.g. \cite[Theorem V.5.1]{dun_sch}).


\subsection{Inductive channel families}\label{ss_ind_fam}

If $X_1$ and $X_2$ are finite sets, 
we write $X_1|X_2$ if there exists $d\in \bb{N}$ such that 
$X_2 = X_1 \times [d]$. 
Assuming that $X_2 = X_1\times [d]$, 
let $\iota_{X_1,X_2} : \cl D_{X_1}\to \cl D_{X_2}$ be the 
unital *-monomorphism, given by $\iota_{X_1,X_2}(T) = T\otimes I_d$, after the 
canonical identification $\cl D_{X_2} = \cl D_{X_1}\otimes \cl D_{[d]}$. 

A family $X = (X_n)_{n\in \bb{N}}$ of finite sets will be called 
\emph{inductive} if $X_n |X_{n+1}$ for every $n\in \bb{N}$. 
The inductive limit of the sequence 
$$\cl D_{X_1}\stackrel{\iota_{X_1,X_2}}{\longrightarrow} 
\cl D_{X_2}\stackrel{\iota_{X_2,X_3}}{\longrightarrow} \cdots
\stackrel{\iota_{X_{n-1},X_{n}}}{\longrightarrow}
\cl D_{X_n}\stackrel{\iota_{X_n,X_{n+1}}}{\longrightarrow}\cdots
$$
in the category of C*-algebras will be denoted by $\cl D_X$. 
Assuming that $X_{n+1} = X_n \times [d^X_n]$, where $d^X_n\in \bb{N}$, $n\in \bb{N}$, 
we note that $\cl D_X$ is *-isomorphic to the C*-algebra $C(\Omega_X)$
of all continuous functions on the Cantor space 
$\Omega_X = \prod_{n=0}^{\infty}\left[ d^X_n\right]$
(where we have set $d_0^X = |X_1|$); 
equivalently, $\cl D_X = \otimes_{n=0}^{\infty} \cl D_{[d^X_n]}$
as an infinite C*-algebraic tensor product.

All tracial algebras will be equipped with normalised traces, and 
dualities will always be with respect to the latter.
For a finite set $X_0$, the
(normalised) trace on $\cl D_{X_0}$ will be denoted by ${\rm tr}_{X_0}$.
We note that, if $X_1|X_2$ then the embedding $\iota_{X_1,X_2}$ is trace-preserving.
For an inductive family $X = (X_n)_{n\in \bb{N}}$, 
we set $\tau_X = \otimes_{n=0}^{\infty} {\rm tr}_{[d^X_n]}$. 
We note that 
$\tau_X|_{\cl D_{X_n}} = {\rm tr}_{X_n}$, $n\in \bb{N}$. 
In the sequel, we write $L^1(\Omega_X)$ for $L^1(\Omega_X,\mu_X)$ and $\mu_{X_n}$ for the uniform probability measure on $X_n$, $n\in \bb{N}$. 
We note that
$L^1(\Omega_X)$ coincides with the 
$L^1$-space $L^1(\cl D_X)$ of the C*-algebra 
$\cl D_X$ with respect to (the trace) $\tau_X$.

 The unital (completely) positive trace-preserving maps 
 $\iota_{X_n,X_{n+1}} : \cl D_{X_n} \to \cl D_{X_{n+1}}$ give rise to the canonical conditional expectations 
 $\cl{E}_{X_{n+1},X_n} : \cl{D}_{X_{n+1}} \to \cl{D}_{X_n}$; we note that $\cl{E}_{X_{n+1},X_n} = \id_{X_n} \otimes \tr_{d_n}$. 
As $\cl D_X \subseteq L^\infty(\Omega_X) \subseteq L^1(\Omega_X)$, 
the canonical map 
$\iota_X^{(n)} : \cl D_{X_n} \to \cl D_X$ induces a weak* continuous, unital *-monomorphism
$\iota_{X,\infty}^{(n)} : \cl D_{X_n} \to L^\infty(\Omega_X)$, as well as a 
unital isometry 
$\iota_{X,1}^{(n)} : \cl D_{X_n} \to L^1(\Omega_X)$ with respect to the trace norm. 
The map $\iota_{X,\infty}^{(n)}$ admits a predual map $\cl E_{X_n} : L^1(\Omega_X) \to \cl D_{X_n}$, which is faithful, trace-preserving, and satisfies the identity 
$\cl E_{X_n} \circ \iota_{X,1}^{(n)} = \id_{\cl D_{X_n}}$. 
Thus, when identifying $\cl D_{X_n}$ as a subspace of $L^1(\Omega_X)$ via $\iota_{X,1}^{(n)}$, the map $\cl E_{X_n}$ is the canonical conditional expectation.
At the von Neumann algebra level, there exists a faithful, trace-preserving, weak* continuous conditional expectation $\cl E_{X_n}^\infty : L^\infty(\Omega_X) \to \cl D_{X_n}$, which is the dual of the isometry $\iota_{X,1}^{(n)}$. 
If $H$ is a Hilbert space, 
we further write $\tilde{\cl E}_{X_n}^\infty := \cl E_{X_n}^\infty \otimes 
\id_{\cl B(H)}$ for the conditional expectation
\[
\tilde{\cl E}_{X_n}^\infty : L^\infty(\Omega_X) \bar{\otimes} \cl B(H) \to \cl D_{X_n} 
\otimes \cl B(H),
\]
$\tilde{\cl E}_{X_{n+1},X_n}^\infty := \cl E_{X_{n+1},X_n}^\infty \otimes 
\id_{\cl B(H)}$ for the conditional expectation
\[
\tilde{\cl E}_{X_{n+1},X_n}^\infty : \cl D_{X_{n+1}} \otimes \cl B(H) \to \cl D_{X_n} \otimes \cl B(H), \ \ \ 
n\in \bb{N},
\]
and 
$\tilde{\iota}_{X,\infty}^{(n)} : \cl D_{X_n} \otimes \cl B(H) \to L^\infty(\Omega_X)\bar{\otimes} \cl B(H)$
and 
$\tilde{\iota}_{X_n,X_{n+1}} : \cl D_{X_n} \otimes \cl B(H)\to \cl D_{X_{n+1}} \otimes \cl B(H)$
for the maps, given by $\tilde{\iota}_{X,\infty}^{(n)} = \iota_{X,\infty}^{(n)}\otimes {\rm id}_{\cl B(H)}$, 
and 
$\tilde{\iota}_{X_n,X_{n+1}} = \iota_{X_n,X_{n+1}}\otimes {\rm id}_{\cl B(H)}$, $n\in \bb{N}$. 
We note that the 
superscript/subscript $\infty$ is used to indicate that the 
domain or the range of the corresponding map
is an $L^{\infty}$-space.

We say that a family $ (\Phi_n)_{n\in \bb N}$, where
$\Phi_n : \cl D_{A_n} \to \cl D_{X_n}\otimes \cl B(H)$ is a 
unital completely positive map, $n\in \bb{N}$,
is \textit{inductive} if 
\begin{equation}\label{eq_inre}
\Phi_n = \tilde{\cl{E}}_{X_{n+1},X_n}\circ \Phi_{n+1}\circ 
\iota_{A_n,A_{n+1}}, \ \ \ n\in \bb N,
\end{equation} 
that is, if the diagram
\begin{center}
          \begin{tikzcd}[row sep=large, column sep=large] 
             \cl D_{A_n}  \arrow[d, "\Phi_n"]    \ar[r,  "\iota_{A_{n},A_{n+1}}"]  &  \cl D_{A_{n+1}}  \arrow[d, "\Phi_{n+1}"]\\
             \cl D_{X_n} \otimes \cl B(H)   &  \cl D_{X_{n+1}} \otimes \cl B(H) 
             \ar[l, "\tilde{\cl{E}}_{X_{n+1},X_n}"] 
              \end{tikzcd}
                \end{center}
is commutative for each $ n\in \bb N$.

\begin{theorem} \label{th_indcorresp}
Let $H$ be a Hilbert space, and $X = (X_n)_{n\in \bb{N}}$ and  
$A = (A_n)_{n\in \bb{N}}$  be inductive families of sets. 
\begin{enumerate} 
    \item 
If
 $\Phi : C(\Omega_A)\rightarrow L^{\infty}(\Omega_X)\bar{\otimes}\cl B(H)$ 
is a unital completely positive map and 
\begin{equation}\label{eq_gammanst}
\Phi_n := \tilde{\cl{E}}_{X_n}^{\infty}\circ\Phi\circ \iota_A^{(n)}, \ \ \ 
n \in \bb N,
\end{equation}
then the family $(\Phi_n)_{n\in \bb N}$ is 
inductive. 

\item 
If $\Phi_n : \cl D_{A_n} \to \cl D_{X_n} \otimes\cl B(H)$
is a unital completely positive map, $n\in \bb{N}$, 
such that the family 
$(\Phi_n)_{n\in \bb N}$ is inductive then there exists 
a unique unital completely positive map 
$\Phi : C(\Omega_A)\rightarrow L^{\infty}(\Omega_X) \bar{\otimes}\cl B(H)$ satisfying
(\ref{eq_gammanst}). 
\end{enumerate}
\end{theorem}

\begin{proof}
(i) 
Since 
\begin{align*}
\tilde{\cl{E}}_{X_n}^{\infty} = \tilde{\cl{E}}_{X_{n+1},X_n}\circ \tilde{\cl{E}}_{X_{n+1}}^{\infty} \ \ \ \text{ and } \ \ \ 
\iota_{A}^{(n)} = \iota_{A}^{(n+1)} \circ \iota_{A_n,A_{n+1}},
\end{align*} 
condition (\ref{eq_inre}) is implied by 
(\ref{eq_gammanst}).

\par (ii) For each $n \in \bb N$, let  
$\Psi_{n} : C(\Omega_{A}) \to L^{\infty}(\Omega_{X}) \bar{\otimes}\cl B(H)$
be defined by setting $\Psi_{n} := \tilde{\iota}^{(n)}_{X,\infty}\circ \Phi_{n} 
\circ \cl E_{A_{n}}^{\infty}|_{C(\Omega_A)}$; 
note that the maps $ \Psi_{n}$ are unital and completely positive.
The  
sequence $(\Psi_n)_{n\in \bb{N}}$ has a BW cluster point 
$\Phi : C(\Omega_{A}) \to L^{\infty}(\Omega_{X}) \bar{\otimes}\cl B(H)$,
whose existence follows from the 
compactness of 
${\rm UCP}(C(\Omega_A),L^\infty(\Omega_X)\bar{\otimes}\cl B(H))$ in the BW topology
(\cite[Theorem 3.14]{gb}).
Next we show that (\ref{eq_gammanst}) is satisfied 
for the map $\Phi$.
Consider $k > n$ and note that 
\begin{align*}
    {\iota}^{(n)}_{A} =  {\iota}^{(k)}_{A} \circ  {\iota}_{A_{n},A_{k}}, \ \ \  \text{ and } \ \ \ 
    \tilde{\cl E}_{X_{n}}^{\infty} =  \tilde{\cl E}_{X_{k},X_{n}}^{}  \circ \tilde{\cl E}_{X_{k}}^{\infty} 
    \end{align*}
so that 
$$\tilde{\cl E}_{X_{n}}^{\infty} \circ \tilde{\iota}^{(k)}_{X,\infty}  
= \tilde{\cl E}_{X_{k},X_{n}} \ \ \ \text{ and } \ \ \  
{\cl E}_{A_{k}}^{\infty}\circ {\iota}^{(n)}_{A,\infty}  = {\iota}_{A_{n},A_{k}}$$
(here we have set $\tilde{\cl E}_{X_{k},X_{n}} = \cl E_{X_{k},X_{n}} \otimes \id_{\cl B(H)}$).
Therefore, for $ k > n$,  we have 
$$\tilde{\cl E}_{X_{n}}^{\infty} \hspace{-0.04cm} \circ \hspace{-0.04cm}
\Psi_{k} \hspace{-0.04cm} \circ \hspace{-0.04cm} \iota^{(n)}_{A,\infty}  
= 
\tilde{\cl E}_{X_{n}}^{\infty}  \hspace{-0.04cm} \circ \hspace{-0.04cm} \tilde{\iota}^{(k)}_{X,\infty}
\hspace{-0.04cm} \circ \hspace{-0.04cm} \Phi_{k} 
\hspace{-0.04cm} \circ \hspace{-0.04cm} \cl E_{A_{k}}^{\infty} 
\hspace{-0.04cm} \circ \hspace{-0.04cm} \iota^{(n)}_{A,\infty} 
=  \tilde{\cl E}_{X_{k},X_{n}}^{} \hspace{-0.04cm} \circ \hspace{-0.04cm} \Phi_{k} 
\hspace{-0.04cm} \circ \hspace{-0.04cm} \iota_{A_{n},A_{k}} 
= \Phi_{n}, $$
where the last equation follows from the inductivity relations (\ref{eq_inre}). 
Letting $k\to\infty$, we obtain (\ref{eq_gammanst}). 

We claim that the unital completely positive map $\Phi$ satisfying 
(\ref{eq_gammanst}) is unique. 
Indeed, if $\Phi'$ is another such 
map satisfying (\ref{eq_gammanst}), then, as $\tilde{\cl E}_{X_{m}}^{\infty}  \circ \Phi(S)=\tilde{\cl E}_{X_{m}}^{\infty}  \circ \Phi'(S)$, for every $S\in \iota_A^{(n)}(\cl D_{A_n})$, $m>n$, we get, using density arguments, that $\Phi(S)=\Phi'(S)$, giving the uniqueness. 
\end{proof}


\subsection{Definition and universal property}

We next identify a canonical operator system, associated with an inductive family of sets,
which will serve as universal 
encoding object for each of the players of a non-local game over Cantor spaces. 
For $d\in \bb{N}$, recall the operator system 
$\cl S_{[d], A_n}$
defined, in (\ref{eq_SXAeq}), as a coproduct of 
$d$ copies of $\cl D_{A_n}$. Denoting by $\iota_k$ the 
embedding of $\cl D_{A_n}$ in the $k$-th term of 
$\cl S_{[d],A_n}$, let 
$\beta_{d,A_n} : \cl D_{A_n}\to \cl S_{[d], A_n}$ be the unital completely positive map, 
given by 
$$\beta_{d,A_n}(u) = \frac{1}{d}\left(\iota_1(u) + \cdots + 
\iota_d(u)\right), \ \ \ u\in \cl D_{A_n}.$$
We have that the maps 
$$\dot{\oplus}_{i=1}^{|X_n|} \iota_{A_n,A_{n+1}} : \cl S_{X_n,A_n}\to \cl S_{X_n,A_{n+1}}$$
and 
$$\dot{\oplus}_{i=1}^{|X_n|} \beta_{d_n^X,A_{n+1}} : \cl S_{X_n,A_{n+1}}\to \cl S_{X_{n+1},A_{n+1}}$$
are unital and completely positive; 
thus, the composition 
$$\gamma_{X_n,A_n} = \left(\dot{\oplus}_{i=1}^{|X_n|} \beta_{d_n^X,A_{n+1}}\right)\circ
\left(\dot{\oplus}_{i=1}^{|X_n|} \iota_{A_n,A_{n+1}}\right)$$
is a unital completely positive map from $\cl S_{X_n,A_n}$ to $\cl S_{X_{n+1},A_{n+1}}$. 
We thus obtain an inductive sequence 
\begin{equation}\label{eq_opsind}
\cl S_{X_1,A_1}\stackrel{\gamma_{X_1,A_1}}{\longrightarrow}
\cl S_{X_2,A_2}\stackrel{\gamma_{X_2,A_2}}{\longrightarrow}
\cl S_{X_3,A_3}\stackrel{\gamma_{X_3,A_3}}{\longrightarrow}\cdots
\end{equation} 
in the operator system category.
 We let 
$\cl S_{X,A} = \varinjlim \cl S_{X_n,A_n}$ 
be the corresponding inductive limit. 
We let $\gamma^{(n)}_{X,A} : \cl S_{X_n,A_n}\to \cl S_{X,A}$
be the canonical unital completely positive map, arising from the inductive sequence 
(\ref{eq_opsind}), $n\in \bb{N}$. 

\begin{remark}\label{p_gammacoi} 
\rm 
Let $(X_n)_{n\in \bb{N}}$ and $(A_n)_{n\in \bb{N}}$ be inductive families of sets. 
Using known results about inductive limits and 
coproducts in the operator system category,
namely \cite[Proposition 4.13]{mt} and \cite[Section 8]{kavruk-nucl}, one can show that 
the maps $\gamma_{X_n,A_n}$ and $\gamma^{(n)}_{X,A}$ are 
unital complete order isomorphisms; thus, $\cl S_{X_n,A_n}$ can be canonically identified 
with an operator subsystem of $\cl S_{X,A}$, $n\in \bb{N}$. 
Since this fact will not be needed in the sequel, 
we do not include its proof.
\end{remark}

\begin{theorem}\label{c_statSXA}
Let $H$ be a Hilbert space, and $(X_n)_{n\in \bb{N}}$ and $(A_n)_{n\in \bb{N}}$ be inductive families of sets. 
The unital completely positive maps $\Gamma : \cl S_{X,A}\to \cl B(H)$ are in a canonical 
bijective correspondence with the unital completely positive maps 
$\Phi : C(\Omega_A)\to L^{\infty}(\Omega_X)\bar{\otimes} \cl B(H)$.
\end{theorem}

\begin{proof}
For $k\in \bb{N}$ and $\omega\in \cl D_{X_k}$, let $L_{\omega} : \cl D_{X_k}\otimes \cl B(H)\to \cl B(H)$
be the slice map, given by 
$$L_{\omega}(S\otimes T) = \langle \omega,S\rangle T, \ \ \ S\in \cl D_{X_k}, T\in \cl B(H)$$
(we recall that the duality is with respect to normalised traces).

By the universal property of the inductive limit, 
the unital completely positive maps $\Gamma : \cl S_{X,A}\to \cl B(H)$ are in a canonical correspondence 
with the sequences $(\Gamma_n)_{n\in \bb{N}}$ of unital completely positive maps, where $\Gamma_n : \cl S_{X_n,A_n}\to \cl B(H)$, $n\in \bb{N}$, 
satisfy the conditions
\begin{equation}\label{eq_snsn+1}
\Gamma_n = \Gamma_{n+1}\circ \gamma_{X_n,A_n}, \ \ \ n\in \bb{N}.
\end{equation}
On the other hand, the unital completely positive maps $\Gamma_n : \cl S_{X_n,A_n} \to \cl B(H)$ are in a canonical correspondence 
with unital completely positive maps  $\Phi_n : \cl D_{A_n}\to \cl D_{X_n}\otimes \cl B(H)$ via the 
assignment
$\Phi_n(\delta_a) = \sum_{x\in X_n} \delta_x\otimes \Gamma_n(e_{x,a})$. 
We note that the latter equality is equivalent to the identities 
\begin{equation}\label{eq_dualslice}
\lvert X_n\rvert
L_{\delta_x}(\Phi_n(\delta_a)) = \Gamma_n(e_{x,a}), \ \ \ x\in X_n, a\in A_n.
\end{equation}
Finally, observe that 
condition (\ref{eq_snsn+1}) is equivalent to condition (\ref{eq_inre}) being satisfied for the family 
$(\Phi_n)_{n\in \bb{N}}$, 
as follows from the fact that, if $x\in X_n$ and $a\in A_n$, then
\begin{eqnarray*}
\Gamma_{n+1}(\gamma_{X_n,A_n}(e_{x,a}))
& = & 
\frac{1}{d_n^X} \sum_{\mu_n = 1}^{d_n^X}
\sum_{\lambda_n =1}^{d_n^A} 
\Gamma_{n+1}(e_{(x,\mu_n), (a,\lambda_n)})\\
& = & 
\lvert X_n\rvert
\sum_{\mu_n = 1}^{d_n^X}
\sum_{\lambda_n =1}^{d_n^A} 
L_{\delta_{(x,\mu_n)}}(\Phi_{n+1}(\delta_{(a,\lambda_n)}))\\
& = & 
\lvert X_n\rvert \sum_{\mu_n = 1}^{d_n^X}
L_{\delta_{(x,\mu_n)}}((\Phi_{n+1}\circ\iota_{A_n,A_{n+1}})(\delta_a))\\
& = & 
\lvert X_n\rvert L_{\iota_{X_n,X_{n+1}}(\delta_x)}((\Phi_{n+1}\circ\iota_{A_n,A_{n+1}})(\delta_a))\\
& = & 
\lvert X_n\rvert
L_{\delta_x}((\tilde{\cl{E}}_{X_{n+1},X_n}
\hspace{-0.05cm}\circ\hspace{-0.05cm}\Phi_{n+1}
\hspace{-0.05cm}\circ \hspace{-0.05cm} \iota_{A_n,A_{n+1}})(\delta_a)).\\
\end{eqnarray*} 
The statement now follows from Theorem \ref{th_indcorresp}. \end{proof}

\begin{remark}\label{r_forfinnew}
\rm 
We point out for further use that the statement of Theorem \ref{c_statSXA} is true
and, up to our knowledge, part of folklore 
in the case where  $X$ and $A$ are finite sets instead of inductive families; 
a proof readily follows from that of Theorem \ref{c_statSXA} with the 
straightforward modifications. 
\end{remark}

\begin{remark}\label{r_forcp}
\rm 
The statement of Theorem \ref{c_statSXA} remains true when 
the map $\Gamma$ is a (not necessarily unital) completely positive map; this 
follows by inspection of the proof, together with the fact that
the operator system inductive limit satisfies a universal 
property for inductive families of completely positive maps that are not necessarily unital.
To show the latter fact, suppose that 
(\ref{eq_indlimS}) is an inductive sequence in the operator system category, 
$\cl R$ is an operator system, and 
$\rho_k : \cl S_k\to \cl R$ are completely positive maps, 
such that $\rho_{k+1}\circ\phi_k = \rho_k$, $k\in \bb{N}$.
Since the map $\phi_k$ in (\ref{eq_indlimS}) is unital, we have that 
there exists $w\in \cl R$ such that 
$\rho_k(1_{\cl S_k}) = w$ for every $k\in \bb{N}$; 
clearly, $0\leq w\leq 1_{\cl R}$.
Let $(s_k)_{k\in \bb{N}}$ be an inductive sequence of states, 
where $s_k : \cl S_k\to \bb{C}$, and $s : \varinjlim \cl S_k\to \bb{C}$ be 
the associated state on the inductive limit operator system. 
Let
$\alpha_k : \cl S_k\to \cl R$ be the map, given by 
$\alpha_k(u) = s_k(u)(1-w)$, 
and $\tilde{\rho}_k = \rho_k + \alpha_k$; thus, 
$\tilde{\rho}_k : \cl S_k\to \cl R$ is a unital completely positive map, $k\in \bb{N}$. 
Moreover, 
$$(\tilde{\rho}_{k+1}\circ\phi_k)(u) = \rho_{k+1}(\phi_k(u)) + s(u)(1-w) = 
\tilde{\rho}_k(u), \ \ \ u\in \cl S_k.$$
By the universal property of the inductive limit (for unital completely positive maps), 
there exists a unital completely positive map 
$\tilde{\rho} : \varinjlim \cl S_k \to \cl R$ such that 
$\tilde{\rho}\circ\phi_{k,\infty} = \tilde{\rho}_k$ for every $k\in \bb{N}$. 

Let $\alpha : \varinjlim \cl S_k\to \cl R$ be the map, given by 
$\alpha(u) = s(u)(1-w)$. 
We have that 
$$\tilde{\rho}^{(n)}(\phi_{k,\infty}^{(n)}(u)) - \alpha^{(n)}(\phi_{k,\infty}^{(n)}(u)) \in 
M_n(\cl R)^+, \ \ \ u\in M_n(\cl S_k)^+.$$ 
By density, $\rho := \tilde{\rho} - \alpha$ is completely positive.
Finally, if $k\in \bb{N}$ then 
$$(\rho\circ\phi_{k,\infty})(u) =  
\tilde{\rho}(\phi_{k,\infty}(u)) - \alpha(\phi_{k,\infty}(u)) = 
\tilde{\rho}_k(u) - \alpha_k(u)
= \rho_k(u).$$  
\end{remark}

\begin{remark}\label{r_BWhomeom}
\rm 
Let $(X_n)_{n\in \bb{N}}$ and $(A_n)_{n\in \bb{N}}$ be inductive families of sets.
Suppose that $x\in X_n$ and $a\in A_n$ for some $n\in \bb{N}$. 
Equation (\ref{eq_dualslice}) can be rewritten as 
$$\lvert X_n\rvert
\langle\delta_x\otimes\omega,\Phi_n(\delta_a)\rangle = \langle\Gamma_n(e_{x,a}),\omega\rangle, 
\ \ \omega\in \cl B(H)_*,$$
that is, 
\begin{equation}\label{eq_forifu}
\left\langle \frac{1}{\tau_X(\iota^{(n)}_{X,1}(\delta_x))}\iota^{(n)}_{X,1}(\delta_x) \otimes\omega,\Phi(\iota^{(n)}_{A,\infty}(\delta_a))\right\rangle = \langle\Gamma(\gamma^{(n)}_{X,A}(e_{x,a})),\omega\rangle
\end{equation}
for every $\omega\in \cl B(H)_*$. 
By identity (\ref{eq_forifu}) and uniform boundedness, 
the map $\Phi\to \Gamma$ is BW continuous. 
Since the linear span of the elements of the form 
$\frac{1}{\tau_X(\iota^{(n)}_{X,1}(\delta_x))}\iota^{(n)}_{X,1}(\delta_x) \otimes\omega$ is 
dense in $L^1(\Omega_X)\hat{\otimes}\cl B(H)_*$, we have, in fact, that the 
correspondence $\Phi \leftrightarrow \Gamma$ is a BW-BW homeomorphism. 
\end{remark}

We note that, as the predual of $L^{\infty}(\Omega_X)$, the space $L^1(\Omega_X)$
admits a canonical operator space structure, and that, if 
$C(\Omega_A)\hat{\otimes} L^1(\Omega_X)$ denotes the operator space projective 
tensor product, up to a canonical complete isometry we have that 
$$\left(C(\Omega_A)\hat{\otimes} L^1(\Omega_X)\right)^* = {\rm CB}(C(\Omega_A),L^{\infty}(\Omega_X))$$
(see \cite[Proposition 7.1.2]{er}). 
By the previous paragraph, there exists a canonical weak*-homeomorphic order 
isomorphism between the positive cones of 
$\cl S_{X,A}^*$ and ${\rm UCP}(C(\Omega_A),L^{\infty}(\Omega_X))$. 
Passing to preduals, we obtain a order isomorphism 
between $\cl S_{X,A}^+$ and a dense subspace of the predual cone 
$\left(C(\Omega_A)\hat{\otimes} L^1(\Omega_X)\right)^+$ of the cone 
${\rm CP}(C(\Omega_A),L^{\infty}(\Omega_X))$. 
A straightforward argument shows that the latter correspondence can be 
extended to the whole of $\cl S_{X,A}$. 
Through the latter identification,
for $x\in X_n$ and $a\in A_n$, the element $\gamma^{(n)}_{X,A}(e_{x,a})$ corresponds to 
the elementary tensor 
$\chi_{\tilde{x}\times\tilde{a}} := \chi_{\tilde{x}}\otimes\chi_{\tilde{a}}$, where 
\begin{equation}\label{eq_cylinddd}
\tilde{x} = \left\{xx' : x'\in \prod_{i=n+1}^{\infty} [d_i^X]\right\} \ \mbox{ and } \ 
\tilde{a} = \left\{aa' : a'\in \prod_{i=n+1}^{\infty} [d_i^A]\right\}.
\end{equation}


\section{Cantor no-signalling correlations}\label{s_rvs}

In this section we define no-signalling correlations
over Cantor spaces and provide characterisations 
thereof in terms of states on the maximal operator system tensor 
product of operator systems from the class introduced in Section \ref{s_ovic}. 
In view of the nuclearity of abelian C*-algebras, in the sequel we will use the symbol $\otimes$ for their C*-algebraic tensor product. 

Assume that $S, T, U$ and $V$ are finite sets. 
A no-signalling correlation 
$p = \{(p(u,v|s,t))_{u,v} : s\in S, t\in T\}$
(see Section \ref{s_finite}) 
gives rise to the unital (completely) positive map
$\Gamma_p :  \cl D_U \otimes \cl  D_V \to \cl 
 D_S\otimes \cl  D_T $, 
given by 
$$\Gamma_p(\delta_u\otimes\delta_v) = 
\sum_{s\in S} \sum_{t\in T} p(u,v|s,t)
\delta_s\otimes\delta_t, \ \ \ u\in U, v\in V;
$$
it is straightforward to verify that, moreover, 
\begin{equation}\label{eq_Gammap}
\Gamma_p(\cl D_U\otimes 1_{\cl D_V})\subseteq \cl D_S \otimes 1_{\cl D_T} \ \mbox{ and } \ 
\Gamma_p(1_{\cl D_U}\otimes \cl D_V)\subseteq 1_{\cl D_S}\otimes \cl D_T.
\end{equation}
Conversely, every unital (completely) positive 
map $\Gamma :  \cl D_U \otimes \cl  D_V \to \cl 
D_S\otimes \cl  D_T$ satisfying the conditions (\ref{eq_Gammap}) is easily seen to have the form 
$\Gamma = \Gamma_p$ for some no-signalling correlation $p$. 
Therefore, by abuse of terminology, we use the term 
\lq\lq no-signalling correlation'' in reference to 
unital (completely) positive maps satisfying 
(\ref{eq_Gammap}). 
Fixing inductive families of finite sets
$X = (X_n)_{n\in \bb{N}}$, $Y = (Y_n)_{n\in \bb{N}}$, 
$A = (A_n)_{n\in \bb{N}}$ and $B = (B_n)_{n\in \bb{N}}$, 
these observations justify the following definition.

\begin{definition}\label{d_nsCOm}
\rm A unital completely positive map $$\Gamma:C(\Omega_A)\otimes C(\Omega_B)\rightarrow L^\infty(\Omega_X)\bar{\otimes}L^\infty(\Omega_Y)$$
will be called a \emph{no-signalling correlation} over the quadruple $(X,Y,A,B)$ if
$$\Gamma\left(C(\Omega_A)\otimes 1_{B}\right)\subseteq L^\infty(\Omega_X)\otimes 1_{Y}$$ and
$$\Gamma\left(1_{A}\otimes C(\Omega_B)\right)\subseteq 1_{X}\otimes L^\infty(\Omega_Y).$$
\end{definition}

We denote by 
$\cl C_{\rm ns}(X,Y,A,B)$ the set of all no-signalling correlations over the quadruple $(X,Y,A,B)$, and write $\cl C_{\rm ns}$ in case no confusion may arise.

Given a unital completely positive map 
$$\Gamma : C(\Omega_A)\otimes C(\Omega_B) \to 
L^\infty(\Omega_X)\bar{\otimes} L^\infty(\Omega_Y),$$ 
define (unital completely positive) maps 
$$\Gamma_n : \cl D_{A_n}\otimes\cl D_{B_n} \to 
\cl D_{X_n}\otimes \cl D_{Y_n}, \ \ \ n \in \bb N,$$ 
by setting
\begin{equation} \label{Gamma_n}
    \Gamma_n:=(\cl{E}_{X_n}^{\infty}\otimes
\cl{E}_{Y_n}^{\infty})\circ\Gamma\circ(\iota_{A}^{(n)}\otimes\iota^{(n)}_{B});
\end{equation} 
by Theorem \ref{th_indcorresp}, 
the family $(\Gamma_n)_{n\in \bb N}$ is inductive, that is,
\begin{align}\label{eq_indrel}
    \Gamma_n=(\cl{E}_{X_{n+1},X_n}\otimes \cl{E}_{Y_{n+1},Y_n})\circ \Gamma_{n+1}\circ (\iota_{A_n,A_{n+1}}\otimes\iota_{B_n,B_{n+1}}), \ \ \ n\in \bb N;
\end{align}
we say that the family $(\Gamma_n)_{n\in \bb{N}}$ is \emph{associated with} 
the map $\Gamma$.

\begin{proposition} \label{p_gammans} 
Let $\Gamma : C(\Omega_A)\otimes C(\Omega_B)\to 
L^\infty(\Omega_X)\bar{\otimes} L^\infty(\Omega_Y)$ be a unital completely positive map and $(\Gamma_n)_{n \in \bb N}$ be the inductive family associated 
with $\Gamma$. The following are equivalent:
\begin{itemize}
\item[(i)] $\Gamma\in\cl{C}_{\rm ns}(X,Y,A,B)$;
\item[(ii)] $\Gamma_n\in\cl{C}_{\rm ns}(X_n,Y_n,A_n,B_n)$ for every $n\in\bb{N}$.
\end{itemize}
\end{proposition}
\begin{proof}
    (i) $\Rightarrow$ (ii) 
    Let $ n \in \bb N$ and $ f\in \cl D_{A_n}$. 
By symmetry, it suffices to show that $\Gamma_n(f \otimes 1_{B_n}) \in \cl D_{X_n} \otimes 1_{B_n}$. 
Since $\Gamma$ is no-signalling,  $\Gamma(\iota^{(n)}_{A}(f)\otimes 1_{B}) \in L^{\infty}(\Omega_{X})\otimes 1_{B}$. 
    The claim follows from the fact that $ (\cl E_{X_n}\otimes \cl E_{Y_n})(L^{\infty}(\Omega_{X})\otimes 1_{B}) = \cl D_{X_n} \otimes 1_{B_n}$. 
    
    (ii) $\Rightarrow$ (i) 
    Assume that $ \Gamma_n$ is no-signalling for every $ n \in \bb N$. It suffices to show that $ \Gamma(\iota^{(n)}(f_n) \otimes 1_{B}) \in L^{\infty}(\Omega_X ) \otimes 1_Y$ for every $ f_n\in D_{A_n}$ and $ n \in \bb N$. Indeed, if we then pick $ f\in C(\Omega_A)$ and set $ f_n:= \cl E_{A_n}(f) \in \cl D_{A_n}$, we have that $ f = \lim_{n\to \infty}\iota^{(n)}_{A}(f_n)$ in norm and, as $\Gamma$ is continuous,
$$ \Gamma(f\otimes 1_{B})= \lim_{n \to \infty} \Gamma(\iota^{(n)}(f_n) \otimes 1_{B}) \in L^{\infty}(\Omega_X ) \otimes 1_Y.$$ 
Let $ f\in \cl D_{A_n}$ for some $ n \in \bb N$.
Let $ k \geq n$, and note that 
\begin{align*}
    & \left((\cl E_{X_k}^{\infty} \otimes \cl E_{Y_k}^{\infty}) \circ\Gamma\right)(\iota_{A}^{(n)}(f)\otimes 1_B)\\ 
    & = (\cl E_{X_k}^{\infty} \otimes \cl E_{Y_k}^{\infty}) \circ\Gamma \circ (\iota_{A}^{(k)} \otimes \iota^{(k)}_{B})( \iota_{A_n,A_{k}} (f)\otimes 1_{B_{k}}) 
    = \Gamma_{k}( \iota_{A_n,A_{k}} (f)\otimes 1_{B_{k}}).  
\end{align*}
Thus, $  (\cl E_{X_k}^{\infty} \otimes \cl E_{Y_k}^{\infty}) \circ\Gamma(\iota_{A}^{(n)}(f)\otimes 1_B) \in \cl D_{X_k} \otimes 1_{Y_k}$ for all $ k \geq n$ since $\Gamma_k$ is no-signalling for all $ k\in \bb N$. Hence, 
\[ (\iota^{(k)}_{X}\otimes \iota^{(k)}_{Y})\circ(\cl E_{X_k}^{\infty} \otimes \cl E_{Y_k}^{\infty}) \circ\Gamma(\iota_{A}^{(n)}(f)\otimes 1_B) \in L^{\infty}(\Omega_{X})\otimes 1_{Y} \]
for all $ k \geq n$ and if we take the limits in norm as 
$k\to\infty$, we conclude that 
\[ \Gamma(\iota_{A}^{(n)}(f)\otimes 1_B) \in L^{\infty}(\Omega_{X})\otimes 1_{Y},\]
as desired. The result follows by symmetry.
\end{proof}


Given a no-signalling correlation $\Gamma$ over the quadruple $(X,Y,A,B)$, we write 
$p_{\Gamma,n}$ for the family of conditional probability distributions on $A_n\times B_n$, indexed by 
$X_n\times Y_n$, corresponding to $\Gamma_n$ by (\ref{Gamma_n}), that is, 
$$p_{\Gamma,n}(a,b|x,y) =\lvert X_n\rvert  \lvert Y_n\rvert \left\langle \delta_x\otimes \delta_y,\Gamma_n(\delta_a\otimes \delta_b)\right\rangle,$$
where $x\in X_n$, $y\in Y_n$, $a\in A_n$, $b\in B_n$ and,
as before, the pairing is given by the normalised traces 
(that is, $\langle \delta_1\otimes\delta_2,\omega_1\otimes\omega_2\rangle=\tr_{X_n}(\delta_1\omega_1)\tr_{Y_n}(\delta_2\omega_2)$). 
By Proposition \ref{p_gammans}, the correlations $p_{\Gamma,n}$ are no-signalling. 
By Theorem \ref{th_orcha}, there exist states $s_{\Gamma_n} : \cl{S}_{X_n,A_n}\otimes_{\rm max}\cl{S}_{Y_n,B_n}\rightarrow\bb{C}$, such that 
$$s_{\Gamma_n}(e_{x,a}\otimes e_{y,b}) = \lvert X_n\rvert\lvert Y_n\rvert \langle \delta_x\otimes\delta_y,\Gamma_n(\delta_a\otimes\delta_b)\rangle$$
for $x\in X_n$, $y\in Y_n$, $a\in A_n$ and $b\in B_n$.
The proof of the next lemma is similar to that of Theorem \ref{c_statSXA} and is 
omitted.

\begin{lemma} \label{lem_ind_stat}
Let $(X_n)_{n\in \bb{N}}$, $(Y_n)_{n\in \bb{N}}$, 
$(A_n)_{n\in \bb{N}}$ and $(B_n)_{n\in \bb{N}}$ be inductive families of sets, and 
$\Gamma_n : \cl D_{A_n}\otimes\cl D_{B_n} \to \cl D_{X_n}\otimes\cl D_{Y_n}$
be a no-signalling correlation, $n\in \bb{N}$. 
The family $(\Gamma_n)_{n\in\bb{N}}$
is inductive if and only if 
$$s_{\Gamma_{n+1}}\circ(\gamma_{X_n,A_n}\otimes\gamma_{Y_n,B_n})
= s_{\Gamma_n},\ \ \ n\in\bb{N}.$$
\end{lemma}

In the proof of the next theorem, we will need an auxiliary fact 
about operator system inductive limits.
Let $\tau$ be an operator system tensor product (see \cite{kptt}). We will say that 
$\tau$ \emph{commutes with inductive limits} if, 
for every inductive sequence
$$\cl S_1\stackrel{\phi_1}{\longrightarrow} 
\cl S_2\stackrel{\phi_2}{\longrightarrow} 
\cl S_3\stackrel{\phi_3}{\longrightarrow} \cdots$$
in the operator system category, and every operator system $\cl T$, 
we have that 
$$\varinjlim (\cl S_k \otimes_{\tau} \cl T) \cong 
(\varinjlim \cl S_k) \otimes_{\tau} \cl T,$$
up to a canonical complete order isomorphism.

\begin{lemma}\label{l_comminl}
Let $\tau$ be an operator system tensor product that commutes with 
inductive limits, and
$$\cl S_1\stackrel{\phi_1}{\longrightarrow} 
\cl S_2\stackrel{\phi_2}{\longrightarrow} 
\cl S_3\stackrel{\phi_3}{\longrightarrow} \cdots$$
and
$$\cl T_1\stackrel{\psi_1}{\longrightarrow} 
\cl T_2\stackrel{\psi_2}{\longrightarrow} 
\cl T_3\stackrel{\psi_3}{\longrightarrow} \cdots$$
be inductive sequences in the operator system category with 
inductive limits $\cl S$ and $\cl T$, respectively. 
Then 
$$\varinjlim (\cl S_k \otimes_{\tau} \cl T_k) \cong 
\cl S \otimes_{\tau} \cl T,$$
up to a canonical complete order isomorphism. 
\end{lemma}

\begin{proof}
For brevity, we will use the symbol $\id_{k}$ to denote the identity map 
on either $\cl S_k$ or $\cl T_k$, depending on the context. 
Write $\cl R := \varinjlim (\cl S_k \otimes_{\tau} \cl T_k)$ and set
$\theta_{n} = \phi_n \otimes \id_{\cl T}$; 
thus, $\theta_{n} : \cl S_n\otimes_{\tau}\cl T \to \cl S_{n+1}\otimes_{\tau}\cl T$
is a unital completely positive map, $n\in \bb{N}$. 
Trivially, 
$$\theta_n \circ (\id\hspace{-0.05cm}\mbox{}_{n}\otimes \psi_{k,\infty}) 
= 
(\id\hspace{-0.05cm}\mbox{}_{n+1}\otimes \psi_{k,\infty})\circ (\phi_{n}\otimes \id\hspace{-0.05cm}\mbox{}_{k}), \ \ k \in \bb N.$$
On the other hand, using  \cite[Remark 2.15]{mt}, 
we have that the diagram 
\begin{center}
          \begin{tikzcd}[row sep=large, column sep=large] 
             \cl S_n\otimes_{\tau} \cl T_n  \arrow[d," \phi_{n,n}\otimes \id_n"] \ar[r, " \id_n \otimes \psi_n"]  &  
             \cl S_n\otimes_{\tau} \cl T_{n+1}  \arrow[d," \phi_{n,n+1}\otimes \id_{n+1}"] \ar[r, " \id_n \otimes \psi_{n+1}"] & 
             \cl S_n\otimes_{\tau} \cl T_{n+2}  \arrow[d," \phi_{n,n+2}\otimes \id_{n+2}"] \ar[r] & 
             \cdots \\
             \cl S_n\otimes_{\tau} \cl T_n  \ar[r, "\phi_{n}\otimes \psi_n"]  &  
             \cl S_{n+1}\otimes_{\tau} \cl T_{n+1}  \ar[r,"\phi_{n+1}\otimes \psi_{n+1}"] & 
             \cl S_{n+2}\otimes_{\tau} \cl T_{n+2}  \ar[r] & 
             \cdots
              \end{tikzcd}
               \end{center}
yields a canonical unital completely positive map 
$$ \gamma_n:\cl S_n\otimes_{\tau}\cl T \to \cl R,$$ 
such that 
$$\gamma_n \circ (\id\hspace{-0.05cm}\mbox{}_n\otimes \psi_{k,\infty})= (\phi_{k}\otimes \psi_{k})_{\infty}\circ (\phi_{n,k}\otimes \id\hspace{-0.05cm}\mbox{}_k), \ \ \ k\geq n.$$ 
Thus, the pair $(\cl R, (\gamma_n)_{n\in \bb N})$ satisfies 
\begin{align*}
(\gamma_{n+1} \circ \theta_{n})\circ
(\id\hspace{-0.05cm}\mbox{}_{n}\otimes \psi_{k,\infty}) 
&
= \gamma_{n+1} \circ (\id\hspace{-0.05cm}\mbox{}_{n+1}\otimes \psi_{k,\infty})\circ (\phi_{n}\otimes \id\hspace{-0.05cm}\mbox{}_{k})\\
&= (\phi_{k}\otimes \psi_{k})_{\infty}\circ (\phi_{n+1,k}\otimes \id\hspace{-0.05cm}\mbox{}_k) \circ (\phi_n\otimes\id\hspace{-0.05cm}\mbox{}_{k})\\
&= (\phi_{k}\otimes \psi_{k})_{\infty}\circ (\phi_{n,k}\otimes \id\hspace{-0.05cm}\mbox{}_k)\\
& = \gamma_{n} \circ (\id\hspace{-0.05cm}\mbox{}_{n}\otimes \psi_{k,\infty})
\end{align*}
for each $k\geq n+1$, hence $\gamma_{n+1} \circ \theta_{n}= \gamma_n$ and by the universal property of the inductive limit $ \varinjlim(\cl S_{k}\otimes_{\tau} \cl T)$ (see \cite[Definition 2.13]{mt}) there exists a 
canonical unital completely positive map 
$\alpha: \cl S\otimes_{\tau}\cl T \to \cl R$, 
such that 
$$\alpha \circ (\phi_{n,\infty}\otimes 
\id\hspace{-0.05cm}\mbox{}_{\cl T}) = \gamma_n, \ \ \  n \in \bb N.$$

Similarly, the diagram 
\begin{center}
          \begin{tikzcd}[row sep=large, column sep=large] 
             \cl S_1\otimes_{\tau} \cl T_1  \arrow[d, "\id_{1}\otimes \psi_{1,\infty}"] \ar[r, "\phi_1 \otimes \psi_1"]  &  
             \cl S_2\otimes_{\tau} \cl T_2  \arrow[d, "\id_{2}\otimes \psi_{2,\infty}"] \ar[r, "\phi_2 \otimes \psi_2"] & 
             \cl S_3\otimes_{\tau} \cl T_3  \arrow[d, "\id_{3}\otimes \psi_{3,\infty}"] \ar[r] & 
             \cdots \\
             \cl S_1\otimes_{\tau} \cl T  \ar[r, "\phi_1 \otimes \id_{\cl T}"]  &  
             \cl S_2\otimes_{\tau} \cl T  \ar[r, "\phi_2 \otimes \id_{\cl T}"] & 
             \cl S_3\otimes_{\tau} \cl T  \ar[r] & 
             \cdots,
              \end{tikzcd}
               \end{center}
 yields a canonical unital completely positive map 
$$\beta : \cl R \to \cl S\otimes_{\tau}\cl T,$$ 
such that 
$$\beta \circ (\phi_{n}\otimes \psi_{n})_{\infty}= (\phi_{n,\infty}\otimes 
\id\hspace{-0.05cm}\mbox{}_{\cl T}) \circ (\id\hspace{-0.05cm}\mbox{}_{n}\otimes \psi_{n,\infty}), \ \ \  n \in \bb N.$$
We show that the maps $\alpha$ and $\beta$ are inverse to each other; 
indeed, 
\begin{align*}
\hspace{-1.3cm}
    \alpha\circ \beta \circ (\phi_{n}\otimes \psi_{n})_{\infty}
    & = \alpha \circ (\phi_{n,\infty}\otimes 
\id\hspace{-0.05cm}\mbox{}_{\cl T}) \circ (\id\hspace{-0.05cm}\mbox{}_{n}\otimes \psi_{n,\infty})\\
    &=\gamma_n \circ (\id\hspace{-0.05cm}\mbox{}_{n}\otimes \psi_{n,\infty})
     = (\phi_{n}\otimes \psi_{n})_{\infty} 
    \hspace{-0.1cm}\circ \hspace{-0.1cm}
    (\phi_{n,n} \otimes \id\hspace{-0.05cm}\mbox{}_{n})\\
    &= (\phi_{n}\otimes \psi_{n})_{\infty}
\end{align*}
for all $ n \in \bb N$. Hence, $ \alpha \circ \beta = \id $. On the other hand,
\begin{align*}
\hspace{-1.3cm}
\beta \circ \alpha \circ (\phi_{n,\infty}\otimes \psi_{k,\infty})
& = 
\beta \circ \alpha \circ (\phi_{n,\infty}\otimes \id\hspace{-0.05cm}\mbox{}_{\cl T}) \circ (\id\hspace{-0.05cm}\mbox{}_{n} \otimes \psi_{k,\infty})\\
& =\beta \circ \gamma_{n} \circ (\id\hspace{-0.05cm}\mbox{}_{n} \otimes \psi_{k,\infty})\\
    &= \beta\circ (\phi_{k} \otimes \psi_{k})_{\infty}\circ (\phi_{n,k} \otimes \id\hspace{-0.05cm}\mbox{}_{k})\\
    &= (\phi_{k,\infty}\otimes\id\hspace{-0.05cm}\mbox{}_{\cl T})
    \hspace{-0.1cm}\circ \hspace{-0.1cm}
    (\id\hspace{-0.05cm}\mbox{}_{k} \otimes \psi_{k,\infty}) 
    \hspace{-0.1cm}\circ \hspace{-0.1cm}
    (\phi_{n,k} \otimes \id\hspace{-0.05cm}\mbox{}_{k})\\
    &= \phi_{n,\infty}\otimes \psi_{k,\infty}
\end{align*}
for all $ n,k \in \bb N$ with $k\geq n$, showing that $\beta \circ \alpha = \id$. 
\end{proof}

For the formulation of the next theorem, we recall the notation (\ref{eq_cylinddd}) 
for the cylinders associated with elements $x\in X_n$ and $a\in A_n$; we 
employ similar notation for cylinders based on $y\in Y_n$ and $b\in B_n$.

\begin{theorem}  \label{th_ns} 
 Let $X = (X_n)_{n\in \bb{N}}$, $Y = (Y_n)_{n\in \bb{N}}$, 
$A = (A_n)_{n\in \bb{N}}$ and $B = (B_n)_{n\in \bb{N}}$ be inductive families 
of finite sets. 
\begin{enumerate}
    \item[(i)] If $\Gamma \in \cl C_{\rm ns}(X,Y,A,B)$ then there exists a state $s_{\Gamma}:\cl{S}_{X,A}\otimes_{\rm max}\cl{S}_{Y,B}\rightarrow \bb{C}$ such that 
     \begin{equation*}
s_{\Gamma}(\chi_{\tilde{x}\times\tilde{a}} \otimes \chi_{\tilde{y}\times\tilde{b}}) 
=
\lvert X_n\rvert \lvert Y_n\rvert \langle \chi_{\tilde{x}}\otimes \chi_{\tilde{y}},\Gamma(\chi_{\tilde{a}}\otimes \chi_{\tilde{b}})\rangle, 
    \end{equation*}
    for all $x\in X_n$, $y\in Y_n$, $a\in A_n$, $b\in B_n$, and all $n\in \bb{N}$.

\item[(ii)] If $s$ is a state on $\cl{S}_{X,A}\otimes_{\rm max} \cl{S}_{Y,B}$ 
    then there exists $\Gamma \in \cl C_{\rm ns}(X,Y,A,B)$
    such that $s = s_{\Gamma}$.
\end{enumerate}
\end{theorem}

\begin{proof}
    (i) Let $\Gamma$ be no-signalling and, for each $n\in\bb{N}$, set 
    $$\Gamma_n:=(\cl{E}_{X_n}^{\infty} \otimes\cl{E}_{Y_n}^{\infty})\circ\Gamma\circ(\iota_A^{(n)}\otimes \iota_B^{(n)}).$$
    By Proposition \ref{p_gammans}, 
    $\Gamma_n$ is no-signalling and hence,
    by Theorem \ref{th_orcha}, there exist states $s_n:\cl{S}_{X_n,A_n}\otimes_{\rm max}\cl{S}_{Y_n,B_n}\rightarrow\bb{C}$, $n\in\bb{N}$, such that 
\begin{equation}\label{eq_snne}
s_n(e_{x,a}\otimes e_{y,b})=\lvert X_n\rvert\lvert Y_n\rvert\langle \delta_x\otimes\delta_y,\Gamma_n(\delta_a\otimes\delta_b)\rangle,\ \ n\in\bb{N},
\end{equation}
    for every $(x,y,a,b)\in X_n\times Y_n\times A_n\times B_n$. By Theorem \ref{th_indcorresp} and Lemma  \ref{lem_ind_stat}, $$s_{n+1}\circ(\gamma_{X_n,A_n}\otimes\gamma_{Y_n,B_n})=s_n,$$
    and therefore by the universal property of inductive limits, Lemma \ref{l_comminl} 
    and \cite[Theorem 4.34]{mt} there exists a state $s:\cl{S}_{X,A}\otimes_{\rm max}\cl{S}_{Y,B}\rightarrow\bb{C}$ such that $s\circ(\gamma_{X,A}^{(n)}\otimes\gamma_{Y,B}^{(n)})=s_n$, $n\in \bb{N}$. 
Using (\ref{eq_snne}), it follows that, if $x\in X_n$, $y\in Y_n$, $a\in A_n$ and $b\in B_n$ then 
\begin{eqnarray*}
&& 
\hspace{-1cm}
s(\chi_{\tilde{x}\times\tilde{a}} \otimes \chi_{\tilde{y}\times\tilde{b}})
=
 s_n(e_{x,a}\otimes e_{y,b})\\
&& \hspace{0.5cm} =
\lvert X_n\rvert\lvert Y_n\rvert \langle\delta_x\otimes\delta_y,(\cl{E}_{X_n}^{\infty}\otimes\cl{E}_{Y_n}^{\infty})\circ \Gamma\circ(\iota_A^{(n)}\otimes\iota_B^{(n)})(\delta_a\otimes\delta_b)\rangle \\
&& \hspace{0.5cm} = \lvert X_n\rvert\lvert Y_n\rvert \langle(\iota_X^{(n)}\otimes\iota_Y^{(n)})(\delta_x\otimes\delta_y),\Gamma(\chi_{\tilde{a}}\otimes \chi_{\tilde{b}})\rangle\\
&& \hspace{0.5cm} = 
\lvert X_n\rvert \lvert Y_n\rvert \langle \chi_{\tilde{x}}\otimes \chi_{\tilde{y}},\Gamma(\chi_{\tilde{a}}\otimes \chi_{\tilde{b}})\rangle.
\end{eqnarray*}

(ii) Setting $s_n = s\circ(\gamma_{X,A}^{(n)}\otimes\gamma_{Y,B}^{(n)})$,
$n\in \bb{N}$, we have that $s_n$ is a state on the tensor product $\cl{S}_{X_n,A_n}\otimes_{\rm max}\cl{S}_{Y_n,B_n}$ and consequently it gives 
rise, via Theorem \ref{th_orcha}, to a no-signalling correlation $ \Gamma_{n}$ over $(X_n,Y_n,A_n,B_n)$, $n \in \bb N$. Note that 
$s_{n+1}\circ(\gamma_{X_n,A_n}\otimes\gamma_{Y_n,B_n})=s_n$, $n\in\bb{N}$, and thus, by Lemma \ref{lem_ind_stat}, the family  $(\Gamma_n)_{n\in \bb N}$ is inductive. By Theorem \ref{th_indcorresp},  there exists a (unique) unital completely positive map $ \Gamma : C(\Omega_A) \otimes C(\Omega_B) \to L^\infty(\Omega_X)\bar{\otimes}L^\infty(\Omega_Y)$ that satisfies the relations
\begin{align*}(\cl{E}_{X_n}^{\infty}\otimes\cl{E}_{Y_n}^{\infty})\circ\Gamma\circ(\iota_{A}^{(n)}\otimes \iota_{B}^{(n)}) = \Gamma_n, \ \ \ n\in \bb N;
\end{align*}
by Proposition \ref{p_gammans}, $\Gamma$ is no-signalling.
If $x\in X_n$, $y\in Y_n$, $a\in A_n$ and $b\in B_n$, then 
\begin{eqnarray*}
s_{\Gamma}(\chi_{\tilde{x}\times\tilde{a}} \otimes \chi_{\tilde{y}\times\tilde{b}}) 
& = & 
\lvert X_n\rvert \lvert Y_n\rvert \langle \chi_{\tilde{x}}\otimes \chi_{\tilde{y}},\Gamma(\chi_{\tilde{a}}\otimes \chi_{\tilde{b}})\rangle \\
& = & 
\lvert X_n\rvert \lvert Y_n\rvert \langle \delta_x \otimes \delta_y, \Gamma_n(\delta_a\otimes \delta_b) \rangle = s_n(e_{x,a}\otimes e_{y,b})\\
& = & 
(s\circ(\gamma_{X,A}^{(n)}\otimes\gamma_{Y,B}^{(n)}))(e_{x,a}\otimes e_{y,b})
= s(\chi_{\tilde{x}\times\tilde{a}} \otimes \chi_{\tilde{y}\times\tilde{b}})
\end{eqnarray*}
and, since the elements 
$\chi_{\tilde{x}\times\tilde{a}} \otimes \chi_{\tilde{y}\times\tilde{b}}$, when $n$ varies,
form a generating set 
for $\cl S_{X,A}\otimes_{\max}\cl S_{Y,B}$, we have that $s_\Gamma = s$.
\end{proof}


\section{The type hierarchy}\label{s_oct}

In this section, we consider other types of correlations over Cantor spaces, that lie within the 
class of all no-signalling correlations defined in Section \ref{s_rvs},
and obtain corresponding operator algebraic descriptions. 
We require some preparations; in the next subsection, we 
develop the bipartite versions of operator-valued channels from Subsection \ref{ss_genb}
that will be needed in the sequel.


\subsection{Bipartite operator-valued channels}\label{ss_bipovch}

Let $\frak{X}$, $\frak{Y}$, $\frak{S}$ and $\frak{T}$ be second countable compact Hausdorff spaces,
$\mu\in M(\frak X)$ and $\nu\in M(\frak Y)$ be probability measures, and $H$ be a Hilbert space.
Given $E\in\frak{C}_\mu(\frak S,\frak X;H)$ and $F\in\frak{C}_\nu(\frak T, \frak Y;H)$, 
and denoting by $\frak{f}$ the flip between the first and the second tensor terms in the three-leg 
expressions below, 
we let 
$$\phi_E:C(\frak S)\rightarrow L^\infty(\frak X,\mu)\bar{\otimes}L^\infty(\frak Y,\nu)\bar{\otimes}\cl{B}(H)$$
and 
$$\phi_F:C(\frak T)\rightarrow L^\infty(\frak X,\mu)\bar{\otimes}L^\infty(\frak Y,\nu)\bar{\otimes}\cl{B}(H)$$
be the maps, defined by setting
$$\phi_E(f) = \frak{f}(1_{\frak Y} \otimes \Phi_E(f))
\ \mbox{ and } \ 
\phi_F(g) = 1_{\frak X} \otimes \Phi_F(g).$$
We say that $E$ and $F$ form a commuting pair if
$\phi_E$ and $\phi_F$ have commuting ranges.

\begin{theorem}\label{cqc_prod}
Let $E\in\frak{C}_\mu(\frak S,\frak X;H)$ and $F\in\frak{C}_\nu(\frak T,\frak Y;H)$ 
be operator-valued channels that form a commuting pair. 
Then there exists a unique, 
up to $\sim_{\mu\times \nu}$-equivalence, channel
$E\cdot F\in\frak{C}_{\mu\times\nu}(\frak S\times \frak T,\frak X\times \frak Y;H)$ 
such that     
\begin{align}\label{prodform}
        (E\cdot F)(\alpha\times\beta|x,y)=E(\alpha|x)F(\beta|y)\hspace{0.3cm} \mu\times\nu \text{-a.e.},  \ \ 
        \alpha\in\frak{B}_{\frak S}, \beta\in\frak{B}_{\frak T}.
    \end{align}
\end{theorem}
\begin{proof}
Since $\phi_E$ and $\phi_F$ are unital completely positive maps with commuting ranges,  
by \cite[Theorem 12.8]{Pa}, there exists a unique unital completely positive map 
$$\phi_E\cdot\phi_F:C(\frak S)\otimes C(\frak T)\rightarrow L^\infty(\frak X\times \frak Y,\mu\times\nu)\bar\otimes\cl{B}(H),$$
such that 
$$(\phi_E\cdot\phi_F)(f\otimes g)=\phi_E(f)\phi_F(g), 
\ \ \ f\in C(\frak S), g\in C(\frak T).$$
Noting the canonical identification $C(\frak S\times \frak T)\cong C(\frak S)\otimes C(\frak T)$, 
we consider 
$\phi_E\cdot\phi_F$ as a map from 
$C(\frak S\times \frak T)$ into $L^\infty(\frak X\times \frak Y,\mu\times\nu)\bar\otimes\cl{B}(H)$. 
By Theorem \ref{mod_mu}, there exists a unique $E\cdot F\in\frak{C}_{\mu\times\nu}(\frak S\times \frak T,\frak X\times \frak Y;H)$ such that, for any $h\in C(\frak S\times \frak T)$ and $\xi,\eta\in H$, we have
$$\langle (\phi_E\cdot\phi_F)(h)(x,y)\xi,\eta\rangle=\int_{\frak S\times \frak T}h(a,b)d(E\cdot F)_{\xi,\eta}(a,b|x,y)\hspace{0.4cm} \mu\times\nu \text{-a.e.}$$
Applying approximation arguments similar to those in the proof of 
\cite[Lemma 3.1]{btt}, we obtain (\ref{prodform}).
\end{proof}

Theorem \ref{cqc_prod} easily yields the following corollary; 
the detailed proof is omitted.

\begin{corollary}
    Let $H$ and $K$ be Hilbert spaces, 
    $E\in\frak{C}_\mu(\frak S,\frak X;H)$ and $F\in\frak{C}_\nu(\frak T,\frak Y;$ $K)$. 
    Then there exists $E\otimes F\in\frak{C}_{\mu\times\nu}(\frak S\times \frak T, \frak X\times \frak Y;H\otimes K)$ such that, for all 
    $ \alpha\in\frak{B}_{\frak S}$ and $\beta\in\frak{B}_{\frak T}$, we have that 
    $$(E\otimes F)(\alpha\times\beta|x,y)=E(\alpha|x)\otimes F(\beta|y)\hspace{0.2cm} \mu\times\nu\text{-almost
    everywhere}.$$
\end{corollary}



\begin{remark}\label{r_appuni}
\rm 
We fix inductive families of sets $X=(X_n)_{n\in\bb{N}}$, $Y=(Y_n)_{n\in\bb{N}}$, $A=(A_n)_{n\in\bb{N}}$ and $B=(B_n)_{n\in\bb{N}}$ and let $\Omega_X$, $\Omega_Y$, $\Omega_A$ and $\Omega_B$ 
be their respective Cantor spaces. 
Given $E\in\frak{C}_{\mu_X}(\Omega_A,\Omega_X;H)$ and $n\in\bb{N}$, we denote by $E_n\in\frak{C}(A_n,X_n;H)$ the 
($\cl{B}(H)$-valued) information channel (from $X_n$ to $A_n$) 
for which the equality
\begin{equation}\label{rel_ovic_fin}
\Phi_{E_n} = \tilde{\cl{E}}^\infty_{X_n}\circ\Phi_E\circ\iota^{(n)}_{A,\infty}
 \end{equation} 
is satisfied, $n\in \bb{N}$.
Set
$$\Psi_n = \tilde{\iota}_{X,\infty}^{(n)}\circ\Phi_{E_n}\circ\cl{E}^\infty_{A_n}|_{C(\Omega_A)},\ \ n\in\bb{N};$$
thus, $(\Psi_n)_{n\in\bb{N}}\subseteq\operatorname{UCP}(C(\Omega_A),L^\infty(\Omega_X)\bar\otimes\cl{B}(H))$. 
We have that $\lim_{n\to\infty} \Psi_n = \Phi_E$ in the BW topology.
Indeed, note that 
$$\lim_{n\to\infty} \tilde{\iota}_{X,\infty}^{(n)}\circ\tilde{\cl{E}}^\infty_{X_{n}}=\id\hspace{-0.05cm}\mbox{}_{L^\infty(\Omega_X)\bar\otimes\cl{B}(H)}$$ 
in the BW topology, and 
$$\lim_{n\to\infty}\iota^{(n)}_{A,\infty}\circ
\cl{E}^\infty_{A_{n}}|_{C(\Omega_A)}
= \id\hspace{-0.05cm}\mbox{}_{C(\Omega_A)}$$ 
in the point-norm topology.
Fix $\omega \in \cl B(H)_*$ with $\|\omega\|_1\leq 1$, $\epsilon > 0$ and $f\in C(\Omega_A)$, and 
let $N\in \bb{N}$ be such that 
$$\left\|(\iota^{(n)}_{A,\infty}\circ\cl{E}^\infty_{A_{n}})(f) - f\right\| < \frac{\epsilon}{2}$$
and 
$$\left|\langle \omega,(\tilde{\iota}_{X,\infty}^{(n)}\circ\tilde{\cl{E}}^\infty_{X_{n}}\circ \Phi_E)(f)\rangle - \langle \omega, \Phi_E(f)\rangle\right| < \frac{\epsilon}{2}$$
whenever $n \geq N$. 
Then 
\begin{eqnarray*}
&& \left|\langle \omega, \Psi_n(f)\rangle \hspace{-0.05cm} - \hspace{-0.05cm}
\langle \omega, \Phi_E(f)\rangle\right|\\
&& \leq 
\left|\langle \omega, \Psi_n(f)\rangle \hspace{-0.05cm} - \hspace{-0.05cm}
\langle \omega, (\tilde{\iota}_{X,\infty}^{(n)}\hspace{-0.03cm}\circ\hspace{-0.03cm}\tilde{\cl{E}}^\infty_{X_{n}}\hspace{-0.03cm}\circ \hspace{-0.03cm}\Phi_E)(f)\rangle\right|\\
&&\hspace{1cm} +  
\left|\langle \omega, (\tilde{\iota}_{X,\infty}^{(n)}\hspace{-0.03cm}\circ\hspace{-0.03cm}\tilde{\cl{E}}^\infty_{X_{n}}\hspace{-0.03cm}\circ \hspace{-0.03cm}\Phi_E)(f)\rangle
\hspace{-0.05cm} - \hspace{-0.05cm} \langle \omega, \Phi_E(f)\rangle \right|\\
&& =  
\left|\langle \omega, 
(\tilde{\iota}_{X,\infty}^{(n)}\circ
\tilde{\cl{E}}^\infty_{X_n}\circ\Phi_E\circ\iota^{(n)}_{A,\infty}
\circ\cl{E}^\infty_{A_n})(f)
\rangle \hspace{-0.05cm} - \hspace{-0.05cm}
\langle \omega, (\tilde{\iota}_{X,\infty}^{(n)}\hspace{-0.03cm}\circ\hspace{-0.03cm}\tilde{\cl{E}}^\infty_{X_{n}}\hspace{-0.03cm}\circ \hspace{-0.03cm}\Phi_E)(f)\rangle\right|\\
&& \hspace{1cm} +  
\left|\langle \omega, (\tilde{\iota}_{X,\infty}^{(n)}\hspace{-0.03cm}\circ\hspace{-0.03cm}\tilde{\cl{E}}^\infty_{X_{n}}\hspace{-0.03cm}\circ \hspace{-0.03cm}\Phi_E)(f)\rangle
\hspace{-0.05cm} - \hspace{-0.05cm} \langle \omega, \Phi_E(f)\rangle \right|
\leq \epsilon
\end{eqnarray*}
for every $n\geq N$. 
\end{remark}

\begin{remark}\label{r_ovic_fin_c} 
   \rm The channels
   $E\in\frak{C}_{\mu_X}(\Omega_A,\Omega_X;H)$ and $F\in\frak{C}_{\mu_Y}(\Omega_B,\Omega_Y;H)$
    form a commuting pair if and only if the (finite) channels $E_n$ and $F_m$ obtained via equation (\ref{rel_ovic_fin}) from $E$ and $F$, respectively, also do so, for all $n$ and $m$. 
Indeed, assume that $(E,F)$ is a commuting pair, then the channels 
$ \Phi_{E}\circ \iota^{(n)}_{A}$, $ \Phi_{F}\circ \iota^{(m)}_{B}$ 
form a commuting pair and we can write
$$\phi_{E_{n}}(f) = (\cl E^{\infty}_{X_{n}}\otimes 1_{Y_m} 
\otimes \id\hspace{-0.05cm}\mbox{}_{\cl B(H)}) ( \Phi_{E}( \iota^{(n)}_{A}(f))_{1,3} )$$
and 
$$\phi_{F_{m}}(g) = (1_{X_n}\otimes\cl E^{\infty}_{Y_{m}} \otimes \id\hspace{-0.05cm}\mbox{}_{\cl B(H)})(\Phi_{F}(\iota^{(m)}_{B}(g))_{2,3}),$$ 
from which the statement follows immediately
(in the displayed equations, we have used standard leg notation). 

On the other hand, assume that $ \Phi_{E_n}$ and $ \Phi_{F_m}$ form a commuting pair for every $n, m\in \bb N$. Then the maps $ \tilde{\iota}^{(n)}_{X,\infty}\circ \Phi_{E_n} \circ \cl E_{A_n}|_{C(\Omega_A)}$ and $ \tilde{\iota}^{(m)}_{Y,\infty}\circ \Phi_{F_m} \circ \cl E_{B_n}|_{C(\Omega_B)}$ also form a commuting pair. By Remark \ref{r_appuni} the latter unital completely positive maps converge to $ \Phi_{E}$ and $ \Phi_{F}$ respectively in the BW topology. By taking iterated limits we conclude that $\Phi_{E}$ and $\Phi_{F}$ have commuting ranges.
\end{remark}

\begin{lemma} \label{lem:commu_corresp} 
Let $\Gamma^1:\cl S_{X,A}\to\cl B(H)$ and $\Gamma^2:\cl S_{Y,B}\to\cl B(H)$ be unital completely positive maps, and let
\[
\Phi^1: C(\Omega_A)\to L^{\infty}(\Omega_X)\bar{\otimes}\cl B(H),\qquad
\Phi^2: C(\Omega_B)\to L^{\infty}(\Omega_Y)\bar{\otimes}\cl B(H)
\]
be their corresponding channels ariding via Theorem~\ref{c_statSXA}.
Then $ \Gamma^1 $ and $ \Gamma^2$ form a commuting pair if and only if $ \Phi^{1}$ and $ \Phi^{2}$ do so.

\begin{proof}
We work at finite levels and pass to the limit. For $n,m$, let
\[
\Gamma^1_n:\cl S_{X_n,A_n}\to\cl B(H),\quad
\Gamma^2_m:\cl S_{Y_m,B_m}\to\cl B(H),
\]
and the corresponding
\[
\Phi^1_n:\cl D_{A_n}\to\cl D_{X_n}\otimes\cl B(H),\qquad
\Phi^2_m:\cl D_{B_m}\to\cl D_{Y_m}\otimes\cl B(H),
\]
with
\[
\Phi^1_n(\delta_a)=\sum_{x\in X_n}\delta_x\otimes\Gamma^1_n(e_{x,a}),\qquad
\Phi^2_m(\delta_b)=\sum_{y\in Y_m}\delta_y\otimes\Gamma^2_m(e_{y,b}).
\]
We compute
\begin{equation}\label{eq:comm-sum}
\big[(\Phi^{1}_{n}(\delta_a))_{1,3},\,(\Phi^{2}_{m}(\delta_b))_{2,3}\big]
=\sum_{x\in X_n}\sum_{y\in Y_m}\delta_x\otimes\delta_y\otimes
\big[\Gamma^1_n(e_{x,a}),\,\Gamma^2_m(e_{y,b})\big].
\end{equation}

If $ \Gamma^1$ and $\Gamma^2$ have commuting ranges then $[\Gamma^1_n(e_{x,a}),\Gamma^2_m(e_{y,b})]=0$ for all $x,y,a,b$, and
\eqref{eq:comm-sum} vanishes on generators; by linearity $ \Phi_{n}^{1}$ and $ \Phi^2_{m}$ have commuting ranges for all $ n,m\in \bb N$. By Remark \ref{r_ovic_fin_c} $ \Phi^1$ and $ \Phi^2$ have commuting ranges.

 If $\Phi^{1}$ and $ \Phi^{2}$ form a commuting pair, by Remark \ref{r_ovic_fin_c}, $[(\Phi^{1}_{n}(\delta_a))_{1,3},\,(\Phi^{2}_{m}(\delta_b))_{2,3}]=0$, apply the slice
$L_{\delta_x\otimes\delta_y}$ to \eqref{eq:comm-sum} to obtain
$[\Gamma^1_n(e_{x,a}),\Gamma^2_m(e_{y,b})]=0$ for all $x,y,a,b$; by linearity this yields
$[\Gamma^1_n(u),\Gamma^2_m(v)]=0$ for all $u,v$ and hence by density $\Gamma^1$ and $\Gamma^2$ have commuting ranges.

\end{proof}
\end{lemma}

\subsection{An ultraproduct channel construction} \label{subseq:ultraprod}
We collect some details about ultraproducts that we will need in the sequel, 
and refer the reader to \cite{AndoHaagerup} for further background.
Fix a free ultrafilter $\omega$ on $\mathbb N$.
For a sequence $(\cl X_n)$ of Banach spaces, set
$\ell^\infty(\cl X_n)=\{(x_n):\sup_n\|x_n\|<\infty\}$ and
$\mathcal N_\omega=\{(x_n):\lim_\omega\|x_n\|=0\}$.
Then the space $(\cl X_n)^\omega:=\ell^\infty(\cl X_n)/\mathcal N_\omega$, endowed with the norm
$\|[x_n]\|=\lim_\omega\|x_n\|$ (where $[x_n]$ denotes the coset 
containing the sequence $(x_n)_{n\in \bb{N}}$), 
is a Banach space, called the \textit{Banach space ultraproduct} of $(\cl X_n)$. 
 For a sequence $(H_n)_{n\in \bb{N}}$ of Hilbert spaces, the
Banach space ultraproduct $H^\omega$ is a Hilbert space when endowed with the inner product
$\langle [x_n],[y_n]\rangle=\lim_\omega\langle x_n,y_n\rangle$. 
For a sequence $(\cl M_n)_{n\in \bb{N}}$ of C*-algebras, the ultraproduct $ (\cl M_n)^{\omega}$ is again a C*-algebra when equipped with the
pointwise multiplication and involution of sequences.  
If $(T_n)_{n\in \bb{N}}$ is a uniformly bounded sequence, 
where $T_n\in\mathcal B(H_n)$, the formula
$\pi_\omega([T_n])[x_n]=[T_nx_n]$ defines an isometric *-homomorphism
$\pi_\omega:(\mathcal B(H_n))^\omega\to \mathcal B(H^\omega)$. For simplicity we will write $ [T_{n}]_{\omega}$ for $ \pi_{\omega}([T_n])$.


Let $\cl M_n\subseteq\mathcal B(H_n)$ be von Neumann algebras, $n \in \bb N$.
Write $\pi_\omega:(\cl M_n)^\omega\to \mathcal B(H^\omega)$ for the canonical representation induced on the
ultraproduct Hilbert space $H^\omega$. The \emph{abstract ultraproduct} {\cite[Definition 3.5]{AndoHaagerup}}
\[ 
\prod^\omega(\cl M_n,H_n) := \overline{\pi_\omega((\cl M_n)^\omega)}^{\rm SOT} \subseteq \mathcal B(H^\omega)
\]
is the strong-operator closure of $\pi_\omega((\cl M_n)^\omega)$.
In particular, when $\cl M_n=\mathcal B(H_n)$ one has 
\( 
\prod^\omega(\mathcal B(H_n),H_n)=\mathcal B(H^\omega)
\) (see \cite[Lemma 3.4]{AndoHaagerup}).

\begin{lemma} \label{lem:ultraprod_ucp} 
Let $\cl S$ be an operator system, $H_n$ be a Hilbert space, $n\in \bb{N}$, and $\Phi_n:\cl S\to\cl B(H_n)$ unital completely positive  maps for every $n\in\bb N$. Then the map
\[
\Phi^\omega:\cl S\longrightarrow \cl B(H^\omega),
\ \mbox{ given by } \ 
\Phi^\omega(s):=\pi_\omega\big([\Phi_n(s)]\big),
\]
is unital and completely positive.

Moreover, if $\cl T$ is another operator system and $\Psi_n:\cl T\to\cl B(H_n)$ are unital completely positive maps such $(\Phi_n,\Psi_n)$ is a commuting pair for every $n$, then $(\Phi^\omega,\Psi^\omega)$ is a commuting pair.
\end{lemma}

\begin{proof}
The map from $\cl S$ into $\ell^\infty(\cl B(H_n))$, 
sending an element $s\in \cl S$ to the sequence 
$(\Phi_n(s))$, is unital and completely positive. The quotient map $\ell^\infty(\cl B(H_n))\to (\cl B(H_n))^\omega$ is a unital $*$-homomorphism because $\mathcal I_\omega=\{(x_n):\lim_\omega\|x_n\|=0\}$ is a closed two-sided $*$-ideal. Finally, $\pi_\omega:(\cl B(H_n))^\omega\to\cl B(H^\omega)$ is a unital $*$-homomorphism. Thus the composition $s\mapsto(\Phi_n(s))\mapsto [\Phi_n(s)]\mapsto \pi_\omega([\Phi_n(s)])$ is unital and completely positive, proving the first claim.

To prove commutation, fix $s\in\cl S$, $t\in\cl T$ and $\xi=[\xi_n]\in H^\omega$. Then
\[
\Phi^\omega(s)\,\Psi^\omega(t)\,\xi
=\big[\Phi_n(s)\Psi_n(t)\,\xi_n\big]
=\big[\Psi_n(t)\Phi_n(s)\,\xi_n\big]
=\Psi^\omega(t)\,\Phi^\omega(s)\,\xi,
\]
since $\Phi_n(s)$ and $\Psi_n(t)$ commute for each $n$. Hence $[\Phi^\omega(s),\Psi^\omega(t)]=0$ in $\cl B(H^\omega)$, as required.
\end{proof}


\subsection{Definitions and  characterisations}\label{ss_tdeffc}

Motivated by the hierarchy of types in the case of correlations over finite input and output sets, we now adapt the definitions of 
no-signalling correlation types from \cite{btt} to the Cantor setup.

\begin{definition}\label{d_diffty}
\rm 
Let $X=(X_n)_{n\in\bb{N}}$, $Y=(Y_n)_{n\in\bb{N}}$, $A=(A_n)_{n\in\bb{N}}$ and $B=(B_n)_{n\in\bb{N}}$ be inductive families of sets.
A unital completely positive map $\Gamma : C(\Omega_A){\otimes} C(\Omega_B) \to 
L^\infty(\Omega_X)\bar{\otimes} L^\infty(\Omega_Y)$ is called a 
\begin{itemize}
\item[(i)] \emph{local correlation} if it is a 
finite convex combination of maps of the form 
$\Phi\otimes\Psi$, where 
$\Phi : C(\Omega_A) \to L^\infty(\Omega_X)$ and 
$\Psi : C(\Omega_B) \to L^\infty(\Omega_Y)$ are unital 
completely positive maps;

    \item[(ii)] \emph{quantum spatial correlation} if there exist separable Hilbert spaces $H$ and $K$, 
    a  unit vector $\xi\in H\otimes K$, and operator-valued channels 
$E\in\frak{C}_{\mu_X}(\Omega_{A},\Omega_{X};H)$ and $F\in\frak{C}_{\mu_Y}(\Omega_{B},\Omega_{Y};K)$, 
such that  
\begin{equation}\label{eq_forminnn}
\langle g, \Gamma(h)\rangle=\langle g\otimes \xi\xi^* ,\Phi_{E\otimes F}(h)\rangle
\end{equation}
whenever    $h\in C(\Omega_A){\otimes}C(\Omega_B)$ and 
$g\in L^1(\Omega_X)\hat{\otimes}L^1(\Omega_Y)$.
   
\item[(iii)] \emph{quantum approximate correlation} if $\Gamma\in\overline{\cl{C}_{\rm qs}}^{\rm BW}$;

\item[(iv)]\emph{quantum commuting correlation} if there exist 
a separable Hilbert space $H$, 
a unit vector $\xi\in H$, and operator-valued channels
$E\in\frak{C}_{\mu_X}(\Omega_A,\Omega_X;H)$  and $F\in\frak{C}_{\mu_Y}(\Omega_B,\Omega_Y;H)$ that form a commuting pair, such that 
\begin{equation}\label{eq_eqforqcde}
\langle g,\Gamma(h)\rangle=\langle g\otimes \xi\xi^* ,\Phi_{E\cdot F}(h)\rangle
\end{equation}
whenever $h\in C(\Omega_A){\otimes}C(\Omega_B)$ and 
$g\in L^1(\Omega_X)\hat{\otimes}L^1(\Omega_Y)$.
\end{itemize}
\end{definition}

In the context of Definition \ref{d_diffty} (iv), 
we will say that the triple $(H,E,F,\xi)$ is a realisation of the correlation $\Gamma$.
We denote by $\cl{C}_{\rm loc}(X,Y,A,B)$ (resp. $\cl{C}_{\rm qs}(X,Y,A,B)$, $\cl C_{\rm qa}(X,Y,A,B)$, $\cl C_{\rm qc}(X,Y,A,B)$)  
the set of all local, (resp. quantum spatial, quantum approximate, quantum commuting) no-signalling correlations over $(X,Y,A,B)$, and simply use $\cl C_{\rm t}$ when 
the quadruple $(X,Y,A,B)$ is clear from the context.


\begin{theorem}\label{th_n_qc}
Let $\Gamma:C(\Omega_A)\otimes C(\Omega_B)\rightarrow L^\infty(\Omega_X)\bar\otimes L^\infty(\Omega_Y)$ be a unital completely positive map and 
$(\Gamma_n)_{n\in \bb N}$ be its associated inductive family of maps.
The following are equivalent: 
   \begin{itemize}
       \item[(i)] $\Gamma\in\cl{C}_{\rm qc}(X,Y,A,B)$;
       \item[(ii)] $\Gamma_n\in\cl{C}_{\rm qc}(X_n,Y_n,A_n,B_n)$ for every $n\in\bb{N}$.  
   \end{itemize}
\end{theorem}

\begin{proof}
We recall that $\tilde{\cl E}^{\infty}_{X_n} : L^{\infty}(\Omega_X)\bar\otimes\cl B(H)\to \cl D_{X_n}\otimes\cl B(H)$ is the canonical expectation.

\smallskip

(i)$\Rightarrow$(ii) Let $\Gamma\in\cl{C}_{\rm qc}$, and 
let $H$ be a Hilbert space, $\xi\in H$ be a unit vector, and 
$E\in\frak{C}_{\mu_X}(\Omega_A,\Omega_X;H)$ and $F\in\frak{C}_{\mu_Y}(\Omega_B,\Omega_Y;H)$ 
be channels forming a commuting pair, such that 
(\ref{eq_eqforqcde}) is satisfied.
Further, let 
$E_n\in\frak{C}(A_n,X_n;H)$ and $F_n\in\frak{C}(B_n,Y_n;H)$ be the channels such that 
$$\Phi_{E_n}=\tilde{\cl{E}}_{X_n}^\infty\circ\Phi_E\circ\iota_{A}^{(n)}
\ \mbox{ and } \ 
\Phi_{F_n}=\tilde{\cl{E}}_{Y_n}^\infty\circ\Phi_F\circ\iota_{B}^{(n)}, \ \ \ n\in\bb{N}.$$ 
By Remark \ref{r_ovic_fin_c}, 
$(E_n,F_n)$ is a commuting pair for every $n\in\bb{N}$. 

We show that $\Gamma_n$ is a quantum commuting correlation with 
realisation $(H,E_n,F_n,\xi)$. Indeed, 
if $x\in X_n$, $y\in Y_n$, $a\in A_n$ and $b\in B_n$ then 
\begin{eqnarray*}
    && \langle \delta_x\otimes\delta_y,\Gamma_n(\delta_a\otimes\delta_b)\rangle=\\ && =\langle \delta_x\otimes\delta_y,(\cl{E}_{X_n}^{\infty}\otimes\cl{E}_{Y_n}^{\infty})\circ\Gamma\circ(\iota_A^{(n)}\otimes\iota_B^{(n)})(\delta_a\otimes\delta_b)\rangle\\
    && =\langle(\iota_{X,1}^{(n)}\otimes\iota_{Y,1}^{(n)})(\delta_x\otimes\delta_y),\Gamma\circ(\iota_A^{(n)}\otimes\iota_B^{(n)})(\delta_a\otimes\delta_b)\rangle\\
    && = \langle (\iota_{X,1}^{(n)}\otimes\iota_{Y,1}^{(n)})(\delta_x\otimes\delta_y)\otimes\xi\xi^*,\Phi_{E\cdot F}\circ(\iota_A^{(n)}\otimes\iota_B^{(n)})(\delta_a\otimes\delta_b)\rangle\\
    && =\langle(\delta_x\otimes\delta_y)\otimes\xi\xi^*, 
(\widetilde{\cl{E}_{X_n}^\infty\otimes\cl{E}_{Y_n}^\infty})\circ\Phi_{E\cdot F}\circ(\iota_{A}^{(n)}\otimes\iota_{B}^{(n)})(\delta_a\otimes\delta_b)\rangle\\
    && =\langle(\delta_x\otimes\delta_y)\otimes\xi\xi^*,\Phi_{E_n\cdot F_n}(\delta_a\otimes\delta_b)\rangle,
\end{eqnarray*}
where the last equality follows from the fact that
\begin{eqnarray*}
&&(\widetilde{\cl{E}_{X_n}^\infty\otimes\cl{E}_{Y_n}^\infty})\circ\Phi_{E\cdot F}\circ(\iota_{A}^{(n)}\otimes\iota_{B}^{(n)})(\delta_a\otimes\delta_b)\\&&=(\widetilde{\cl{E}_{X_n}^\infty\otimes\cl{E}_{Y_n}^\infty})((\Phi_E(\iota_{A}^{(n)}(\delta_a)))_{1,3}(\Phi_F(\iota_{B}^{(n)}(\delta_b)))_{2,3})\\
&&=(\tilde{\cl{E}}_{X_n}^\infty(\Phi_E(\iota_{A}^{(n)}(\delta_a)))_{1,3}(\tilde{\cl{E}}_{Y_n}^\infty(\Phi_F(\iota_{B}^{(n)}(\delta_b)))_{2,3}.
\end{eqnarray*}

\smallskip

(ii)$\Rightarrow$(i)  
By assumption, for each $n\in\Bbb N$ there exist a Hilbert space $H_n$, a unit vector
$\xi_n\in H_n$, and channels
\[
E_n\in \frak C(A_n,X_n;H_n) \ \mbox{ and } \ F_n\in \frak C(B_n,Y_n;H_n),
\]
forming a commuting pair and realising $\Gamma_n$:
\begin{equation}\label{eq:realise-n}
\langle g,\Gamma_n(h)\rangle
=\langle g\otimes \xi_n\xi_n^*,\,\Phi_{E_n\cdot F_n}(h)\rangle,
\quad g\in \cl D_{X_n}\otimes \cl D_{Y_n},\ h\in \cl D_{A_n}\otimes \cl D_{B_n}.
\end{equation}
In addition, we have the inductivity relations
\begin{equation}\label{eq:Gamma-ind}
\Gamma_n
=(\cl E_{X_{n+1},X_n}\otimes \cl E_{Y_{n+1},Y_n})\circ
\Gamma_{n+1}\circ(\iota_{A_n,A_{n+1}}\otimes \iota_{B_n,B_{n+1}}), \ \ \ n\in \bb{N}.
\end{equation}

Combining equations (\ref{eq:realise-n}) and (\ref{eq:Gamma-ind}), for all $g\in \cl D_{X_n}\otimes \cl D_{Y_n}$,  $\ h\in \cl D_{A_n}\otimes \cl D_{B_n}$ we have 
\begin{eqnarray*}
    && \big\langle g\otimes \xi_n\xi_n^*,\, \Phi_{E_n\cdot F_n}(f)\big\rangle =\langle g,\Gamma_n(h)\rangle \\
    && =\langle g,(\cl E_{X_{n+1},X_n}\otimes \cl E_{Y_{n+1},Y_n})\circ
\Gamma_{n+1}\circ(\iota_{A_n,A_{n+1}}\otimes \iota_{B_n,B_{n+1}})(h) \rangle\\
&& =\langle ( \iota_{X_{n+1},X_n}\otimes \iota_{Y_{n+1},Y_n})( g),
\Gamma_{n+1}\circ(\iota_{A_n,A_{n+1}}\otimes \iota_{B_n,B_{n+1}})(h) \rangle\\
&& =\langle ( \iota_{X_{n+1},X_n}\otimes \iota_{Y_{n+1},Y_n})( g) \otimes \xi_{n+1}\xi_{n+1}^{*},
\Phi_{E_{n+1}\cdot F_{n+1}}\circ(\iota_{A_n,A_{n+1}}\otimes \iota_{B_n,B_{n+1}})(h) \rangle.
\end{eqnarray*}
Thus, for all $g\in \cl D_{X_n}\otimes \cl D_{Y_n}$,  $\ h\in \cl D_{A_n}\otimes \cl D_{B_n}$, we obtain
\begin{equation}\label{eq:Phi-product-ind}
\begin{aligned}
\big\langle g\otimes \xi_n\xi_n^*,\, \Phi_{E_n\cdot F_n}(h)\big\rangle
&=\Big\langle g\otimes \xi_{n+1}\xi_{n+1}^*,\,
(\widetilde{\mathcal E_{X_{n+1},X_n}\otimes \mathcal E_{Y_{n+1},Y_n}})\\
&\hspace{3.2cm}\circ
\Phi_{E_{n+1}\cdot F_{n+1}}\big((\iota_{A_n,A_{n+1}}\otimes \iota_{B_n,B_{n+1}})(h)\big)
\Big\rangle.
\end{aligned}
\end{equation}

 Define the channels
\[
\Phi'_{E_n}:=\tilde{\iota}_{X,\infty}^{(n)}\circ\Phi_{E_n}\circ \cl E_{A_n}^\infty|_{C(\Omega_A)}:
C(\Omega_A)\to L^\infty(\Omega_X)\bar\otimes \cl B(H_n),
\]
\[
\Phi'_{F_n}:=\tilde{\iota}_{Y,\infty}^{(n)}\circ \Phi_{F_n} \circ \cl E_{B_n}^\infty|_{C(\Omega_B)}:
C(\Omega_B)\to L^\infty(\Omega_Y)\bar\otimes \cl B(H_n),
\]
which form a commuting pair at each level $n$.
By Theorem~\ref{c_statSXA} (with $H=H_n$), there exist unique unital completely positive maps
\[
\Theta_{n}:\mathcal S_{X,A}\to \cl B(H_n) \ \mbox{ and } \ 
\Lambda_{n}:\mathcal S_{Y,B}\to \cl B(H_n),
\]
such that
\begin{equation}\label{eq:slices-level}
\begin{aligned}
   & |X_m|\,L_{\iota_{X,1}^{(m)}(\delta_x)}\!\big(\Phi'_{E_n}(\iota_A^{(m)}(\delta_a))\big)
=\Theta_n\big(\gamma^{(m)}_{X,A}(e_{x,a})\big),\\
&|Y_m|\,L_{\iota_{Y,1}^{(m)}(\delta_y)}\!\big(\Phi'_{F_n}(\iota_B^{(m)}(\delta_b))\big)
=\Lambda_n\big(\gamma^{(m)}_{Y,B}(e_{y,b})\big)
\end{aligned}
\end{equation}
for all $m\in\Bbb N$, $x\in X_m$, $a\in A_m$, $y\in Y_m$, $b\in B_m$ (see (\ref{eq_dualslice})).
In particular, by Lemma \ref{lem:commu_corresp} we obtain that $(\Theta_n, \Lambda_n)$ a commuting pair, $n\in \bb N$.

Fix a free ultrafilter $\omega$ on $\mathbb N$ and form the Hilbert ultrapower $H^\omega$ and the abstract ultraproduct
$\prod^\omega\big(\mathcal B(H_n),H_n\big)\ =\ \mathcal B(H^\omega)$
 (see Subsection \ref{subseq:ultraprod}).
Let $\xi=[\xi_n]$; thus, $\xi$ is a unit vector in $H^\omega$.
By Lemma \ref{lem:ultraprod_ucp}, we have unital completely positive maps
$\Theta^{\omega} : \cl S_{X,A}\to \cl B(H^\omega)$ and 
$\Lambda^{\omega} : \cl S_{Y,B}\to \cl B(H^\omega)$, given by 
\[
\Theta^{\omega}(s):=[\Theta_n(s)]_\omega \ \mbox{ and } \ 
\Lambda^{\omega}(t):=[\Lambda_n(t)]_\omega,
\]
that form a commuting pair. 
Apply Theorem~\ref{c_statSXA} again (with $H=H^\omega$) to obtain unital completely positive maps
\[
\Phi: C(\Omega_A)\to L^\infty(\Omega_X)\bar\otimes \cl B(H^\omega) \ \mbox{ and } \ 
\Psi: C(\Omega_B)\to L^\infty(\Omega_Y)\bar\otimes \cl B(H^\omega),
\]
satisfying 
\begin{equation}\label{eq:slices-omega}
\begin{aligned}&|X_k|\,L_{\iota_{X,1}^{(k)}(\delta_x)}\!\big(\Phi(\iota_A^{(k)}(\delta_a))\big)
=\Theta^{\omega}\big(\gamma^{(k)}_{X,A}(e_{x,a})\big),\\
&|Y_k|\,L_{\iota_{Y,1}^{(k)}(\delta_y)}\!\big(\Psi(\iota_B^{(k)}(\delta_b))\big)
=\Lambda^{\omega}\big(\gamma^{(k)}_{Y,B}(e_{y,b})\big)
\end{aligned}
\end{equation}
for all $x\in X_k$, $a\in A_k$, $y\in Y_k$, $b\in B_k$, $k\in \bb{N}$.
By Lemma \ref{lem:commu_corresp} again, the pair $(\Phi,\Psi)$ is commuting.

Now, by Theorem \ref{mod_mu}, there exist operator-valued channels
$E\in\frak{C}_{\mu_X}(\Omega_A,\Omega_X;H)$  and $F\in\frak{C}_{\mu_Y}(\Omega_B,\Omega_Y;H)$ such that $ \Phi= \Phi_{E}$ and $ \Psi= \Psi_{F}$, and since $E$ and $F$ form a commuting pair, by Theorem \ref{cqc_prod} we obtain the channel $ E \cdot F$, giving rise to the unital completely positive map 
\[
\Phi_{E\cdot F}: C(\Omega_A)\otimes C(\Omega_B)\to
L^\infty(\Omega_X)\bar\otimes L^\infty(\Omega_Y)\bar\otimes \cl B(H^\omega).
\]
Fix $n\in\Bbb N$, $a\in A_n$, $b\in B_n$, and take any $k\ge n$, $x\in X_k$, $y\in Y_k$.
By inductivity and \eqref{eq:realise-n},
\begin{multline}\label{eq:first-id}
\Big\langle(\iota_{X,1}^{(k)}\otimes\iota_{Y,1}^{(k)})(\delta_x\otimes\delta_y),\
\Gamma\big((\iota_A^{(n)}\otimes\iota_B^{(n)})(\delta_a\otimes\delta_b)\big)\Big\rangle\\
=\Big\langle\delta_x\otimes\delta_y,\
\Gamma_k\big((\iota_{A_n,A_k}\otimes\iota_{B_n,B_k})(\delta_a\otimes\delta_b)\big)\Big\rangle
\\
= \Big\langle \delta_x\otimes\delta_y\otimes\xi_k\xi_k^*,\
\Phi_{E_k\cdot F_k}\big((\iota_{A_n,A_k}\otimes\iota_{B_n,B_k})(\delta_a\otimes\delta_b)\big)\Big\rangle.
\end{multline}

Using (\ref{eq:slices-omega}), we have 

\begin{eqnarray*}
& &\Big\langle (\iota_{X,1}^{(k)}\otimes \iota_{Y,1}^{(k)})(\delta_x\otimes \delta_y)\otimes \xi\xi^*,\,
\Phi_{E\cdot F}(\iota_A^{(n)}(\delta_a) \otimes \iota_B^{(n)}(\delta_b)) \Big\rangle\\
&=& \Big\langle
L_{\iota_{X,1}^{(k)}(\delta_x)}(\Phi(\iota_A^{(n)}(\delta_a)))\,
L_{\iota_{Y,1}^{(k)}(\delta_y)}(\Psi(\iota_B^{(n)}(\delta_b)))\,\xi, \,\xi \Big\rangle\\[0.5ex]
&=& \Big\langle
L_{\iota_{X,1}^{(k)}(\delta_x)}(\Phi(\iota_A^{(k)} \circ\iota_{A_{n},A_{k}}(\delta_a)))\,
L_{\iota_{Y,1}^{(k)}(\delta_y)}(\Psi(\iota_B^{(k)}\circ \iota_{B_{n},B_{k}}(\delta_b)))\,\xi,\, \xi\, \Big\rangle\\[0.5ex]
&=& \frac{1}{|X_k|\,|Y_k|}\,
\Big\langle 
\Theta^{\omega}
\left(\gamma^{(k)}_{X,A}\left(\sum_{\vec\lambda}\ e_{x,(a,\vec\lambda)}\right)\right)\,
\Lambda^{\omega}\left(\gamma^{(k)}_{Y,B}\left(\sum_{\vec\mu}\ e_{y,(b,\vec\mu)} \right)\right)\,\xi,\, \xi\,\Big\rangle, 
\end{eqnarray*}
where the summations are over $\vec\lambda=(\lambda_{n+1},\ldots,\lambda_k)\in [d_{n+1}^A]\times\ldots\times[d_k^A]$ and $\vec\mu=(\mu_{n+1},\ldots,\mu_k)\in [d_{n+1}^B]\times\ldots\times[d_k^B]$.
Taking the limit along the ultrafilter and using (\ref{eq:slices-level}), 
we obtain 
\begin{eqnarray*}
& &\Big\langle (\iota_{X,1}^{(k)}\otimes \iota_{Y,1}^{(k)})(\delta_x\otimes \delta_y)\otimes \xi\xi^*,\,
\Phi_{E\cdot F}(\iota_A^{(n)}(\delta_a) \otimes \iota_B^{(n)}(\delta_b)) \Big\rangle\\
&=& \frac{1}{|X_k|\,|Y_k|}\,
\Big\langle 
\Theta^{\omega}(\gamma^{(k)}_{X,A}(\sum_{\vec\lambda}\ e_{x,(a,\vec\lambda)}))\,
\Lambda^{\omega}(\gamma^{(k)}_{Y,B}(\sum_{\vec\mu}\ e_{y,(b,\vec\mu)} ))\,\xi,\, \xi\,\Big\rangle\\[0.5ex]
&=&\lim_{m\to\omega}\frac{1}{|X_k|\,|Y_k|}\,
\Big\langle 
\Theta_m(\gamma^{(k)}_{X,A}(\sum_{\vec\lambda}\ e_{x, (a,\vec\lambda)}))\,
\Lambda_m(\gamma^{(k)}_{Y,B}(\sum_{\vec\mu}\ e_{y,(b,\vec\mu)}))\,\xi_{m},\, \xi_m\,\Big\rangle\\[0.5ex]
& =& \lim_{m\to\omega}\,
\Big\langle 
L_{\iota_{X,1}^{(k)}(\delta_x)}(\Phi_{E_m}'(\iota_A^{(n)}(\delta_a)))\,
L_{\iota_{Y,1}^{(k)}(\delta_y)}(\Phi_{F_m}'(\iota_B^{(n)}(\delta_b)))\,\xi_{m},\, \xi_m\,\Big\rangle\\[0.5ex]
& =& \lim_{m\to\omega}\, \Big\langle \delta_x\otimes \delta_y\otimes \xi_{m}\xi_{m}^*,\,
((\tilde{\mathcal E}_{X_k}^\infty\circ \Phi_{E_m}')(\iota_A^{(n)}(\delta_a)))_{1,3}\,
((\tilde{\mathcal E}_{Y_k}^\infty\circ \Phi_{F_m}')(\iota_B^{(n)}(\delta_b)))_{2,3}\Big\rangle; 
\end{eqnarray*}
if $ m \geq k$ then, by the inductivity relations (\ref{eq:Phi-product-ind}) and equation (\ref{eq:first-id}), we have 
\begin{eqnarray*}
& & \Big\langle \delta_x\otimes \delta_y\otimes \xi_{m}\xi_{m}^*,\,
((\tilde{\mathcal E}_{X_k}^\infty \circ\Phi_{E_m}')(\iota_A^{(n)}(\delta_a)))_{1,3}\,
((\tilde{\mathcal E}_{Y_k}^\infty \circ\Phi_{F_m}')(\iota_B^{(n)}(\delta_b)))_{2,3}\Big\rangle\\[0.5ex]
& = &
\hspace{-0.25cm}
\Big\langle \hspace{-0.05cm} \delta_x \hspace{-0.1cm}\otimes \hspace{-0.1cm} \delta_y
\hspace{-0.1cm}\otimes \hspace{-0.1cm}\xi_{m}\xi_{m}^*,
\hspace{-0.05cm}
(\hspace{-0.02cm}(\tilde{\mathcal E}_{X_m,X_k} \hspace{-0.05cm}\circ \hspace{-0.05cm} \Phi_{E_m}\hspace{-0.05cm})(\iota_{A_{n,A_m}}(\delta_a)\hspace{-0.02cm})\hspace{-0.02cm})_{1,3}
(\hspace{-0.02cm}(\tilde{\mathcal E}_{Y_m,Y_k}\hspace{-0.05cm}\circ \hspace{-0.05cm}\Phi_{F_m}\hspace{-0.05cm})(\iota_{B_n,B_m}(\delta_b)\hspace{-0.02cm})\hspace{-0.02cm})_{2,3}\hspace{-0.05cm}\Big\rangle\\[0.5ex]
 &=& 
 \hspace{-0.25cm}
 \Big\langle \hspace{-0.05cm} \delta_x\otimes \delta_y\otimes \xi_{m}\xi_{m}^*,\,
\left(\widetilde{\mathcal E_{X_m,X_k}\otimes  \mathcal E_{Y_m,Y_k}}\circ\Phi_{E_m\cdot F_m}\circ(\iota_{A_{n},A_m}\otimes \iota_{B_n,B_m})\right)\hspace{-0.1cm}(\delta_a \otimes \delta_b)\hspace{-0.05cm}\Big\rangle\\[0.5ex]
&=& 
\hspace{-0.25cm}
\Big\langle \delta_x\otimes \delta_y\otimes \xi_{k}\xi_{k}^*,
\left(\Phi_{E_k\cdot F_k}\circ(\iota_{A_{n},A_k}\otimes \iota_{B_n,B_k})\right)
\hspace{-0.02cm}(\delta_a \otimes \delta_b)\Big\rangle\\[0.5ex]
&=&
\hspace{-0.25cm}
\Big\langle(\iota_{X,1}^{(k)}\otimes\iota_{Y,1}^{(k)})(\delta_x\otimes\delta_y),\
\Gamma\big((\iota_A^{(n)}\otimes\iota_B^{(n)})(\delta_a\otimes\delta_b)\big)\Big\rangle. 
\end{eqnarray*} 
Combining the previous calculations, we conclude that
\begin{multline*}\label{eq:first-id}
\Big\langle(\iota_{X,1}^{(k)}\otimes\iota_{Y,1}^{(k)})(\delta_x\otimes\delta_y),\
\Gamma\big((\iota_A^{(n)}\otimes\iota_B^{(n)})(\delta_a\otimes\delta_b)\big)\Big\rangle\\
=\ \Big\langle (\iota_{X,1}^{(k)}\otimes \iota_{Y,1}^{(k)})(\delta_x\otimes \delta_y)\otimes \xi\xi^*,\,
\Phi_{E\cdot F}(\iota_A^{(n)}(\delta_a) \otimes \iota_B^{(n)}(\delta_b)) \Big\rangle;
\end{multline*}
therefore
$\Gamma\in \mathcal C_{\rm qc}(X,Y,A,B)$.
\end{proof}

\begin{corollary}\label{c_clofqc}
Let $X$, $Y$, $A$ and $B$
be inductive families of finite sets. 
Then the set $\cl C_{\rm qc}(X,Y,A,B)$ is closed in the 
BW topology.
\end{corollary}

\begin{proof}
The claim follows from Theorem \ref{th_n_qc}, 
the fact that the map $\Gamma\to \Gamma_n$ is 
continuous in the BW topology, and 
the fact that the class of all quantum commuting 
correlations over a quadruple of finite sets is closed (see Theorem \ref{th_orcha}). 
\end{proof}

\begin{proposition}
Let $X$, $Y$, $A$ and $B$
be inductive families of finite sets. Then, writing 
$\cl C_{\rm t} = \cl C_{\rm t}(X,Y,A,B)$, we have 
\begin{equation*}
   \cl C_{\rm loc} \subseteq \cl C_{\rm qs} \subseteq \cl C_{\rm qa}\subseteq \cl C_{\rm qc}  \subseteq \cl C_{\rm ns}.
\end{equation*}
\end{proposition}

\begin{proof}
We show that quantum commuting correlations are no-signalling. Assume that $ \Gamma \in \cl C_{\rm qc}$;  we claim that 
$ \Gamma( C(\Omega_A) \otimes 1_{B}) \subseteq L^{\infty}(\Omega_X) \otimes 1_{Y}$ and $\Gamma(1_{A} \otimes C(\Omega_B)) \subseteq  
1_{X} \otimes L^{\infty}(\Omega_Y)$. Fix  $h \in C(\Omega_A)$; 
it suffices to show that $L_{\omega}(\Gamma(h\otimes 1_{B})) \in \bb{C}\cdot 1_{Y}$ for all $ \omega \in L^{1}(\Omega_X)$. 
Let $ g \in L^{1}(\Omega_Y)$ and note that
\begin{align*}
   & \langle g, L_{\omega}(\Gamma(h\otimes 1_{B})) \rangle = \langle g \otimes \omega , \Gamma(h\otimes 1_{B}) \rangle \\
     & = \langle (\omega \otimes g) \otimes \xi \xi^*, \Phi_{E\cdot F}(h \otimes 1_{B})\rangle 
 = \langle (\omega \otimes g) \otimes \xi \xi^*, \Phi_{E}(h) \otimes 1_{Y}\rangle\\
&     =  \langle  g , L_{\omega \otimes \xi \xi^*}(\Phi_{E}(h)) \cdot 1_{Y}\rangle
\end{align*}
and, since $ g$ is arbitrary, we have that $$L_{\omega}(\Gamma(h\otimes 1_{B})) = L_{\omega \otimes \xi \xi^*}(\Phi_{E}(h)) \cdot 1_{Y}$$ for every $ \omega \in  L^{1}(\Omega_X)$, as desired. The fact that $\Gamma(1_{A} \otimes C(\Omega_B)) \subseteq  
1_{X} \otimes L^{\infty}(\Omega_Y)$ follows by symmetry.

The first inclusion 
can be shown in a standard way following the finite case, the second one is
trivial, while the third follows from 
Corollary \ref{c_clofqc} and 
the fact that $\cl C_{\rm qs}\subseteq \cl C_{\rm qc}$.
\end{proof}


Our next aim is to obtain an operator algebraic
description of quantum commuting and approximately quantum correlations over 
Cantor spaces. 

\begin{remark} \rm  \label{r_ind_stat_qc}
\rm   
Given a correlation $\Lambda$ of quantum commuting type over a quadruple 
$(S,T,U,V)$ of finite sets, let $s_{\Lambda}$
be the (unique) state of $\cl S_{S,U}\otimes_{\rm c} \cl S_{T,V}$
corresponding to $\Lambda$ via 
Theorem \ref{th_orcha}. 
By Lemma \ref{lem_ind_stat}, 
if $\Gamma_n$ is a no-signalling correlation over $(X_n,Y_n,A_n,B_n)$, 
$n\in \bb{N}$, 
then the family $(\Gamma_n)_{n\in\bb{N}}$ 
is inductive if and only if 
\begin{align}\label{ind_stat_qc}
s_{\Gamma_{n+1}}\circ(\gamma_{X_n,A_n}\otimes\gamma_{Y_n,B_n}) = s_{\Gamma_n},\ \ \ n\in\bb{N}.
\end{align}
\end{remark}

\begin{lemma}\label{l_ampliqc} 
Let ${\rm t}\in\{\rm loc, qs,qc, ns\}$, $n\in \bb{N}$ and
$\Gamma_n\in \cl C_{\rm t}(X_n,Y_n,A_n,B_n)$. 
Then the correlation $\tilde{\Gamma}_n$ defined by letting $$\tilde{\Gamma}_n = (\iota_{X,\infty}^{(n)}\otimes \iota_{Y,\infty}^{(n)})\circ \Gamma_n\circ (\cl E_{A_{n}}^{\infty} 
\otimes \cl E_{B_{n}}^{\infty})|_{C(\Omega_A)\otimes C(\Omega_B)}$$
belongs to $\cl C_{\rm t}(X,Y,A,B)$.
\end{lemma}

\begin{proof}
We only include the proof for the case ${\rm t} = {\rm qc}$; 
the case ${\rm t} = {\rm qs}$ is similar and the cases $ \rm t = loc, ns$ are immediate. 
Let $(H,E,F,\xi)$ be a realisation of $\Gamma_n$; thus, 
$\Phi_E : \cl D_{A_n}\to \cl D_{X_n}\otimes\cl B(H)$ and 
$\Phi_F : \cl D_{B_n}\to \cl D_{Y_n}\otimes\cl B(H)$ are unital completely positive map
and $\xi\in H$ is a unit vector.
Let $\tilde{E} \in \frak{C}_\mu(\Omega_A,\Omega_X;H)$ and
$\tilde{F} \in \frak{C}_\mu(\Omega_B,\Omega_Y;H)$ be the operator-valued 
channels, satisfying
$$\Phi_{\tilde{E}} = \iota_{X,\infty}^{(n)}\circ \Phi_E\circ \cl E_{A_{n}}^{\infty}|_{C(\Omega_A)} \ 
\mbox{ and } \ 
\Phi_{\tilde{F}} = \iota_{Y,\infty}^{(n)}\circ \Phi_F\circ \cl E_{B_{n}}^{\infty}|_{C(\Omega_B)}.$$
It is straightforward to check that $(H,\tilde{E},\tilde{F},\xi)$ is a realisation 
of $\tilde{\Gamma}_n$.
\end{proof}

For the formulation of the next theorem, recall once again the notation (\ref{eq_cylinddd})
for cylinders in Cantor spaces.

\begin{theorem} \label{th_qc} 
Let $X = (X_n)_{n\in \bb{N}}$, $Y = (Y_n)_{n\in \bb{N}}$, 
$A = (A_n)_{n\in \bb{N}}$ and $B = (B_n)_{n\in \bb{N}}$ be inductive families of finite sets, and 
let $\Gamma:C(\Omega_A)\otimes C(\Omega_B)\rightarrow L^\infty(\Omega_X)\bar\otimes L^\infty(\Omega_Y)$ be a 
no-signalling correlation. The following are equivalent: 

\begin{enumerate}
    \item[(i)]
$\Gamma \in \cl C_{\rm qc}(X,Y,A,B)$ (resp. $\Gamma \in \cl C_{\rm qa}(X,Y,A,B)$);

\item[(ii)]
there exists a state 
$s : \cl{S}_{X,A}\otimes_{\rm c}\cl{S}_{Y,B}\rightarrow \bb{C}$ 
(resp. $s : \cl{S}_{X,A}\otimes_{\min}\cl{S}_{Y,B}\rightarrow \bb{C}$), 
such that 
     \begin{equation*}\label{eq_onag}
s(\chi_{\tilde{x}\times\tilde{a}} \otimes \chi_{\tilde{y}\times\tilde{b}}) 
=
\lvert X_n\rvert \lvert Y_n\rvert \langle \chi_{\tilde{x}}\otimes \chi_{\tilde{y}},\Gamma(\chi_{\tilde{a}}\otimes \chi_{\tilde{b}})\rangle, 
    \end{equation*}
    for all $x\in X_n$, $y\in Y_n$, $a\in A_n$, $b\in B_n$, and all $n\in \bb{N}$.
    \end{enumerate}
\end{theorem}

\begin{proof}
We first establish the equivalence in the quantum commuting case.

\smallskip

(i)$\Rightarrow$(ii) 
Assume that 
$\Gamma\in\cl{C}_{\rm qc}(X,Y,A,B)$, and let 
$H$ be a separable Hilbert space, 
$\xi\in H$ be a unit vector, and 
$E\in\frak{C}_{\mu_X}(\Omega_A,\Omega_X;H)$  and $F\in\frak{C}_{\mu_Y}(\Omega_B,\Omega_Y;H)$ 
be operator-valued channels, satisfying (\ref{eq_eqforqcde}).
Let 
$\Phi_{E} : C(\Omega_A)\rightarrow L^\infty(\Omega_X)\bar{\otimes} \cl{B}(H)$
and 
$\Phi_{F} : C(\Omega_B)\rightarrow L^\infty(\Omega_Y)\bar{\otimes} \cl{B}(H)$
be the unital completely positive maps, associated with $E$ and $F$, respectively, via 
Theorem \ref{mod_mu}.
Let $(\Phi_n)_{n\in \bb{N}}$ and $(\Psi_n)_{n\in \bb{N}}$ be the inductive families, associated with 
$\Phi_E$ and $\Phi_F$, respectively, via (\ref{eq_gammanst}). 
Let 
$\tilde{\Phi}_n : \cl S_{X_n,A_n}\to \cl{B}(H)$ and 
and 
$\tilde{\Psi}_n : \cl S_{Y_n,B_n}\to \cl{B}(H)$ be the 
unital completely positive maps, arising from $\Phi_n$ and $\Psi_n$, respectively, through
Remark \ref{r_forfinnew}. 
It follows from the proof of Theorem \ref{c_statSXA} that 
$$\tilde{\Phi}_n = \tilde{\Phi}_{n+1}\circ \gamma_{X_n,A_n} \ \mbox{ and } \ 
\tilde{\Psi}_n = \tilde{\Psi}_{n+1}\circ \gamma_{Y_n,B_n}, \ \ \ n\in \bb{N}.$$
By the universal property of the operator system inductive limit, there exist 
unital completely positive maps 
$\tilde{\Phi} : \cl S_{X,A}\to \cl{B}(H)$ and 
and 
$\tilde{\Psi} : \cl S_{Y,B}\to \cl{B}(H)$, such that 
$$
\tilde{\Phi} \circ \gamma^{(n)}_{X,A} = \tilde{\Phi}_n \ \mbox{ and } \ 
\tilde{\Psi} \circ \gamma^{(n)}_{Y,B} = \tilde{\Psi}_n, \ \ \ n\in \bb{N}.
$$
Since the pair $(\Phi_{E},\Phi_{F})$ is commuting, so is $(\tilde{\Phi},\tilde{\Psi})$.
By the definition of the operator system commuting tensor product, the map 
$\tilde{\Phi}\cdot \tilde{\Psi} : \cl S_{X,A}\otimes_{\rm c} \cl S_{Y,B}\to \cl B(H)$ is 
(unital and) completely positive. 
Let $s :  \cl S_{X,A}\otimes_{\rm c} \cl S_{Y,B} \to \bb{C}$ be the state, given by 
$$s(u) = \langle (\tilde{\Phi}\cdot \tilde{\Psi})(u),\xi,\xi\rangle, \ \ \ u\in \cl S_{X,A}\otimes_{\rm c} \cl S_{Y,B}.$$

Set 
$s_n := s\circ (\gamma^{(n)}_{X}\otimes \gamma^{(n)}_{A})$; thus, $s_n$ is a state on
$\cl S_{X_n,A_n}\otimes_{\rm c}\cl S_{Y_n,B_n}$. 
Let $\Gamma_n$ be given via (\ref{Gamma_n}), $n\in \bb{N}$, 
and observe that, if $x\in X_n$, $y\in Y_n$, $a\in A_n$ and $b\in B_n$, then 
\begin{eqnarray*}
s(\chi_{\tilde{x}\times\tilde{a}} \otimes \chi_{\tilde{y}\times\tilde{b}}) 
& = & 
s_n(e_{x,a}\otimes e_{y,b})
= 
\langle \tilde{\Phi}_n(e_{x,a}) \tilde{\Psi}_n(e_{y,b})\xi,\xi \rangle\\
& = & 
\langle L_{\delta_x}(\Phi_n(\delta_a)) L_{\delta_y}(\Psi_n(\delta_b))\xi,\xi \rangle\\
& = & 
|X_n||Y_n| \langle \delta_x\otimes\delta_y, \Gamma_n(\delta_a\otimes\delta_b)\rangle\\
& = & 
\lvert X_n\rvert \lvert Y_n\rvert \langle (\iota_{X,1}^{(n)}\otimes \iota_{Y,1}^{(n)})(\delta_x\otimes\delta_y),\Gamma(\chi_{\tilde{a}}\otimes \chi_{\tilde{b}})\rangle\\
& = & 
\lvert X_n\rvert \lvert Y_n\rvert \langle \chi_{\tilde{x}}\otimes \chi_{\tilde{y}},\Gamma(\chi_{\tilde{a}}\otimes \chi_{\tilde{b}})\rangle.
\end{eqnarray*}

(ii)$\Rightarrow$(i)
 Setting     $$s_n=s\circ (\gamma^{(n)}_{X,A}\otimes \gamma^{(n)}_{Y,B}),\ \ n\in\bb{N},$$ we obtain a family $(s_n)_{n\in \bb{N}}$ of states, where $s_n:\cl{S}_{X_n,A_n}\otimes_{\rm c}\cl{S}_{Y_n,B_n}\rightarrow\bb{C}$, $n\in\bb{N}$, satisfying (\ref{ind_stat_qc}) and therefore quantum commuting correlations $\Gamma_n$ over $(X_n,Y_n,A_n,B_n)$ which by Remark \ref{r_ind_stat_qc} form an inductive family. Theorem \ref{th_indcorresp} gives rise to a unital completely positive map $\Gamma:C(\Omega_A)\otimes C(\Omega_B)\rightarrow L^{\infty}(\Omega_X)\bar{\otimes} L^\infty(\Omega_Y)$  that satisfies
\begin{align*}
(\cl{E}_{X_n}^{\infty}\otimes \cl{E}_{Y_n}^{\infty})\circ\Gamma\circ
(\iota_{A}^{(n)}\otimes \iota_{B}^{(n)}) = \Gamma_n, \ \ \ n\in \bb{N}.
\end{align*} 
By Theorems \ref{th_orcha} and \ref{th_n_qc}, 
$\Gamma\in\cl{C}_{\rm qc}(X,Y,A,B)$.

\smallskip

We now consider the approximately quantum case. 
Assume that the implication (i)$\Rightarrow$(ii) holds in the case 
where $\Gamma\in \cl C_{\rm qs}(X,Y,A,B)$; to conclude it in the full 
generality, let $\Gamma\in \cl C_{\rm qa}(X,Y,A,B)$ and 
$(\Gamma^{(k)})_{k\in \bb{N}} \subseteq \cl C_{\rm qs}(X,Y,A,B)$
be a sequence with BW limit $\Gamma$. 
Let $s_k : \cl S_{X,A}\otimes_{\min}\cl S_{Y,B} \to \bb{C}$ be a state 
yielding $\Gamma^{(k)}$, $k\in \bb{N}$. 
Let $(s_{k_l})_{l\in \bb{N}}$ be a subsequence such that $s_{k_l}\to_{l\to \infty} s$
in the weak* topology; then 
$s(\gamma_{X,A}^{(k_l)}(e_{x,a})\otimes \gamma_{Y,B}^{(k_l)}(e_{y,b}))$ agrees with 
$|X_{k_l}||Y_{k_l}|\langle \chi_{\tilde{x}}\otimes \chi_{\tilde{y}},\Gamma(\chi_{\tilde{a}}\otimes \chi_{\tilde{b}})\rangle$ for all 
$x\in X_{k_l}$, $y\in Y_{k_l}$, $a\in A_{k_l}$ and $b\in B_{k_l}$ and all $l\in \bb{N}$. 
By uniform boundedness, the density of the linear span of the elements 
$\chi_{\tilde{a}}\otimes \chi_{\tilde{b}}$ in $C(\Omega_A)\otimes C(\Omega_B)$, 
and the weak* density of the elements $\chi_{\tilde{x}}\otimes \chi_{\tilde{y}}$ in $L^{\infty}(\Omega_X)\bar\otimes L^{\infty}(\Omega_Y)$, we conclude that 
$s = s_{\Gamma}$. 

To complete the proof of the implication (i)$\Rightarrow$(ii), 
we note that in the case where
$\Gamma\in \cl C_{\rm qs}(X,Y,A,B)$ 
the statement follows readily by inspecting the proof of the same 
implication in the quantum commuting case, working with tensor, instead of operator products, 
and using the fact that unital completely positive maps on the individual terms 
tensor to a unital completely positive map on the minimal operator system tensor product. 

For the implication (ii)$\Rightarrow$(i) in the 
approximately quantum case,
let $s : \cl{S}_{X,A}\otimes_{\min}\cl{S}_{Y,B}\rightarrow \bb{C}$ be a state, and 
$s_n = s\circ (\gamma_{X,A}^{(n)}\otimes_{\min} \gamma_{Y,B}^{(n)})$; thus, 
$s_n$ is a state on $\cl{S}_{X_n,A_n}\otimes_{\min}\cl{S}_{Y_n,B_n}$, $n\in \bb{N}$.
By Theorem \ref{th_orcha}, $s_n$ gives rise to a no-signalling correlation 
$\Gamma_n : \cl D_{A_n}\otimes\cl D_{B_n}\to \cl D_{X_n}\otimes\cl D_{Y_n}$
of approximately quantum type. 
Let $\tilde{\Gamma}_n : C(\Omega_{A})\otimes C(\Omega_{B})
\to L^{\infty}(\Omega_{X})\bar{\otimes} L^{\infty}(\Omega_{Y})$ be the 
map given by letting
\begin{equation}\label{eq_forti}
\tilde{\Gamma}_n = (\iota_{X,\infty}^{(n)}\otimes \iota_{Y,\infty}^{(n)})\circ \Gamma_n\circ (\cl E_{A_{n}}^{\infty} 
\otimes \cl E_{B_{n}}^{\infty})|_{C(\Omega_A)\otimes C(\Omega_B)}.
\end{equation}
Using the argument from Remark \ref{r_appuni}, we see that
$\tilde{\Gamma}_n\to_{n\to \infty} \Gamma$ in the BW topology. 

It therefore suffices to show that 
$\tilde{\Gamma}_n\in \cl C_{\rm qa}(X,Y,A,B)$, $n\in \bb{N}$. 
To see that, fix $n\in \bb{N}$, and let $(\Theta_k)_{k\in \bb{N}}\subseteq \cl C_{\rm qs}(X_n,Y_n,A_n,B_n)$ be a sequence, such that $\Theta_k\to_{k\to \infty} \Gamma_n$. 
It is straightforward to see that, if $\tilde{\Theta}_k$ arises from $\Theta_k$ as in 
(\ref{eq_forti}) then $\tilde{\Theta}_k\to_{k\to \infty} \tilde{\Gamma}_n$ in the BW topology. 
On the other hand, if $H$ and $K$ are separable Hilbert spaces, 
$\xi\in H\otimes K$ is a unit vector, and  
$E\in\frak{C}_{\mu_{{X_n}}}(A_n,X_n;H)$ and $F\in\frak{C}_{\mu_{Y_n}}(B_n,Y_n;K)$ are operator-valued channels 
such that 
$$\langle\Theta_k(\delta_a\otimes\delta_b),\delta_x\otimes\delta_y\rangle
= \langle \delta_x\otimes\delta_y\otimes\xi\xi^*, (\Phi_E\otimes\Phi_F)(\delta_a\otimes\delta_b)\rangle,$$
then $\tilde{\Theta}_k$ has the form (\ref{eq_forminnn})
for the channels $\tilde{E}$ and $\tilde{F}$, satisfying 
\begin{equation}\label{eq_onsing}
\Phi_{\tilde{E}} = \iota_{X,\infty}^{(n)}\circ \Phi_E\circ \cl E_{A_{n}}^{\infty}|_{C(\Omega_A)} \ 
\mbox{ and } \ 
\Phi_{\tilde{F}} = \iota_{Y,\infty}^{(n)}\circ \Phi_F\circ \cl E_{B_{n}}^{\infty}|_{C(\Omega_B)}.
\end{equation}
The proof is complete. 
\end{proof}

The next statement, which is a straightforward 
consequence of Theorem \ref{th_qc}, complements
Theorem \ref{th_n_qc} in the approximately quantum 
case.

\begin{corollary}\label{r_qacase}
Let $\Gamma:C(\Omega_A)\otimes C(\Omega_B)\rightarrow L^\infty(\Omega_X)\bar\otimes L^\infty(\Omega_Y)$ be a unital completely positive map and 
$(\Gamma_n)_{n\in \bb N}$ be its associated inductive family of maps. Then 
$\Gamma\in\cl{C}_{\rm qa}(X,Y,A,B)$ if and only if
$\Gamma_n\in\cl{C}_{\rm qa}(X_n,Y_n,A_n,B_n)$ for every $n\in\bb{N}$.  
\end{corollary}

\begin{remark}\label{r_locqson}
\rm 
It is straightforward to see that, if 
$\Gamma\in\cl{C}_{\rm qs}(X,Y,A,B)$ (resp.\,$\Gamma\in\cl{C}_{\rm loc}(X,Y,A,B)$)
then 
$\Gamma_n\in\cl{C}_{\rm qs}(X_n,Y_n,A_n,B_n)$ 
(resp.\,$\Gamma_n\in\cl{C}_{\rm loc}(X,Y,A,B)$) for every $n\in\bb{N}$.
We finish this section by showing that Corollary \ref{r_qacase} does not hold in the quantum spatial case and that the class $\cl{C}_{\rm qs}(X,Y,A,B)$ is not closed. 
\end{remark}

\begin{theorem}\label{th_qs}
There exist inductive families of finite sets
$X = (X_n)_{n\in \bb{N}}$, $Y = (Y_n)_{n\in \bb{N}}$, 
$A = (A_n)_{n\in \bb{N}}$ and $B = (B_n)_{n\in \bb{N}}$ 
and a no-signalling correlation $\Gamma: C(\Omega_A)\otimes C(\Omega_B)\to L^\infty(\Omega_X)\bar\otimes L^\infty(\Omega_Y)$ such that, if 
$(\Gamma_n)_{n\in\mathbb N}$ is its associated inductive family, then
$\Gamma_n\in \cl C_{\rm qs}(X_n,Y_n,A_n,B_n)$ for every $n\in \bb{N}$, but $\Gamma\not\in \cl C_{\rm qs}(X,Y,A,B)$. 
\end{theorem}

\begin{proof}
By \cite{slofstra},  there exist   finite sets $\mathbb X$, $\mathbb Y$, $\mathbb A$, $\mathbb B$, such that $\cl C_{\rm qs}(\mathbb X,\mathbb Y,\mathbb A,\mathbb B)$ is not closed, that is, there exist correlations $p_n\in\cl C_{\rm qs}(\mathbb X,\mathbb Y,\mathbb A,\mathbb B)$, $n\in\mathbb N$,  and  a correlation $p\in \cl C_{\rm qa}(\mathbb X,\mathbb Y,\mathbb A,\mathbb B)\setminus \cl C_{\rm qs}(\mathbb X,\mathbb Y,\mathbb A,\mathbb B)$ such that $p_n(a,b|x,y)\to p(a,b|x,y)$ for all $a$, $b$, $x$, $y$.

Let $f$ and  $f_n$ be the states  on $\cl S_{\mathbb X,\mathbb A}\otimes_{\rm min}\cl S_{\mathbb Y,\mathbb B}$ yielding $p$ and $p_n$ respectively via Theorem \ref{th_orcha}; thus, 
$$p(a,b|x,y)=f(e_{x,a}\otimes e_{y,b}) \ \mbox{ and } p_n(a,b|x,y)=f_n(e_{x,a}\otimes e_{y,b})$$
for all $n\in \bb{N}$. 
Set $X_n=\prod_{i=0}^{n-1}\mathbb X$ and let $X=(X_n)_{n\in\mathbb N}$ be the corresponding inductive family of sets; define the families $Y$, $A$, $B$ similarly.  Let  $s_n : \cl S_{X_n,A_n}\otimes_{\rm min} \cl S_{Y_n,B_n} \to \bb{C}$, $n\in \bb{N}$, 
be the linear maps, defined inductively by letting $s_1=f_1$ and 
$$s_{n+1}\left(e_{(xx'),(aa')}\otimes e_{(yy'),(bb')}\right) = s_{n}(e_{x,a}\otimes e_{y,b})f_{n+1}(e_{x',a'}\otimes e_{y',b'}),$$ 
where $x\in X_{n}$, $x'\in \mathbb X$, $y\in Y_{n}$, $y'\in \mathbb Y$, $a\in A_{n}$, $a'\in \mathbb A$ and $b\in B_{n}$, $b'\in \mathbb 
B$. 
To see that the maps $s_n$ are well-defined, we refer to the 
commutativity of the minimal tensor product and the 
universal property of the C*-algebra $\cl A_{X_{n+1},A_{n+1}}$, 
according to which the map
$$\cl A_{X_{n+1},A_{n+1}} \mapsto \cl A_{X_{n}, A_{n}}\otimes_{\rm min} \cl A_{\mathbb X,\mathbb A}; \ \ \ e_{xx',aa'}\mapsto e_{x,a}\otimes e_{x',a'},$$
gives rise to a $*$-homomorphism that restricts to 
a unital completely positive map from $\cl S_{X_{n+1},A_{n+1}}$ to $\cl S_{X_n,A_n}\otimes_{\rm min} \cl S_{\mathbb X,\mathbb A}$, $n\in \bb{N}$. 
Moreover, 
\begin{eqnarray*}
&&s_{n+1}(\gamma_{X_n,A_n}(e_{x,a})\otimes\gamma_{Y_n,B_n}(e_{y,b}))\\&&=\frac{1}{|\mathbb X||\mathbb Y|}
\sum_{x'\in \bb{X}} \sum_{y'\in \bb{Y}}
\sum_{a'\in \bb{A}} \sum_{b'\in \bb{B}}
s_{n+1}(e_{(xx'),(aa')}\otimes e_{(yy'),(bb')})\\&&=\frac{1}{|\mathbb X||\mathbb Y|}
\sum_{x'\in \bb{X}} \sum_{y'\in \bb{Y}}
\sum_{a'\in \bb{A}} \sum_{b'\in \bb{B}}
s_{n}(e_{x,a}\otimes e_{y,b})f_{n+1}(e_{x',a'}\otimes e_{y',b'}) =s_{n}(e_{x,a}\otimes e_{y,b}).
\end{eqnarray*}
By Lemma \ref{lem_ind_stat}, the family $(s_n)_{n\in\mathbb N}$ gives rise to the inductive family of correlations $(\Gamma_n)_{n\in \mathbb N}$. 
It is easy to see that, since $p_n$ is of quantum spatial type, so is $\Gamma_n$, $n\in \bb{N}$. 
Let $\Gamma: C(\Omega_A)\otimes C(\Omega_B)\to L^\infty(\Omega_X)\bar\otimes L^\infty(\Omega_Y)$ be the unique unital completely positive map associated to the sequence $(\Gamma_n)_{n\in\mathbb N}$ and $s: \cl S_{X,A}\otimes_{\rm min}\cl S_{Y,B}\to\mathbb C$ be the corresponding state, arising via Theorem \ref{th_qc}. Then $s\circ(\gamma_{X,A}^{(n)}\otimes\gamma_{Y,B}^{(n)})=s_n$, $n\in \bb{N}$. 

We now show that $\Gamma\not\in \cl C_{\rm qs}(X,Y,A,B)$.  
Fix $x\in \mathbb X$, $y\in \mathbb Y$, $a\in \mathbb A$ and $b\in \mathbb B$, consider the sets $\Lambda_x^n=\{(x_i)_i\in\Omega_X: x_n=x\}$ and define $\Lambda_a^n$, $\Lambda_y^n$ and $\Lambda_b^n$ similarly. 
Assuming, towards a contradiction, that $\Gamma$ is in $\cl C_{\rm qs}(X,Y,A,B)$ we can find Hilbert spaces $H$ and $K$, unital completely positive maps $\Phi: C(\Omega_A)\to L^\infty(\Omega_X)\bar\otimes \cl B(H)$ and $\Psi: C(\Omega_B)\to L^\infty(\Omega_Y)\bar\otimes\cl B(K)$, 
and a unit vector $\xi\in H\otimes K$, such that 
\begin{eqnarray*}
&&\langle\chi_{\Lambda_x^n}\otimes\chi_{\Lambda_y^n}\otimes\xi\xi^*,\Phi(\chi_{\Lambda_a^n})\otimes\Psi(\chi_{\Lambda_b^n})\rangle\\
&& =\langle\chi_{\Lambda_x^n}\otimes\chi_{\Lambda_y^n},\Gamma(\chi_{\Lambda_a^n}\otimes\chi_{\Lambda_b^n})\rangle
=
s(\chi_{\Lambda_x^n\times\Lambda_a^n}\otimes \chi_{\Lambda_y^n\times\Lambda_b^n})\\
&& =
\sum_{x',y',a',b'} 
s((\gamma_{X,A}^{(n)}\otimes\gamma_{Y,B}^{(n)})(e_{(x'x), (a'a)}\otimes e_{(y'y),(b'b)}))\\
&&=
\sum_{x',y',a',b'} 
s_n(e_{(x'x),(a'a)}\otimes e_{(y'y),(b'b)})\\
&&=
\sum_{x',y',a',b'} 
s_{n-1}(e_{x',a'}\otimes e_{y',b'})f_n(e_{x,a}\otimes e_{y,b})
=
|X_{n-1}||Y_{n-1}|f_n(e_{x,a}\otimes e_{y,b}),
\end{eqnarray*}
where the summation is over $(x',y',a',b')\in X_{n-1}\times Y_{n-1}\times A_{n-1}\times B_{n-1}$.

Let $E_{x,a}^n=\frac{1}{|X_{n-1}|}L_{\chi_{\Lambda_x^n}}(\Phi(\chi_{\Lambda_a^n}))\in \cl B(H)$ and $F_{y,b}^n=\frac{1}{|Y_{n-1}|}L_{\chi_{\Lambda_y^n}}(\Phi(\chi_{\Lambda_b^n}))\in \cl B(K)$. Then
$$\langle (E_{x,a}^n\otimes F_{y,b}^n)\xi,\xi\rangle=f_n(e_{x,a}\otimes e_{y,b})=p_n(a,b|x,y),$$
for all $x\in \bb{X}$, $y\in \bb{Y}$, $a\in \bb{A}$ and 
$b\in \bb{B}$.
Moreover, $(E_{x,a}^n)_{a\in\mathbb A}$ and $(F_{y,b}^n)_{b\in\mathbb B}$ are families of POVM's for each $x\in\mathbb X$ and $y\in\mathbb Y$. Choose subnets $(E_{x,a}^{n_\alpha})_{\alpha}$, $(F_{y,b}^{n_\alpha})_{\alpha}$  converging to
 $E_{x,a}$ and $F_{y,b}$, respectively, in the weak* topology.  We have 
\begin{equation}\label{eq_Epntop}
\langle (E_{x,a}\otimes F_{y,b})\xi,\xi\rangle=\lim_{\alpha}\langle (E_{x,a}^{n_\alpha}\otimes F_{y,b}^{n_\alpha})\xi,\xi\rangle
= \lim_{n\to\infty} p_n(a,b|x,y) = p(a,b|x,y).
\end{equation}
 As the families $(E_{x,a})_{a\in\mathbb A}$, $(F_{y,b})_{b\in\mathbb B}$, $x\in\mathbb X$, $y\in\mathbb Y$ are again POVM's, identity (\ref{eq_Epntop}) 
 contradicts the fact that 
 $p$ is not of quantum spatial type. 
\end{proof}

\begin{corollary}\label{c_clofqs}
There exist inductive families of finite sets
$X$, $Y$, $A$ and $B$ for which
the set $\cl C_{\rm qs}(X,Y,A,B)$ is not closed in the 
BW topology.
\end{corollary}

\begin{proof}
    Let $(\Gamma_n)_{n\in\mathbb N}$ be the inductive family from Theorem \ref{th_qs} and let $\tilde\Gamma_n$ be the associated no-signalling correlations via Lemma \ref{l_ampliqc}. By Lemma \ref{l_ampliqc}, $\tilde \Gamma_n\in \cl C_{\rm qs}$. As $\tilde \Gamma_n\to \Gamma$ in the BW topology, the statement follows from Theorem \ref{th_qs}. 
\end{proof}

\noindent {\bf Remark.} 
Since the elements of the form 
$(\gamma^{(n)}_{X,A}\otimes \gamma^{(n)}_{Y,B})(e_{x,a} \otimes e_{y,b})$, where 
$x\in X_n, y\in Y_n$, $a\in A_n$ and $b\in B_n$, $n\in \bb{N}$, 
generate a dense subspace of the operator system $\cl S_{X,A}\otimes_{\rm c} \cl S_{Y,B}$, 
the correspondence between 
$\cl C_{\rm qc}(X,Y,A,B)$ (resp. $\cl C_{\rm qa}(X,Y,A,B)$) 
and the state space of 
$\cl S_{X,A}\otimes_{\rm c} \cl S_{Y,B}$ (resp. $\cl S_{X,A}\otimes_{\min} \cl S_{Y,B}$) 
from Theorem \ref{th_qc} is bijective.


\section{Cantor games}\label{s_Cantorg}

In this section, we define values of non-local games over Cantor sets, based on the correlation types studied in the previous sections, and establish continuity results thereof. 
We recall that, if $S$, $T$, $U$ and $V$ are finite sets, a non-local game
over the quadruple $(S,T,U,V)$ is a pair $G = (\lambda,\mu)$, where 
$\lambda : S\times T\times U\times V\to \{0,1\}$ and $\mu$ is a 
probability measure on $S\times T$; here, $S$ (resp.\,$U$) is interpreted as 
the set of inputs (resp.\,outputs) for player Alice, and 
$T$ (resp.\,$V$) -- as 
the set of inputs (resp.\,outputs) for player Bob. 
Alice and Bob play collaboratively against a third party, Verifier. 
In each round, the Verifier choses a pair $(s,t)\in S\times T$ of questions according 
to the probability measure $\mu$, 
and the players return a pair $(u,v)\in U\times V$; 
the tandem Alice-Bob wins (resp.\,loses) the round if $\lambda(s,t,u,v) = 1$
(resp.\,$\lambda(s,t,u,v) = 0$). 
Given a correlation type ${\rm t}$ over $(S,T,U,V)$, the \emph{{\rm t}-value}
of $G$ is the parameter
$$\omega_{\rm t}(\lambda,\mu) = \sup_{p\in \cl C_{\rm t}}
\sum_{s\in S}\sum_{t\in T}\mu(s,t) \sum_{u\in U}\sum_{v\in V}
\lambda(s,t,u,v) p(u,v|s,t)$$
(we note that $\omega_{\rm qs}(\lambda,\mu) = \omega_{\rm qa}(\lambda,\mu)$).

Measurable games and their values were defined in \cite{btt}; 
here, we specialise those to the case of Cantor topological spaces. 
For an inductive family $X$ of finite sets, we let for brevity $\frak{B}_X = \frak{B}_{\Omega_X}$. 
Let $X = (X_n)_{n\in \bb{N}}$, $Y = (Y_n)_{n\in \bb{N}}$, 
$A = (A_n)_{n\in \bb{N}}$ and $B = (B_n)_{n\in \bb{N}}$ be inductive families of finite sets. 
A \emph{Cantor game} is a pair $(\kappa,\mu)$, where 
$\kappa\subseteq \Omega_X\times \Omega_Y\times \Omega_A\times\Omega_B$ is a closed subset and $\mu$ is a probability measure on $\Omega_X\times \Omega_Y$. 
In the sequel, we consider only the case where $\mu$ is the uniform probability 
measure, that is, $\mu = \mu_X\times\mu_Y$. 

We equip $X_n\times Y_n$ with the uniform probability measure, 
and hence identify a non-local game $\cl G_n$ over $(X_n,Y_n,A_n,B_n)$
with its rule function $\lambda_n : X_n\times Y_n \times A_n\times B_n\to \{0,1\}$,
$n\in \bb{N}$. 
Let, further $\kappa_{\cl{G}_n}$ be the subset 
of $\Omega_X\times \Omega_Y\times \Omega_A\times\Omega_B$,
given by 
\begin{align}\label{eq_kappaGn}
    \kappa_{\cl{G}_n} = \{((x_k)_k),((y_k)_k),&((a_k)_k),((b_k)_k) : \\
&\left((x_k)_{k=0}^{n-1},(y_k)_{k=0}^{n-1},
(a_k)_{k=0}^{n-1},(b_k)_{k=0}^{n-1}\right)\in\supp\lambda_n\},\nonumber
\end{align} 
and note that $\kappa_{\cl{G}_n}$ is (open and) closed. 
We say that the family $(\cl G_n)_{n\in \bb{N}}$ is \emph{nested}
if $\kappa_{\cl{G}_{n+1}}\subseteq \kappa_{\cl{G}_n}$ for every $n$. 
For a nested family $\cl G = (\cl G_n)_{n\in \bb{N}}$ of games,
we set 
$\kappa_\cl{G} := \cap_{n\in\bb{N}}\kappa_{\cl{G}_n}$, 
and note that $\kappa_\cl{G}$
is a closed subset of $\Omega_X\times \Omega_Y\times \Omega_A\times\Omega_B$. 
We call the Cantor games of the latter form \emph{nested}.

\begin{lemma}
Every Cantor game $\kappa\subseteq \Omega_X\times \Omega_Y\times \Omega_A\times\Omega_B$ is nested.
\end{lemma}
\begin{proof}
Let $\pi_n : \Omega_X\times \Omega_Y\times \Omega_A\times\Omega_B\to 
X_n\times Y_n\times A_n\times B_n$ be the projection, and 
set $\kappa_n = \pi_n^{-1}(\pi_n(\kappa))$, $n\in \bb{N}$.
Clearly, $\kappa_{n+1}\subseteq \kappa_n$ for every $n$, and 
$\kappa\subseteq \cap_{n=1}^{\infty}\kappa_n$. 
Assuming that $\omega\in \cap_{n=1}^{\infty}\kappa_n$,
let $\omega_n\in X_n\times Y_n\times A_n\times B_n$ and 
$\omega_n'\in \prod_{i \geq n} [d_i^X]\times [d_i^Y]\times [d_i^A]\times [d_i^B]$ be such 
that $\omega_n\in \pi_n(\kappa)$ and $\omega = \omega_n\omega_n'$, $n\in \bb{N}$. 
Since $\omega_n\in \pi_n(\kappa)$, there exists 
$\omega_n''\in \prod_{i \geq n} [d_i^X]\times [d_i^Y]\times [d_i^A]\times [d_i^B]$, such that 
$\omega^{(n)} := \omega_n\omega_n'' \in \kappa$. 
Let $\omega'$ be a cluster point of the sequence $(\omega^{(n)})_{n\in \bb{N}}$; 
since $\kappa$ is closed, $\omega'\in \kappa$. 
Since $\pi_n(\omega) = \pi_n(\omega')$ for infinitely many $n\in \bb{N}$, 
we have that $\omega = \omega'$, implying that $\omega\in \kappa$.
\end{proof}

In the sequel, we write for brevity $\mu_{XY} = \mu_X\times \mu_Y$.  
By \cite[Theorem 3.11]{gb}, every correlation 
$\Gamma : C(\Omega_A){\otimes}C(\Omega_B)\rightarrow L^\infty(\Omega_X)\bar{\otimes}L^\infty(\Omega_Y)$
gives rise to a (unique) classical information channel 
$p_\Gamma\in\frak{C}_{\mu_{XY}}(\Omega_A\times \Omega_B,\Omega_X\times \Omega_Y;\bb{C})$,
viewed as a $\Omega_X\times\Omega_Y$-measurable family 
$p_{\Gamma} = (p_{\Gamma}(\cdot,\cdot|x,y)_{x,y})_{(x,y)\in \Omega_X\times\Omega_Y}$ of 
Borel probability measures over $\Omega_A\times\Omega_B$, 
such that 
\begin{align}\label{eq_dual_map}
    \Gamma(f)(x,y)=\int_{\Omega_A\times \Omega_B}f(a,b)dp_\Gamma(a,b|x,y), \ \ \ 
    \mu_{XY}\text{-almost everywhere,}
\end{align}
for every $f\in C(\Omega_A)\otimes C(\Omega_B)$. 
We have that
the map $\Gamma$ extends uniquely to the space of all bounded Borel functions on $\Omega_A\times\Omega_B$ and satisfies
\begin{align}\label{eq_char}
\Gamma(\chi_\delta)(x,y)=p_\Gamma(\delta|x,y),\ \delta\in\frak{B}_{A}\otimes \frak{B}_{B}, \ \  
\mu_X\times \mu_Y\mbox{-almost everywhere}.
\end{align}
Indeed, using Stinespring's Theorem and \cite[Theorem 2.6.3]{arveson_short_course}, the map $\Gamma$ can be extended uniquely to the space of bounded measurable functions on $\Omega_A\times\Omega_B$ in such a way that it has the following property: for every uniformly bounded sequence of measurable functions $(f_n)_{n\in\mathbb N}$ which converges pointwise to zero, the sequence $(\Gamma(f_n))_{n\in\mathbb N}$ converges strongly to zero. In particular, taking a uniformly bounded sequence of continuous functions $(f_n)_{n\in\mathbb N}$ converging pointwise to $\chi_{\delta}$, $\delta\in\frak{B}_{A}\otimes \frak{B}_{B}$, for $\xi,\eta\in L^2(X\times Y)$ we obtain
\begin{eqnarray*}
\langle\Gamma(\chi_\delta)\xi,\eta\rangle
& = & 
\lim_{n\to\infty}\langle\Gamma(f_n)\xi,\eta\rangle\\
& = & 
\lim_{n\to\infty}\iint f_n(a,b)dp_\Gamma(a,b|x,y)
\xi(x,y)\overline{\eta(x,y)}d\mu_{XY}(x,y)\\
& =& 
\langle p_\Gamma(\delta|\cdot,\cdot)\xi,\eta\rangle,
\end{eqnarray*} 
giving
$\Gamma(\chi_\delta)(x,y)=p_\Gamma(\delta|x,y)$  $\mu_{XY}$-almost everywhere.

If, further, $\pi$ is a Borel probability measure on 
$\Omega_X\times\Omega_Y$, let 
$\pi\otimes p_{\Gamma}$ be the \emph{compound measure} of $\pi$ and $p_{\Gamma}$, 
that is, the Borel probability measure on 
$\Omega_X\times\Omega_Y\times \Omega_A\times\Omega_B$, given by 
\begin{equation}\label{eq_piotimep}
(\pi\otimes p_{\Gamma})(M) = \int_X p_{\Gamma}(M_{x,y} | x,y)d\pi(x,y), \ \ \ 
M\in\frak{B}_{X}\otimes\frak{B}_Y\otimes\frak{B}_A\otimes\frak{B}_B,
\end{equation}
where 
$$M_{x,y} := \{(a,b)\in \Omega_A\times\Omega_B : 
(x,y,a,b)\in M\}$$ 
is the $(x,y)$-section of $M$ 
(see \cite{kakihara} and \cite{btt}).
For ${\rm t}\in\{\rm loc,qs,qc,ns\}$, let 
the \emph{$\rm t$-value} $\omega_{\rm t}(\kappa,\mu_{XY})$ of 
$\kappa$ with respect to the measure $\mu_{XY}$ be 
given by
$$\omega_{\rm t}(\kappa,\mu_{XY}) = \sup_{\Gamma\in \cl{C}_{\rm t}} (\mu_{XY}\otimes p_\Gamma)(\kappa).$$ 
We note that
if $\kappa_1$ and $\kappa_2$ are Cantor games with 
$\kappa_1\subseteq \kappa_2$ then 
\begin{equation}\label{eq_monoto}
\omega_{\rm t}(\kappa_1,\mu_{XY}) \leq \omega_{\rm t}(\kappa_2,\mu_{XY}), 
\ \ \ {\rm t}\in \{{\rm loc}, {\rm qs}, {\rm qc},{\rm ns}\}.
\end{equation}

\begin{remark}\label{r_mu_meas}
    \rm The setup of no-signalling correlations in this paper is established using almost everywhere defined operator-valued information channels. Since we are about to exploit the framework developed in \cite{btt}, we note that the compound measure $\mu_{XY}\otimes p_{\Gamma}$ is independent of the use of $\mu_{XY}$-information channels  as opposed to everywhere defined information channels. We refer the reader to 
\cite[Remark 6.1]{btt} for a detailed argument.
\end{remark}


\begin{lemma}\label{l_altexpGxy}
Let $X$, $Y$, $A$ and $B$ be inductive families of finite sets, $\kappa$ be a Cantor game over $(X,Y,A,B)$ 
and $\Gamma, \Gamma_n \in \cl C_{\rm ns}(X,Y,A,B)$, $n\in \bb{N}$, be such that $\Gamma_n\to_{n\to \infty}\Gamma$ in the BW topology. If a subset
$\kappa \subseteq \Omega_X\times\Omega_Y\times \Omega_A\times \Omega_B$ is closed and open then
$$(\mu_{XY}\otimes p_{\Gamma_n})(\kappa) \to_{n\to\infty} 
(\mu_{XY}\otimes p_{\Gamma})(\kappa).$$
\end{lemma}

\begin{proof}
If $f\in C(\Omega_X\times\Omega_Y)$ and $g\in C(\Omega_A\times\Omega_B)$ then
\begin{eqnarray*}
\langle\mu_{XY}\otimes p_\Gamma, f\otimes g\rangle &=& \int_{\Omega_X\times\Omega_Y\times\Omega_A\times\Omega_B}f(x,y)g(a,b)d(\mu_{XY}\otimes p_\Gamma)(x,y,a,b)\nonumber\\&=&\int_{\Omega_X\times\Omega_Y}\left(\int_{\Omega_A\times\Omega_B}g(a,b)dp_\Gamma(a,b|x,y)\right)f(x,y)d\mu_{XY}(x,y)\\&=&\int_{\Omega_X\times\Omega_Y}\Gamma(g)(x,y)f(x,y)d\mu_{XY}(x,y).\nonumber
\end{eqnarray*}
If $\kappa$ is closed and open then $\chi_{\kappa}$ is the 
finite sum of functions of the form 
$\chi_{\alpha}\otimes\chi_{\beta}$, where 
$\alpha\subseteq \Omega_{X\times Y}$ and 
$\beta\subseteq \Omega_{A\times B}$. The claim is now immediate. 
\end{proof}

For a closed and open subset $\kappa\subseteq \Omega_X\times\Omega_Y\times\Omega_A\times\Omega_B$, 
we write 
$\Gamma(\kappa) : \Omega_X\times\Omega_Y\to \bb{C}$
for the (measurable) function, given by 
$\Gamma(\kappa)(x,y) = \Gamma(\chi_{\kappa_{x,y}})(x,y)$,
and note that the proof of Lemma \ref{l_altexpGxy}
shows that 
$$\langle\mu_{XY}\otimes p_\Gamma, \chi_{\kappa}\rangle
= \int_{\Omega_X\times\Omega_Y}
\Gamma(\kappa)(x,y)d\mu_{XY}(x,y).
$$

\begin{lemma}\label{l_value_emb}
    Fix $n\in \bb{N}$ and let 
    $\cl{G}_n=(X_n,Y_n,A_n,B_n,\lambda_n)$ be a non-local game.
    If $\Gamma_n\in\cl{C}_{\rm t}$ and 
    $\tilde{\Gamma}_n =(\iota_{X,\infty}^{(n)}\otimes\iota_{Y,\infty}^{(n)})\circ\Gamma_n\circ(\cl{E}_{A_n}^\infty\otimes\cl{E}_{B_n}^\infty)|_{C(\Omega_A)\otimes C(\Omega_B)}$
    then 
 $$    (\mu_{X_nY_n}\otimes p_{\Gamma_n})(\supp\lambda_n)=(\mu_{XY}\otimes p_{\tilde{\Gamma}_n})(\kappa_n);$$
 consequently,
 $\omega_{\rm t}(\cl{G}_n,\mu_{X_n,Y_n})=\omega_{\rm t}(\kappa_n,\mu_{XY})$, $\rm t\in\{\rm loc, qs, qc, ns\}$.
\end{lemma}

\begin{proof}
Letting $G_n = \supp\lambda_n$ and
$\kappa_n$ be defined as in (\ref{eq_kappaGn}),  
using the definition of the compound measure and (\ref{eq_char}) we have 
that 
\begin{eqnarray*}
(\mu_{XY}\otimes p_{\tilde{\Gamma}_n})(\kappa_n)
& = & 
\langle 1_{XY},\tilde{\Gamma}_{n}({\kappa_n})\rangle\\
& = & \langle1_{XY},(\iota_{X,\infty}^{(n)}\otimes\iota_{Y,\infty}^{(n)})\circ\Gamma_n\circ(\cl{E}_{A_n}^\infty\otimes\cl{E}_{B_n}^\infty)(\kappa_n)\rangle\\
&= &
\langle(\cl{E}_{X_n}\otimes\cl{E}_{Y_n})(1_{XY}),\Gamma_n\circ(\cl{E}_{A_n}^\infty\otimes\cl{E}_{B_n}^\infty)(\kappa_n)\rangle\\
& = &
\langle1_{X_nY_n},\Gamma_{n}(G_n)\rangle
=(\mu_{X_nY_n}\otimes p_{\Gamma_n})(G_n).
 \end{eqnarray*}
 Now, since $\Gamma_n\in\cl{C}_{\rm t}(X_n,Y_n,A_n,B_n)\implies \tilde{\Gamma}_n\in\cl{C}_{\rm t}(X,Y,A,B)$,
 by Lemma \ref{l_ampliqc} we obtain $$\omega_{\rm t}(\cl{G}_n,\mu_{X_n,Y_n})\leq\omega_{\rm t}(\kappa_n,\mu_{XY}).$$
 Next let $\Gamma\in\cl{C}_{\rm t}(X,Y,A,B)$,  set 
 $$\Gamma_n = \left(\cl{E}^\infty_{X_n}\otimes \cl{E}^\infty_{Y_n}\right)
 \circ\Gamma\circ \left(\iota_A^{(n)}\otimes \iota_B^{(n)}\right), \ \ n\in\bb{N},$$
 and note that $(\mu_{X_nY_n}\otimes p_{\Gamma_n})(G_n)=(\mu_{XY}\otimes p_\Gamma)(\kappa_n)$. Indeed, 
 \begin{align*}
     (\mu_{X_nY_n}\otimes p_{\Gamma_n})(G_n) &= \langle 1_{X_nY_n}, \Gamma_n(G_n)\rangle\\
     & =\langle 1_{X_nY_n}, (\cl{E}^\infty_{X_n}\otimes \cl{E}^\infty_{Y_n})\circ\Gamma\circ(\iota_A^{(n)}\otimes \iota_B^{(n)})(G_n)\\
     & =\langle \cl{E}_{X_n}\otimes \cl{E}_{Y_n})(1_{XY}), (\cl{E}^\infty_{X_n}\otimes \cl{E}^\infty_{Y_n})\circ\Gamma(\kappa_n)\rangle\\
     & = \langle 1_{XY}, \Gamma(\kappa_n)\rangle=(\mu_{XY}\otimes p_\Gamma)(\kappa_n). 
 \end{align*}
 Now since $\Gamma\in\cl{C}_{\rm t}(X,Y,A,B)\implies \Gamma_n\in\cl{C}_{\rm t}(X_n,Y_n,A_n,B_n)$
 (see Theorems \ref{th_ns} and \ref{th_n_qc}, and Remark \ref{r_locqson}), 
 one obtains $\omega_{\rm t}(\cl{G}_n,\mu_{X_nY_n})=\omega_{\rm t}(\kappa_n,\mu_{XY}).$ 
 \end{proof}

\begin{theorem}\label{th_valnegaat}
    Let $\cl G=(\cl{G}_n)_{n\in\bb{N}}$ be a nested family, where $\cl G_n$ is a 
    non-local game over $(X_n,Y_n,A_n,B_n)$, $n\in\bb{N}$.
   If ${\rm t}\in \{\rm loc,qs,qc,ns\}$ then
    $$\lim_{n\to\infty} \omega_{\rm t}(\cl{G}_n,\mu_{X_nY_n})
    = \inf_{n\in \bb{N}} \omega_{\rm t}(\cl{G}_n,\mu_{X_nY_n})
    = \omega_{\rm t}(\cl{G},\mu_{XY}).$$
\end{theorem}
\begin{proof}
    Let $\kappa=\kappa_{\cl{G}}$ and 
    $\kappa_n = \pi_n^{-1}(\supp\lambda_n)$.  We have $\kappa=\cap_{n\in\mathbb N}\kappa_n$. Fix $\Gamma\in\cl{C}_{\rm t}$ and note that for any $m\geq n$, by monotonicity of the measures as $(\kappa_n)_{n\in\mathbb N}$ is a decreasing sequence of sets, one has 
    \begin{align*}\label{eq_41ga}
        (\mu_{XY}\otimes p_\Gamma)(\kappa_m)\leq (\mu_{XY}\otimes p_{\Gamma})(\kappa_n).
    \end{align*}
    By taking supremum over $\Gamma\in\cl{C}_{\rm t}$ we
    obtain that 
$\omega_{\rm t}(\kappa_m,\mu_{XY}) \leq
\omega_{\rm t}(\kappa_n,\mu_{XY})$.
By Lemma \ref{l_value_emb}, 
$$    \omega_{\rm t}(\cl{G}_m,\mu_{X_mY_m})\leq\omega_{\rm t}(\cl{G}_n,\mu_{X_nY_n}).$$
Again by Lemma \ref{l_value_emb},
 $$   \omega_{\rm t}(\cl{G},\mu_{XY})\leq\inf_{n\in\bb{N}}\omega_{\rm t}(\cl{G}_n,\mu_{X_nY_n}).$$

For the converse, fix $\epsilon>0$, and let $\Gamma_m\in\cl{C}_{\rm t}$ such that  

\begin{equation} \label{eq_appromeg}
    \omega_{\rm t}(\cl{G}_m,\mu_{X_mY_m})-(\mu_{XY}\otimes p_{\Gamma_m})(\kappa_m)<\frac{\epsilon}{m},
\end{equation}
and assume, without loss of generality, that 
$(\tilde{\Gamma}_m)_{m \in \bb N}$ converges to $\Gamma$ in the BW topology. By monotonicity of the measures again, 
$$(\mu_{XY}\otimes p_{\tilde{\Gamma}_m})(\kappa_m)\leq(\mu_{XY}\otimes p_{\tilde{\Gamma}_m})(\kappa_n) \ \ \mbox{ for all } m\geq n,$$
and thus, using the convergence of 
$\tilde{\Gamma}_m$ to $\Gamma$, Lemma \ref{l_altexpGxy}, Lemma \ref{l_value_emb} and (\ref{eq_appromeg}), we get
\begin{align}\label{rel_nes}
    \inf_{m\in \bb{N}}\omega_{\rm t} (\kappa_m,\mu_{XY})\leq (\mu_{XY}\otimes p_{\Gamma})(\kappa_n).
\end{align}
By the monotonicity of measure and the fact that 
$\kappa=\cap_{n\in\bb{N}}\kappa_n$, we obtain 
\begin{align}\label{rel_cts_ab}
    (\mu_{XY}\otimes p_\Gamma)(\kappa)=\lim_{n\in\bb{N}}(\mu_{XY}\otimes p_\Gamma)(\kappa_n).
\end{align}
Finally, combining relations (\ref{rel_nes}) and (\ref{rel_cts_ab}) one obtains the desired result.
\end{proof}

\bigskip

Our next aim is to provide tensor norm descriptions of the quantum spatial and the 
quantum commuting value of a Cantor game. 
Suppose first that $G = (\lambda,\mu)$ is a non-local game over a 
quadruple $(S,T,U,V)$ of finite sets. The element 
$$t_G := \sum_{(x,y)\in S\times T} 
\sum_{(a,b)\in U\times V}
\mu(x,y) 
\lambda(x,y,a,b) e_{x,a}\otimes e_{y,b}$$
of the (algebraic) tensor product $\cl S_{S,U}\otimes \cl S_{T,V}$ is 
usually referred to as 
the \emph{full game tensor} of $G$. 
It can easily be seen from Theorem \ref{th_orcha} that
\begin{equation}\label{eq_tensexp}
\omega_{\rm qs}(\lambda,\mu) = \|t_G\|_{\min}, \ \ 
\omega_{\rm qc}(\lambda,\mu) = \|t_G\|_{\rm c} 
\ \mbox{ and } \ 
\omega_{\rm ns}(\lambda,\mu) = \|t_G\|_{\max}.
\end{equation}

Assume that $\cl G = (\cl{G}_n)_{n\in\bb{N}}$ is a nested family, where $\cl G_n$ is a 
non-local game over $(X_n,Y_n,A_n,B_n)$ with rule function $\lambda_n$, $n\in\bb{N}$. 
Letting $\kappa_n = \kappa_{\cl G_n}$, write 
$$t_{\cl G}^{(n)} = \frac{1}{|X_n||Y_n|}
\sum_{(x,y)\in X_n\times Y_n} 
\sum_{(a,b)\in A_n\times B_n}
\lambda_n(x,y,a,b) \chi_{\tilde{x}\times\tilde{a}} \otimes \chi_{\tilde{y}\times\tilde{b}},$$
considered as an element of $\cl S_{X,A}\otimes \cl S_{Y,B}$. 

\begin{lemma}\label{l_monoten}
Let $\cl G = (\cl{G}_n)_{n\in\bb{N}}$ be a nested family of games. 
\begin{itemize}
\item[(i)] If ${\tau}\in \{{\min}, {\rm c}, \max\}$ 
then
$t_{\cl G}^{(n+1)}\leq t_{\cl G}^{(n)}$ in 
$\cl S_{X,A}\otimes_{\tau} \cl S_{Y,B}$, $n\in \bb{N}$.

\item[(ii)] The limit $\lim_{n\to \infty} t_{\cl G}^{(n)}$ exists in the 
weak* topology of $\cl S_{X,A}^{**}\bar{\otimes} \cl S_{Y,B}^{**}$. 
\end{itemize}
\end{lemma}

\begin{proof}
Write $G_n = \supp(\lambda_n)$ for brevity.

(i) 
Fix $n\in \bb{N}$. 
Since the family $\cl G$ is nested, 
$$\supp (\lambda_{n+1}) \subseteq \supp(\lambda_n) \times 
\left([d_n^X]\times
[d_n^Y]\times [d_n^A]\times [d_n^B]\right).
$$
Letting 
$$
\hspace{-1.5cm}
G_{n+1}' = \{(xx',yy',aa',bb') : (x,y,a,b)\in G_n, $$
$$\hspace{3.5cm} (x',y',a',b')\in [d_n^X]\times
[d_n^Y]\times [d_n^A]\times [d_n^B]\},$$
we have 
\begin{eqnarray*} 
&& 
\frac{1}{|X_{n}||Y_{n}|}\sum_{(x,y,a,b)\in G_n} 
\chi_{\tilde{x}\times\tilde{a}} \otimes \chi_{\tilde{y}\times\tilde{b}} 
\hspace{0.1cm} \\
& &
\hspace{3cm}
- \frac{1}{|X_{n+1}||Y_{n+1}|} \hspace{-0.1cm}
\sum_{(x,y,a,b)\in G_{n+1}} 
\chi_{\tilde{x}\times\tilde{a}} \otimes \chi_{\tilde{y}\times\tilde{b}}\\
& = & 
\frac{1}{|X_{n}||Y_{n}|} \cdot \frac{1}{d_{n}^{X} d_{n}^{Y}} \sum_{(x,y,a,b)\in G_{n+1}'} 
\chi_{\tilde{x}\times\tilde{a}} \otimes \chi_{\tilde{y}\times\tilde{b}} 
\hspace{0.1cm} \\
& & \hspace{3cm} - \frac{1}{|X_{n+1}||Y_{n+1}|} \hspace{-0.1cm}
\sum_{(x,y,a,b)\in G_{n+1}} 
\chi_{\tilde{x}\times\tilde{a}} \otimes \chi_{\tilde{y}\times\tilde{b}}\\
& = & 
\frac{1}{|X_{n+1}||Y_{n+1}|} \sum_{(x,y,a,b)\in G_{n+1}'\setminus G_{n+1}} 
\chi_{\tilde{x}\times\tilde{a}} \otimes \chi_{\tilde{y}\times\tilde{b}}.
\end{eqnarray*}
Since $\chi_{\tilde{x}\times\tilde{a}}\in \cl S_{X,A}^+$ and 
$\chi_{\tilde{y}\times\tilde{b}}\in \cl S_{Y,B}^+$, we have that 
$$\sum_{(x,y,a,b)\in G_{n+1}'\setminus G_{n+1}} 
\chi_{\tilde{x}\times\tilde{a}} \otimes \chi_{\tilde{y}\times\tilde{b}}
\in (\cl S_{X,A}\otimes_{\max} \cl S_{Y,B})^+,$$
and hence 
$$t_{\cl G}^{(n)} - t_{\cl G}^{(n+1)} 
\geq \frac{1}{|X_{n+1}||Y_{n+1}|} \sum_{(x,y,a,b)\in G_{n+1}'\setminus G_{n+1}} 
\chi_{\tilde{x}\times\tilde{a}} \otimes \chi_{\tilde{y}\times\tilde{b}}\geq 0$$
in $\cl S_{X,A}\otimes_{\max} \cl S_{Y,B}$. 
The rest of the conclusions follow from the fact that 
$$(\cl S_{X,A}\otimes_{\max} \cl S_{Y,B})^+\subseteq 
(\cl S_{X,A}\otimes_{\rm c} \cl S_{Y,B})^+\subseteq
(\cl S_{X,A}\otimes_{\min} \cl S_{Y,B})^+.$$

(ii) 
Consider $t_{\cl G}^{(n)}$ as an element of $\cl S_{X,A}^{**}\bar{\otimes} \cl S_{Y,B}^{**}$, $n\in \bb{N}$. By (i), the sequence  
$(t_{\cl G}^{(n)})_{n\in \bb{N}}$ is monotone decreasing and bounded from below 
(by the zero element). 
The conclusion is now immediate. 
\end{proof}

In view of Lemma \ref{l_monoten}, set 
$t_{\cl G} := {\rm w}^*\mbox{-}\lim_{n\to \infty} t_{\cl G}^{(n)}$,
considered as an element 
of $\cl S_{X,A}^{**}\bar{\otimes} \cl S_{Y,B}^{**}$. 

\begin{theorem}\label{th_tenformu}
Let $(\cl{G}_n)_{n\in\bb{N}}$ be a nested family, where $\cl G_n$ is a 
non-local game over $(X_n,Y_n,A_n,B_n)$, $n\in\bb{N}$. 
Then 
\begin{itemize}
\item[(i)] $\omega_{\rm ns}(\cl{G},\mu_{XY}) = \lim_{n\to \infty} \|t_{\cl G}^{(n)}\|_{\max}$;

\item[(ii)] 
$\omega_{\rm qc}(\cl{G},\mu_{XY}) = \lim_{n\to \infty} \|t_{\cl G}^{(n)}\|_{\rm c}$, and 

\item[(iii)] 
$\omega_{\rm qs}(\cl{G},\mu_{XY}) = 
\lim_{n\to \infty} \|t_{\cl G}^{(n)}\|_{\min}\geq 
\|t_{\cl G}\|_{\min}$.
\end{itemize}
\end{theorem}

\begin{proof}
The statements follow from (\ref{eq_tensexp}) and Theorem \ref{th_valnegaat}.
\end{proof}

\begin{example}[The IID case]\label{ex_IID}
\rm Let $X_0$, $Y_0$, $A_0$ and $B_0$ be finite sets, and 
$\lambda : X_0\times Y_0\times A_0\times B_0 \to \{0,1\}$ be a rule function of a game over the 
quadruple $(X_0,Y_0,A_0,B_0)$. Write $E = {\rm supp}(\lambda)$ and, letting 
$X_n = X_0^n$, $Y_n = Y_0^n$, $A_n = A_0^n$ and $B_n = B_0^n$, $n\in \bb{N}$, 
consider the game over $(X_n,Y_n,A_n,B_n)$, with rule function whose support is the set 
\begin{eqnarray*}
&&\hspace{-0.5cm} E_n = \{((x_i)_{i=1}^n,(y_i)_{i=1}^n,(a_i)_{i=1}^n,(b_i)_{i=1}^n) : 
(x_i,y_i,a_i,b_i)\in E \mbox{ for every } i\in [n]\}.
\end{eqnarray*}
Let $X = (X_n)_{n=1}^{\infty}$, $Y = (Y_n)_{n=1}^{\infty}$, $A = (A_n)_{n=1}^{\infty}$ and 
$B = (B_n)_{n=1}^{\infty}$ be the corresponding inductive sequences, and
embed $E_n$ in the first $n$ coordinates, 
yielding a set $\kappa_n\subseteq\Omega_X\times\Omega_Y\times \Omega_A\times\Omega_B$, 
$n\in \bb{N}$. 
The sequence $(\kappa_n)_{n=1}^{\infty}$ is nested and, if $\kappa = \cap_{n\in \bb{N}}\kappa_n$, then 
\begin{eqnarray*}
&&\hspace{-0.8cm} \kappa = \{((x_i)_{i=1}^\infty,(y_i)_{i=1}^\infty,(a_i)_{i=1}^\infty,(b_i)_{i=1}^\infty) : 
(x_i,y_i,a_i,b_i)\in E \mbox{ or } \mbox{ for every } i\in \bb{N}\}.
\end{eqnarray*}
The Cantor game $\kappa$ encodes the infinite parallel
repetition of the game $E$ (that is, the 
product of infinitely many copies of $E$).
\end{example}

\begin{example}[Markov type]\label{ex_Markov}
\rm 
We describe a class of examples of non-IID Cantor games, and therefore of games, to which Theorems 
\ref{th_valnegaat} and \ref{th_tenformu} apply. 
As in Example \ref{ex_IID}, let 
$X_0$, $Y_0$, $A_0$ and $B_0$ be finite sets, and 
$\lambda : X_0\times Y_0\times A_0\times B_0 \to \{0,1\}$ be a rule function of a game over the 
quadruple $(X_0,Y_0,A_0,B_0)$. Write $E = {\rm supp}(\lambda)$ and
consider the game over $(X_n,Y_n,A_n,B_n)$, with rule function whose support is the set 
\begin{eqnarray*}
&&\hspace{-0.5cm} E_n = \{((x_i)_{i=1}^n,(y_i)_{i=1}^n,(a_i)_{i=1}^n,(b_i)_{i=1}^n) : 
(x_i,y_i,a_i,b_i)\in E \mbox{ or }\\
&& \hspace{2.5cm}
(x_{i+1},y_{i+1},a_{i+1},b_{i+1})\in E,
\mbox{ for every } i\in [n-1]\}.
\end{eqnarray*}
Let $X = (X_n)_{n=1}^{\infty}$, $Y = (Y_n)_{n=1}^{\infty}$, $A = (A_n)_{n=1}^{\infty}$ and 
$B = (B_n)_{n=1}^{\infty}$ be the corresponding inductive sequences, and
embed $E_n$ in the first $n$ coordinates, 
yielding a set $\kappa_n\subseteq\Omega_X\times\Omega_Y\times \Omega_A\times\Omega_B$, 
$n\in \bb{N}$. 
A straightforward inspection shows that the sequence $(\kappa_n)_{n=1}^{\infty}$ is nested.
The corresponding Cantor game has a set $\kappa$ of admissible quadruples, given by 
\begin{eqnarray*}
&&\hspace{-0.8cm} \kappa = \{((x_i)_{i=1}^\infty,(y_i)_{i=1}^\infty,(a_i)_{i=1}^\infty,(b_i)_{i=1}^\infty) : 
(x_i,y_i,a_i,b_i)\in E \mbox{ or }\\
&& \hspace{2.7cm}
(x_{i+1},y_{i+1},a_{i+1},b_{i+1})\in E,
\mbox{ for every } i\in \bb{N}\};
\end{eqnarray*}
heuristically, this means that the rules of the Cantor 
game require that, in any two individual rounds of 
the underlying finite game, the players win at least 
once.
\end{example}




\begin{thebibliography}{99}

\bibitem{AndoHaagerup}
{\sc H.~Ando and U.~Haagerup},
{\it Ultraproducts of von Neumann algebras},
{\rm J.\ Funct.\ Anal. 266 (2014), no.~12, 6842--6913}.

\bibitem{arveson-acta}
\textsc{W.\,B.\,Arveson}, 
{\it Subalgebras of C*-algebras}, 
{\rm Acta Math. 123 (1969), 141-224}. 

\bibitem{arveson_short_course}
\textsc{W.\,B.\,Arveson}, {\it A short course on spectral theory},
{\rm Springer-Verlag, 2002}.


\bibitem{bctt}
{\sc G.\,Baziotis, A.\,Chatzinikolaou, I.\,G.\,Todorov and L.\,Turowska},
{\it Cantor correlations II. Synchronicity},
{\rm in preparation}.


\bibitem{gb}
{\sc G.\,Baziotis},
{\it Distances for operator-valued information channels},
{\rm preprint (2024), arXiv:2409.17052, to appear in Arkiv Math.}


\bibitem{btt}
{\sc G.\,Baziotis, I.\,G.\,Todorov and L.\,Turowska},
{\it Measurable no-signalling correlations},
{\rm preprint (2024), arXiv:2409.17206}.




\bibitem{csuu}
{\sc R.\,Cleve, W.\,Slofstra, F.\,Unger and S.\,Upadhyay}, 
{\it Perfect parallel repetition theorem for quantum XOR proof systems},
{\rm Computational Complexity 17 (2008), no. 2, 
282-299}.

\bibitem{connes}
{\sc A.\,Connes},
{\it Classification of injective factors. Cases $II_1$, $II_{\infty}$, 
$III_{\lambda}$, $\lambda\neq 1$}, 
{\rm Ann. of Math. (2) 104 (1976), no. 1, 73-115}.



\bibitem{cjpp}
{\sc T.\,Cooney, M.\,Junge, C.\,Palazuelos and D.\,Perez-Garci\'a}, 
{\it Rank one quantum games},
{\rm Comput. Comlex. 24 (2015), 133-196}.

\bibitem{cltt}
{\sc J.\,Crann, R.\,H.\,Levene, I.\,G.\,Todorov and L.\,Turowska},
{\it Values of cooperative quantum games},
{\rm to appear in Trans. Amer. Math. Soc., arXiv:2310.17735 (60 pages)}.


\bibitem{ctt-self}
{\sc J.\,Crann, I.\,G.\,Todorov and L.\,Turowska},
{\it An operator system approach to self-testing},
{\rm preprint (2025), arXiv:2506.17980 (65 pages)}.



\bibitem{dun_sch}
{\sc N.\,Dunford and J.\,T.\,Schwartz},
{\it Linear Operators. Part I},
{\rm Interscience, 1958}.


\bibitem{er}
{\sc E.\,G.\,Effros and Zh.-J.\,Ruan},
{\it Operator spaces},
{\rm Oxford University Press, 2000}.


\bibitem{ffp}
{\sc D.\,Farenick, R.\,Floricel and S.\,Plosker}, 
{\it Approximately clean quantum probability measures}, 
{\rm J. Math. Phys. 54 (2013), no. 5, 052201, 15 pp.}

\bibitem{fri-CEP}
{\sc T.\,Fritz},
{\it Tsirelson's problem and Kirchberg's conjecture},
{\rm Rev. Math. Phys. 24 (2012), no. 5, 1250012, 67 pp.}


\bibitem{fri}
{\sc T.\,Fritz},
{\it Operator system structures on the unital direct sum of $C^*$-algebras},
{\rm Rocky Mountain J. Math. 44 (2014), no.3, 913-936}.


\bibitem{jnvwy}
{\sc Z.\,Ji, A.\,Natarajan, T.\,Vidick, J.\,Wright and H.\,Yuen},
{\it MIP*=RE},
{\rm preprint (2020), arXiv:2001.04383}.

\bibitem{jnppsw}
{\sc M.\,Junge, M.\,Navascues, C.\,Palazuelos, D.\,Perez-Garcia, V.\,Scholz and R.\,F.\,Werner}, 
{\it Connes' embedding problem and Tsirelson's problem}, 
{\rm J. Math. Phys. 52 (2011), no. 1, 012102, 12 pp.}



\bibitem{kakihara}
{\sc Y.\,Kakihara}, 
{\it Abstract methods in information theory},
{\rm World Scientific Publishing, 2016}.


\bibitem{kavruk-nucl}
{\sc A.\,S.\,Kavruk},
{\it Nuclearity related properties in operator systems},
{\rm J. Operator Theory 71 (2014), no. 1, 95-156}.

\bibitem{kptt}
{\sc A.\,S.\,Kavruk, V.\,I.\,Paulsen, I.\,G.\,Todorov and M.\,Tomforde}, 
{\it Tensor products of operator systems},
{\rm J. Funct. Anal. 261 (2011), 267-299}.


\bibitem{lmprsstw}
{\sc M.\,Lupini, L.\,Man\v{c}inska, V.\,Paulsen, D.\,E.\,Roberson, G.\,Scarpa, S.\,Severini, I.\,G.\,Todorov and A.\,Winter},
{\it Perfect strategies for non-signalling games}, 
{\rm Math Phys Anal Geom 23 (2020), 7}.

\bibitem{lmr}
{\sc M.\,Lupini, L.\,Man\v{c}inska and D.\,E.\,Roberson},
{\it Nonlocal games and quantum permutation groups},
{\rm J. Funct. Anal. 279 (2020), no. 5, 108592, 44 pp.}

\bibitem{mptw}
{\sc L.\,Man\v{c}inska, V.\,I.\,Paulsen, I.\,G.\,Todorov and A.\,Winter},
{\it Products of synchronous games},
{\rm Studia Math. 272 (2023), no. 3, 299-317}.

\bibitem{mt}
{\sc L.\,Mawhinney and I.\,G.\,Todorov},
{\it Inductive limits in the operator system and related categories},
{\rm Dissertationes Math. 536 (2018), 57 pp.}

\bibitem{mror}
{\sc M.\,Musat and M.\,R\o rdam},
{\it Non-closure of quantum correlation matrices and factorizable channels that require infinite dimensional ancilla},
{\rm Comm. Math. Phys. 375 (2020), no. 3, 1761-1776}.


\bibitem{oz} 
{\sc N.\,Ozawa}, 
{\it About the Connes' embedding problem -- algebraic approaches},
{\rm Japan. J.  Math. 8 (2013), no. 1, 147-183}.

\bibitem{pszz}
{\sc C.\,Paddock, W.\,Slofstra, Y.\,Zhao and Y.\,Zhou},
{\it An operator-algebraic formulation of self-testing},
{\rm Ann. Henri Poincare 25 (2024), no. 10, 4283-4319}.


\bibitem{pal-vid}
{\sc C.\,Palazuelos and T.\,Vidick},
{\it Survey on nonlocal games and operator space theory},
{\rm J. Math. Phys. 57 (2016), no. 1, 015220, 41 pp.}

\bibitem{Pa} 
{\sc V.\,I.\,Paulsen}, 
{\it Completely bounded maps and operator algebras}, 
{\rm Cambridge University Press, 2002}.


\bibitem{psstw}
{\sc V.\,I.\,Paulsen, S.\,Severini, D.\,Stahlke, I.\,G.\,Todorov and A.\,Winter},
{\it Estimating quantum chromatic numbers}, 
{\rm J. Funct. Anal. 270 (2016), 2188-2222}.


\bibitem{pt}
{\sc V.\,I.\,Paulsen and I.\,G.\,Todorov},
{\it Quantum chromatic numbers via operator systems},
{\rm Q. J. Math. 66 (2015), no. 2, 677-692}.

\bibitem{Pi2} 
{\sc G. Pisier}, 
{\it Tensor products of C*-algebras and operator spaces: The Connes-Kirchberg problem},
{\rm Cambridge University Press, 2020}.


\bibitem{r}
{\sc R.\,Raz}, 
{\it A parallel repetition theorem},
{\rm SIAM J. Comput. 27 (1998), 763-803}.

\bibitem{slofstra}
{\sc W.\,Slofstra},
{\it The set of quantum correlations is not closed},
{\rm Forum of Mathematics, Pi. 2019;7:e1.}

 


\bibitem{tsirelson}
{\sc B.\,S.\,Tsirelson}, 
{\it Some results and problems on quantum Bell-type inequalities}, 
{\rm Hadronic J. Supplement 8 (1993), no. 4, 329-345}.



\end{thebibliography}
\end{document}